\documentclass[EJP, preprint]{ejpecp} 

\usepackage[utf8x]{inputenc}
\usepackage[T1]{fontenc}
\usepackage{hyperref}
\usepackage{verbatim}

\usepackage{enumitem}
\usepackage{tikz}
\usepackage{mathrsfs}
\usetikzlibrary{trees}
\usepackage{dsfont}
\usepackage{upgreek}
\usepackage[titletoc,title]{appendix}

\newcommand{\R}{\mathbb{R}}
\newcommand{\N}{\mathbb{N}}

\newcommand{\pr}{\mathbb{P}}
\newcommand{\ex}{\mathbb{E}}

\newcommand{\h}{\mathcal{H}}

\newcommand{\blangle}{\big\langle}
\newcommand{\brangle}{\big\rangle}

\newtheorem{lem}{Lemma}
\newtheorem{cor}{Corollary}
\newtheorem{prop}{Proposition}
\newtheorem{thm}{Theorem}

\newtheorem{dfn}{Definition}
\newtheorem{innercustomthm}{Condition}
\newenvironment{customthm}[1]
{\renewcommand\theinnercustomthm{#1}\innercustomthm}
{\endinnercustomthm}
\DeclareMathOperator*{\esssup}{ess\,sup}
\theoremstyle{definition}
\newtheorem{rem}{Remark}
\makeatletter
\newcommand{\leqnomode}{\tagsleft@true\let\veqno\@@leqno}
\newcommand{\reqnomode}{\tagsleft@false\let\veqno\@@eqno}
\makeatother

\numberwithin{lem}{section}
\numberwithin{cor}{section}
\numberwithin{prop}{section}
\numberwithin{thm}{section}
\numberwithin{dfn}{section}

\SHORTTITLE{ Large deviations of slow-fast systems driven by fractional Brownian motion }

\TITLE{Large deviations of slow-fast systems driven by fractional Brownian motion } 

\AUTHORS{%
  Siragan Gailus\footnote{Arbol Inc.
    \EMAIL{siragan@arbolmarket.com}}
  \and 
  Ioannis Gasteratos\footnote{Technische Universit\"at Berlin, Institute of Mathematics. \BEMAIL{i.gasteratos@tu-berlin.de} }}


\AMSSUBJ{60F10; 60G22; 60H10; 93E20}

\ABSTRACT{We consider a multiscale system of stochastic differential equations in which the slow component is perturbed by a small fractional Brownian motion with Hurst index $H>1/2$ and the fast component is driven by an independent Brownian motion. Working in the framework of Young integration, we use tools from fractional calculus and weak convergence arguments to establish a Large Deviation Principle in the homogenized limit, as the noise intensity and time-scale separation parameters vanish at an appropriate rate. Our approach is based in the study of the limiting behavior of an associated controlled system. We show that, in certain cases,  the non-local rate function admits an explicit non-variational form. The latter allows us to draw comparisons to the case $H=1/2$ which corresponds to the classical Freidlin-Wentzell theory. Moreover, we study the asymptotics of the rate function as $H\rightarrow{1/2}^+$ and show that it is discontinuous at $H=1/2.$}

\begin{document}

\section{Introduction}\label{intro}

In this paper we study the asymptotic tail behavior of the following multiscale system of Stochastic Differential Equations (SDEs) 
\begin{equation}\label{model}
\left \{\begin{aligned}
& dX^{\epsilon,\eta}_t= \bigg[\tfrac{\sqrt{\epsilon}}{\sqrt{\eta}}b(Y^{\epsilon,\eta}_t)+ c(X_t^{\epsilon,\eta}, Y^{\epsilon,\eta}_t)\bigg]dt+\sqrt{\epsilon}\bigg[\sigma_1(X^{\epsilon,\eta}_t,Y^{\epsilon,\eta}_t)dB^H_t+\sigma_2(X^{\epsilon,\eta}_t,Y^{\epsilon,\eta}_t)dW_t\bigg]\\&
dY^{\epsilon,\eta}_t= \tfrac{1}{\eta}\bigg[f( Y^{\epsilon,\eta}_t)+\tfrac{\sqrt\eta}{\sqrt\epsilon}g(X^{\epsilon,\eta}_t, Y^{\epsilon,\eta}_t)\bigg]dt+\tfrac{1}{\sqrt{\eta}}\tau(Y^{\epsilon,\eta}_t)dW_t\\&
X^{\epsilon,\eta}_0=x_0\in\R^m\;,\;\; Y^{\epsilon,\eta}_0=y_0\in\R^{d-m},
\end{aligned} \right.
\end{equation}
where $t>0,$  $\epsilon>0$ is a small parameter and $\eta=\eta(\epsilon)$ is a time-scale separation parameter that vanishes as $\epsilon$ goes to zero. Here, $B^H$ is a $k$-dimensional fractional Brownian motion (fBm) with Hurst index $H\in (1/2, 1),$ $W$ is a standard $\ell$-dimensional Brownian motion independent of $B^H$ and the pair $(B^H, W)$ is defined on a complete filtered probability space $(\Omega,\mathscr{F},\{\mathscr{F}_t\}_{t\geq 0},\pr)$. The process $X^{\epsilon,\eta}$ is driven by small noise of intensity $\sqrt\epsilon$ and shall be called the slow motion, while the process $Y^{\epsilon,\eta}$ evolves on a time-scale of order $1/\eta$ and can be thought of as the fast motion. Throughout this work we assume that $\sqrt{\eta}/\sqrt{\epsilon}$ vanishes as $\epsilon$ goes to zero. Thus, the coefficients $b, g$ account for the effect of fast intermediate scales on the evolution of the slow and fast components respectively.

The typical (or effective) dynamics of the slow motion are described by the homogenized limit of $X^{\epsilon,\eta}$ as $\epsilon$ (and hence $\eta$) are taken to zero. In the case $H=1/2,$ $B^{1/2}$ is a standard Brownian motion and the corresponding homogenization theory of SDEs has been extensively studied; see e.g. \cite{freidlin1978averaging,khas1968averaging,pardoux2003poisson,veretennikov1991averaging} as well as \cite{bensoussan2011asymptotic} (Chapter 3), \cite{freidlin2012random} (Chapter 7), \cite{pavliotis2008multiscale} (Chapters 11,18). The evolution of the limiting process depends upon the asymptotic behavior of the singular term $\tfrac{\sqrt{\epsilon}}{\sqrt{\eta}}b,$ as well as the long-time behavior of the fast motion. Assuming that the latter is uniquely ergodic, we expect that the contribution of the fast variable to the dynamics of $X^{\epsilon,\eta}$ will be averaged with respect to the invariant measure of $Y^{\epsilon,\eta}$. Moreover, it turns out that the limiting contribution of the singular perturbation can be captured in terms of the solution of an associated Poisson equation (see e.g. Theorem 3 in \cite{veretennikov1991averaging}).

The development of averaging and homogenization theory in the case $H>1/2$ has only recently attracted attention in the literature. From a modeling perspective, such systems provide a more accurate description of random phenomena that feature memory and long-range dependence. From a mathematical perspective, the study of these SDEs is challenging due to the fact that $B^H$ is neither Markovian nor a martingale. Thus, the tools of It\^o calculus are no longer available and the analysis, i.e. estimates, well-posedness and properties of solutions, depends on the interpretation of the stochastic integral $\sigma_1dB^H.$ The reader is referred to  \cite{bourguin2020typical,hairer2020averaging,hairer2021generating,pei2020averaging} for recent results on averaging, homogenization and analysis of the fluctuations around the typical dynamics, under different sets of assumptions on the coefficients. In particular, if $b=\sigma_2=g=0,\sigma_1(x,y)=\sigma_1(y)$ and $\sigma_1dB^H$ is interpreted as a divergence integral (see e.g. \cite{nualart2006malliavin}, Section 5.2),  Theorem 1 of \cite{bourguin2020typical} asserts that for each $T>0,$ there exists a small enough $p>1$ such that $X^{\epsilon,\eta}$ converges in $L^p(\Omega; C([0,T];\R^m))$ to the unique solution of the deterministic differential equation
\begin{equation*}\label{ave1}
\left \{\begin{aligned}
& d\bar{X}_t=  \bar{c}(\bar{X}_t)dt\\&
\bar{X}_0=x_0\in\R^m,
\end{aligned} \right.
\end{equation*}
where $\bar{c}(x)=\int_{\R^{d-m}}c(x,y)d\mu(y)$ and $\mu$ is the unique invariant measure of the fast motion. The same conclusion can be recovered from Theorem 1.1 of \cite{pei2020averaging} in the case where $b=\sigma_2=g=0,\sigma_1(x,y)=\sigma_1(x)$ and $\sigma_1dB^H$ is interpreted as a pathwise Young integral in the sense of Z\"ahle (see e.g. \cite{zahle1999link} and \cite{nualart2006malliavin} Section 5.3.1).

The case where the coefficient $\sigma_1$ depends on both the slow and the fast variables presents additional challenges. These are related to the proof of estimates for the integral $\sigma_1dB^H$ that are uniform over small values of $\epsilon$ and $\eta$. If the latter is a divergence integral,  then standard estimates require bounds on the Malliavin derivative of $X^{\epsilon,\eta},$ which is then expressed in terms of its second Malliavin derivative, and lead to a closure problem (see e.g. Section 1 of \cite{bourguin2020typical} for a discussion of this issue). Turning to the setting of \cite{pei2020averaging} (where $\epsilon=1$ and $b=0$),  pathwise estimates for the Young integral require control over the H\"older seminorm of $Y^{\epsilon,\eta}$ which is unbounded with respect to $\epsilon.$ Finally, the authors of \cite{hairer2020averaging} were able to overcome this issue by constructing an extension of the pathwise Young integral that is well-defined for "rougher" integrands (we refer the reader to  \cite{hairer2020averaging}, Section 3.3 for more details).

As mentioned above, the homogenized limit describes the typical behavior of the slow motion and thus can be viewed as a Law of Large Numbers (LLN) for $X^{\epsilon,\eta}$. A subsequent step in the asymptotic analysis of such models lies in the study of large deviations from the LLN limit.  In particular, the goal is to characterize the exponential decay rate of rare-event probabilities, as $\epsilon\to 0,$ via an appropriate rate function. In the case $H=1/2,$ the Large Deviation theory of multiscale SDEs has been well studied under different assumptions on the coefficients and interaction regimes between $\epsilon$ and $\eta;$  see e.g. \cite{baldi1991large,dupuis2012large,freidlin1999comparison,spiliopoulos2013large} as well as \cite{freidlin2012random}, Chapter 7.4. 

In the present work we aim to prove a Large Deviation Principle (LDP) for the slow motion \eqref{model}, in the case where $\sigma_1dB^H$ is a (pathwise) Young integral (see Definition \ref{Young} below) and $\sigma_2dW$ is a standard It\^o integral. The latter amounts to finding a rate function $S^H_{x_0}: C([0,T];\R^m)\rightarrow[0,\infty]$ with compact sublevel sets, such that for all bounded, continuous $h:C([0,T];\R^m)\rightarrow\R $ 
\begin{equation}\label{ldp}
\lim_{\epsilon\to 0}\epsilon\log\ex_{x_0}\big[ e^{-h(X^{\epsilon,\eta})/\epsilon}      \big]= -\inf_{\phi\in C([0,T];\R^m) }\big[ S^{H}_{x_0}(\phi) +h(\phi) \big],
\end{equation}
where, for each $T>0$ and $x_0\in\R^m,$ $\{X_t^{\epsilon,\eta}\}_{t\in[0,T]}$ is the (unique) strong solution of \eqref{model} and $\ex_{x_0}$ indicates that the initial condition is given by $x_0.$ 

Our approach is based on the weak convergence method, which connects the proof of Laplace asymptotics for $X^{\epsilon,\eta}$ to the limit of a stochastic control problem for an associated controlled slow-fast system. Due to its effectiveness in streamlining the proofs of LDPs, this method has been widely used in several different settings and its development can be traced to the monograph \cite{dupuis2011weak} (see also \cite{boue1998variational,zhang2009variational} for the cases of Brownian motion and abstract Wiener spaces respectively). The weak convergence method relies on a variational formula for exponential functionals of the noise. To be more precise 
we have, for all bounded, continuous $h:C([0,T];\R^m)\rightarrow\R ,$
\begin{equation}\label{varrepcon}
\begin{aligned}-\epsilon\log\ex_{x_0}\big[ e^{-h(X^{\epsilon,\eta})/\epsilon}\big]=  \inf_{u\in\mathcal{A}_b}\ex_{x_0}\bigg[ \frac{1}{2}\|u\|^2_{\h_H\oplus\h_{1/2}}+h\big(X^{\epsilon,\eta,u}\big) \bigg],   \end{aligned}\end{equation}
where, for any $H\in(1/2, 1),$ $\h_H\oplus\h_{1/2}$ is the Cameron-Martin space of the noise $( B^H, W)$ and $\mathcal{A}_b$ denotes the family of stochastic controls $u=(u_1,u_2)$ that are adapted to the common filtration $\{\mathscr{F}_t\}_{t\in[0,T]}$ and take values in $\h_H\oplus\h_{1/2}$ with probability $1.$ Here, $X^{\epsilon,\eta,u}$ corresponds to a controlled slow-fast pair $(X^{\epsilon,\eta,u}, Y^{\epsilon,\eta,u})$ (see \eqref{consys} below) which results from \eqref{model} by perturbing the paths of the noise by an appropriately rescaled control.

In view of \eqref{ldp} and \eqref{varrepcon} , it becomes clear that the LDP follows after taking  the limit, as $\epsilon\to 0,$ of the right-hand side of \eqref{varrepcon}. In particular, one needs to understand the limiting behavior of $X^{\epsilon,\eta,u}$ and, before doing so, prove  tightness estimates for the family $\{X^{\epsilon,\eta,u}; \epsilon\ll 1, u\in\mathcal{A}_b  \}$ (in fact, we only need to obtain a tightness result with respect to a smaller class $\mathcal{A}_N\subset \mathcal{A}_b$ of uniformly bounded stochastic controls; for more details we refer the reader to \eqref{AN} and Section \ref{lapluppersec} below). The latter is the first technical part of the current work (Section \ref{tightsec}). The main difficulties lie in the proof of uniform estimates for the pathwise integral $\sigma_1dB^H,$ as well as the presence of the stochastic controls.  Taking advantage of the small-noise regime, we work with the standard Young integral and show tightness in the topology of an appropriate fractional Sobolev space by means of a fractional integration-by-parts formula (see Proposition \ref{tightnessprop}).

Turning to the limiting behavior of $X^{\epsilon,\eta,u},$ note that the presence of stochastic controls $u$ implies that the invariant measure of $Y^{\epsilon,\eta,u}$ depends a priori on $u$. In order to characterize the weak limit points, we introduce a family of random occupation measures $P^{\epsilon}$ \eqref{occupation} that keep track of the stochastic controls and controlled process $Y^{\epsilon,\eta,u},$ as well as a notion of viable pairs (Definition \ref{viable}). The latter are pairs of a trajectory and measure $(\psi,P)$ that capture both the limit averaging dynamics of $X^{\epsilon,\eta,u}$ and the long-time behavior of the controlled fast process $Y^{\epsilon,\eta, u}$. Using these tools, we show in Theorem \ref{main1}  that any limit point of the family $\{ (X^{\epsilon,\eta,u}, P^\epsilon); \epsilon\ll1, u\in\mathcal{A}_N     \},$ in the sense of convergence in distribution, is almost surely a viable pair. Moreover, our assumptions on $\epsilon$ and $\eta$ lead to a decoupling of the controls and the limiting invariant measure of $Y^{\epsilon,\eta,u}.$ In particular, the measure $P$ satisfies $P(dudydt)=\Theta(du|y,t)\mu(dy)dt$, where $ \Theta$ is a stochastic kernel characterizing the control and $\mu$ is the invariant measure of the uncontrolled fast motion.

The viable pair construction has been successfully applied to prove LDPs in the case $H=1/2$ (see e.g. Theorems 2.8 and 3.2 in \cite{dupuis2012large,spiliopoulos2013large} respectively). What is different in our setting is the fact that the controls $u$ have two components $u_1,u_2$ that take values in different Cameron-Martin spaces. In particular, $u_1$ corresponds to the fBm $B^H$ and introduces a non-local term to the limiting dynamics of $X^{\epsilon,\eta,u}.$ The limit of this term as $\epsilon\to 0$ can be expressed as an average with respect to the product measure $P\otimes P,$ per Proposition \ref{sigmaave}. With this characterization of limit points at hand, the Laplace Principle upper bound follows by an application of the Portmanteau lemma. 

As in the case $H=1/2,$ the proof of the lower bound is more complicated because it relies on the explicit construction of approximate minimizing controls that asymptotically achieve the bound. An additional issue, that is completely absent when $H=1/2,$ is that the diffusion coefficient of the limiting dynamics is expressed in terms of a fractional integral operator acting on the time variable. Consequently, our proof makes use of a non-local "effective diffusivity" operator $\mathcal{Q}_H$ \eqref{QH1}, defined on a space of square integrable functions on $[0,T],$ which allows us to construct nearly-optimal controls by inverting the dynamics. The Laplace Principle lower bound, along with the compactness of the sublevel sets of $S^H_{x_0},$ complete the proof of our LDP, Theorem \ref{main2}.

To the best of our knowledge, an LDP for slow-fast systems driven by fBm is established for the first time in this paper. Our contribution is twofold: first, we extend the classical theory to the setting $H>1/2$ by appropriately modifying the weak convergence method. Moreover, we identify a sufficient condition on $H$ and $\sqrt{\eta}/\sqrt{\epsilon}$ (Condition \ref{C6}$(i)$ below) that allows us to consider a fully dependent coefficient $\sigma_1.$ Second, we show that in certain cases the rate function $S^H$ has a non-variational form which is reminiscent of the classical Freidlin-Wentzell rate function $S^{1/2}$. Finally, we identify a functional $\tilde{S}^{1/2}$ \eqref{Sbar} with $\tilde{S}^{1/2}\neq S^{1/2} $ and show that $S^H(\phi)\rightarrow\tilde{S}^{1/2}(\phi),$ as $H\rightarrow{1/2}^+$ for all functions $\phi$ that satisfy certain regularity properties.  Thus, we rigorously prove that $S^H$ is discontinuous at $H=1/2$ (see Proposition \ref{Hlimprop} and Remark \ref{Hlimrem} below).

We believe that the aforementioned discontinuity of the LDP rate function is attributed to the vanishing of long-memory properties of fBm, as $H$ tends to $1/2$ from above, rather than path regularity reasons. To be more precise, the fractional Brownian scale for continuous-time models confounds two very different effects: memory and path regularity. While perturbations of $H$ at any point in $(0, 1)\setminus\{1/2\}$ lead to continuity in $H$ for many statistics of interest, this is not the case for perturbations around $H=1/2.$ However, the theory of Gaussian processes indicates that such processes' moduli of continuity behave well around $H=1/2.$ While a theory for non-linear functionals of Gaussian processes (such as the ones considered in this paper) is not as well-developed, there is no fundamental reason why the continuity of regularity scales should fail when the nonlinearities are sufficiently well behaved. On the contrary, the memory length and Markovian or non-Markovian character of fBm goes through a severe discontinuity as $H$ passes from below $1/2,$ through and at the critical point, to above $1/2.$ For this reason, our aforementioned discontinuity results (Proposition \ref{Hlimprop} and Remark \ref{Hlimrem}) seem to be a consequence of memory effects while many of the technical details of this work are geared at dealing with regularity properties.

The rest of this paper is organized as follows: In Section \ref{notationsec}, we introduce the necessary notation and a few preliminary facts on Young integrals, fractional Brownian motion and Cameron-Martin spaces. Then we state our assumptions for system \eqref{model}. In Section \ref{resultssec}, we give an outline of the weak convergence method and state our main results. Section \ref{tightsec} is devoted to the proof of tightness estimates for the family $\{X^{\epsilon,\eta,u}; \epsilon, u\}$ of controlled slow processes. In Section \ref{limsec} we study the limiting behavior of the family $\{(X^{\epsilon,\eta,u}, P^\epsilon); \epsilon, u\}$ where $P^\epsilon$ are the occupation measures introduced in Section \ref{resultssec}. The proof of the LDP for $X^\epsilon$ is given in Section \ref{LDPsec}. In Section \ref{comparisonsec}, we provide an explicit form of the rate function $S^H$, study its limit as $H\rightarrow{1/2}^+,$ and discuss its differences to the case $H=1/2$. Section \ref{concsec} is devoted to conclusions, directions for future work and also includes a brief discussion on the extendability of our LDP (Theorem \ref{main2}) to slow processes driven by different Gaussian noises. Finally,  Appendix \ref{AppA} collects the proofs of auxiliary lemmas some of which might be of independent interest.

\section{Notation assumptions and preliminaries}\label{notationsec}
\noindent Throughout this work and unless otherwise stated we set $\mathcal{X}:=\R^m, \mathcal{Y}:=\R^{d-m}, \mathcal{U}_1:=\R^k, \mathcal{U}_2:=\R^\ell $. We write $\lesssim$ to denote inequality up to a multiplicative constant that is independent of asymptotic parameters. The lattice notation $\vee,\wedge$ is used to denote maximum and minimum respectively. For any set $A,$ $\mathds{1}_A$ denotes the indicator function of $A.$ Finally, $\oplus$ denotes the Hilbert space direct sum.
\subsection{Function and measure spaces}  For a non-empty open subset $O$ of a Euclidean space, a Banach space $V$ and $k\in\N,$ we denote the  vector space of $k$-times continuously differentiable functions $f:O\rightarrow V$ by $C^k(O; V)$ and the subspace of functions with bounded derivatives up to the $k$-th order by $C^k_b(O; V)$. For a compact set $K,$ $C(K;V)$ denotes the Banach space of continuous functions $f:K\rightarrow V,$ endowed with the topology of uniform convergence. 

Let $T>0$.  The "time"-derivative of a function $\psi$ defined on $[0,T]\times V$ will be frequently denoted by $\dot{\psi}$. For any $\alpha\in(0,1),$ $C^\alpha([0,T];V)$ denotes the Banach space of $\alpha$-H\"older continuous paths, endowed with the norm 
\begin{equation}\label{+norm}
\|X\|_{C^\alpha}=\|X\|_{\infty}+[X]_{C^\alpha}:=\sup_{t\in[0,T]}|X_t|+\sup_{s,t\in[0,T],s\neq t}\frac{|X_t-X_s|}{|t-s|^\alpha}.
\end{equation}
For any $V-$valued path $\{X_t\}_{t\in[0,T]}$, $\alpha\in(0,1)$ and $0\leq s<t\leq T$ we define 
\begin{equation}\label{Deltapm}
\Delta_{\alpha}X_{s,t}:=\int_{s}^{t}\frac{X_t-X_r}{(t-r)^{\alpha+1}}dr\;,\quad	\Delta^{-}_{\alpha}X_{s,t}:=\int_{s}^{t}\frac{X_r-X_s}{(r-s)^{\alpha+1}}dr
\end{equation}
and
\begin{equation}\label{Deltapmabs}
|\Delta_{\alpha}|X_{s,t}:=\int_{s}^{t}\frac{|X_t-X_r|}{(t-r)^{\alpha+1}}dr\;,\quad	|\Delta^{-}_{\alpha}|X_{s,t}:=\int_{s}^{t}\frac{|X_r-X_s|}{(r-s)^{\alpha+1}}dr.
\end{equation}

We denote by $W_{0}^{\alpha,\infty}([0,T];V)$ and $W_{T}^{\alpha,\infty}([0,T];V)$ the vector spaces of measurable paths, up to equality almost everywhere, such that  
\begin{equation}\label{W0norm}
\|X\|_{0,\alpha,\infty}:=\sup_{t\in[0,T]}\big(|X_t|+|\Delta_\alpha|X_{0,t}\big)<\infty
\end{equation}
and 
\begin{equation}\label{-norm}
\|X\|_{T,\alpha,\infty}:=\sup_{s,t\in[0,T],t\neq s}\bigg(\frac{|X_t-X_s|}{|t-s|^{\alpha}}+|\Delta^{-}_{\alpha}|X_{s,t}\bigg)<\infty
\end{equation}
respectively. The spaces  $(W_{0}^{\alpha,\infty}([0,T];V), \|\cdot\|_{0,\alpha,\infty}  ),$ $(W_{T}^{\alpha,\infty}([0,T];V), \|\cdot\|_{T,\alpha,\infty}  )$ are Banach spaces that interpolate between H\"older spaces. In particular, for all $\alpha<1/2$, $\gamma_1\in(0,\alpha), \gamma_2>0$ we have the continuous inclusions
\begin{equation}\label{pmembeds}
\begin{aligned}
&C^{\alpha+\gamma_1}([0,T];V)\subset W_{0}^{\alpha,\infty}([0,T];V)\subset C^{\alpha-\gamma_1}([0,T];V),\\ \\&
C^{1-\alpha+\gamma_2}([0,T];V)\subset W_{T}^{1-\alpha,\infty}([0,T];V)\subset C^{1-\alpha}([0,T];V),
\end{aligned}	      
\end{equation}
see e.g. \cite{nualart2006malliavin}, Section 5.3.1.

For two Banach spaces $V_1,V_2,$ $\mathscr{L}(V_1;V_2)$ is the Banach space of bounded linear maps from $L:V_1\rightarrow V_2$ endowed with the norm
$$\|L\|_{V_1\to V_2}:=\sup_{|v|_1\leq 1}|Lv|_2.$$
We shall also use the simpler notation $\mathscr{L}(V_1):=\mathscr{L}(V_1;V_1)$ when the domain and codomain coincide.

The product measure space of two measure spaces $\{(E_i, \mathcal{M}_i, \mu_i)\}_{i=1,2}$ is denoted by $(E_1\times E_2, \mathcal{M}_1\otimes \mathcal{M}_2,\mu_1\otimes\mu_2).$ If $\mathcal{M}_1=\mathcal{M}_2=\mathcal{M}$ and $\mu_1=\mu_2=\mu$ we shall write  $\mu^{\otimes 2}$ for the product measure. For $p\in[1,\infty),$ we denote the Lebesgue spaces of $p$-integrable and essentially bounded classes of measurable functions $f:E_1\rightarrow V$ by $L^p(E_1;V)$ and $L^{\infty}(E_1;V)$ respectively. The latter are Banach spaces when endowed with the norms $\|f\|^p_{L^p}:=\int_{E_1}|f|^pd\mu_1$ and $\|f\|_{L^\infty}:=\esssup|f|$. The family of $\mu_1$-measurable functions will be denoted by $L^0(E_1,V).$

For the purposes of this paper a Polish space is defined to be a separable, completely metrizable topological space. The Borel $\sigma$-algebra on a Polish space $\mathcal{E}$ is denoted by $\mathscr{B}(\mathcal{E}).$ The space of finite Borel measures on $\mathcal{E},$  endowed with the topology of weak convergence of measures, is denoted by $\mathscr{P}(\mathcal{E})$ (the latter is itself a Polish space, see e.g. \cite{ethier1986markov}, Theorem 1.7, pp. 101 and Theorem 3.8, pp. 108 for a proof).
\subsection{Fractional calculus and generalized Stieltjes integration.} In this section we introduce a few necessary notions from fractional calculus. For any $f\in L^1(a,b)$ and $\alpha>0$,  the left-sided and right-sided fractional \textit{Riemann-Liouville integrals} of $f$ of order $\alpha$ are defined for almost all $t\in(a,b)$ by
\begin{equation}\label{Idef}
I^\alpha_{a^+}f(t):=\frac{1}{\Gamma(\alpha)}\int_{a}^{t}(t-r)^{\alpha-1}f(r)dr
\end{equation}
and
\begin{equation*}\label{RL}
I^\alpha_{b^-}f(t):=\frac{(-1)^{-\alpha}}{\Gamma(\alpha)}\int_{t}^{b}(r-t)^{\alpha-1}f(r)dr,
\end{equation*} 
where $\Gamma$ is the Euler gamma function. For $\alpha\in(0,1)$ and
$f\in I^\alpha_{a^+}[L^p(a,b)]$  the left-sided \textit{Marchaud derivative}
of $f$ of order $\alpha$ is defined by
\begin{equation}\label{Weyl+}
D^\alpha_{a^+}f(t)=\frac{d}{dt}I^{1-\alpha}_{a^+}f(t)=\frac{1}{\Gamma(1-a)}\bigg[\frac{f(t)}{(t-a)^\alpha}+\alpha \Delta_{\alpha}f_{a,t}\bigg]\mathds{1}_{(a,b)}(t).
\end{equation}
For $f\in I^\alpha_{b^-}[L^p(a,b)]$ we define the right-sided \textit{Marchaud derivative}
of $f$ of order $\alpha$ by
\begin{equation}\label{Weyl-}
D^\alpha_{b^-}f(t)=\frac{d}{dt}I^{1-\alpha}_{b^-}f(t)=\frac{(-1)^\alpha}{\Gamma(1-a)}\bigg[\frac{f(t)}{(b-t)^\alpha}+\alpha \Delta^-_{\alpha}f_{t,b}\bigg]\mathds{1}_{(a,b)}(t),
\end{equation}
where we recall that $\Delta_\alpha, \Delta^{-}_\alpha$ are defined in \eqref{Deltapm}. The second equalities in \eqref{Weyl+}, \eqref{Weyl-} are also known in the literature as the Weyl representations of the Marchaud fractional derivatives, see e.g. \cite{Samko1993FractionalIA}, Chapter 13.1 and \cite{zahle1999link}, Section 2. For a detailed exposition of fractional calculus and general properties of fractional operators the reader is referred to the monograph \cite{Samko1993FractionalIA}.  What is useful in our setting is a fractional integration-by-parts formula which provides the following extension of the Stieltjes integral. 
\begin{dfn}
	\label{Young} Let $p,q\in(1,\infty)$ be conjugate exponents, $\alpha<1/p$
	and for any measurable function $g$ define $$g(b^-):=\lim_{\epsilon\downarrow 0} g(b-\epsilon)$$ and 
	\begin{equation*}\label{gminus}
	g_{b^-}(t):=[g(t)-g(b^-)]\mathds{1}_{(a,b)}(t).
	\end{equation*}
	For any $f\in  I^\alpha_{a^+}[L^p(a,b)]$ and $g$ such that $g_{b^-}\in I^{1-\alpha}_{b^-}[L^q(a,b)]$, the integral of $f$ with respect to $g$ is defined by
	\begin{equation}\label{ibpformula}
	\int_{a}^{b}fdg:=(-1)^\alpha\int_{a}^{b} D^\alpha_{a^+}f(t)D^{1-\alpha}_{b^-}g_{b^-}(t)dt.
	\end{equation}
\end{dfn}
\begin{rem}
	For $f\in C^{\theta_1}, g\in C^{\theta_2}$ such that $\theta_1+\theta_2>1$ one can take $p=q=\infty$, $\alpha\in(1-\theta_2, \theta_1)$ and show that the integral coincides with the extension of the classical Stieltjes integral studied by Young in \cite{young1936inequality}. Due to the latter, the integral defined above is commonly known as a Young integral.  For the use of Young integrals in the development of stochastic calculus with respect to fractional Brownian motion, we refer the reader to the work of Z\"ahle \cite{zahle1999link}.
\end{rem} 

\subsection{Fractional Brownian motion and Cameron-Martin spaces.}\label{fbmsec} A (one-dimensional) fractional Brownian motion (fBm)  $\{B^H_t\}_{t\geq 0}\subset L^2(\Omega)$ is a centered Gaussian process characterized by its covariance function
$$R_H(t,s)=\ex[B^H_tB^H_s ]=\frac{1}{2}\bigg( s^{2H} +t^{2H} -|t-s|^{2H}     \bigg).$$
It is straightforward to verify that increments of fBm are stationary. The parameter $H\in (0, 1)$ is usually referred to as the Hurst exponent, Hurst parameter, or Hurst index. Note that for $H=1/2$ we obtain $R_{1/2}(t,s)=t\wedge s.$ Thus, one sees that $B^{1/2}$ is a standard Brownian motion, and in particular that its disjoint increments are independent. In contrast to this, when $H\neq 1/2$ , nontrivial increments are not independent and, when $H > 1/2$ , the process exhibits long-range dependence.

By Kolmogorov's continuity criterion, a $d$-dimensional fBm $B^H$ admits, for any $T>0$ and $\beta<H$, a modification with sample paths in $C^\beta([0,T];\R^d).$ Moreover, for any $\alpha\in(1-H,1/2),\theta\in(0,2),$ the random variable $\|B^H\|_{T,1-\alpha,\infty}$ has finite moments of all orders and,  by virtue of Fernique's theorem, 
\begin{equation}\label{fbmexpint}
\ex\bigg[\exp\bigg(\frac{\theta\|B^H\|_{T,1-\alpha,\infty}}{\Gamma(\alpha)\Gamma(1-\alpha)}\bigg)\bigg]<\infty;
\end{equation}
see e.g. Lemma 7.5 of \cite{rascanu2002differential} and (2.3) in \cite{pei2020averaging}.

From this point on, we fix $H\in(1/2,1).$ We denote by $\h_H$ the Cameron-Martin space of $B^H,$ defined by
\begin{equation}\label{CMdef}
\h_H:=\big( K_H(L^2[0,T]),   \langle \cdot, \cdot\rangle_{\h_H}        \big),
\end{equation}
where the operator $K_H:L^2[0,T]\rightarrow L^2[0,T]$ is given by
\begin{equation}\label{KH}
K_H(f)(t):=c_HI^{1}_{0^+}\big(\Phi\cdot I^{H-\frac{1}{2}}_{0^+} (\Phi^{-1}\cdot f     )  \big)(t),\;\;\Phi(t)=t^{H-\frac{1}{2}},
\end{equation}
\begin{equation}\label{cH}
c^2_H:=\frac{2H\Gamma(\frac{3}{2}-H)\Gamma(H+\frac{1}{2})}{\Gamma(2-2H)}
\end{equation}
and the inner product is defined by $$\langle f, g\rangle_{\h_H}:=\langle K^{-1}_Hf, K^{-1}_Hg\rangle_{L^2}.$$
Note that, by construction, $\h_H$ is a Hilbert space and the operator $K_H$ is an isomorphism between $L^2[0,T]$ and $I^{H+1/2}_{0^+}(L^2[0,T]).$ The inverse $K_H^{-1}$ is well-defined for functions $f\in I^{H+1/2}_{0^+}(L^2[0,T])$ and is given by
\begin{equation}\label{Kinverse}
\begin{aligned}
K^{-1}_H(f)(t)&= c^{-1}_H\Phi D^{H-\frac{1}{2}}_{0^+}\big( \Phi^{-1}\frac{df}{dt}\big)(t)\\&
=\bigg(c_H\Gamma\big(\tfrac{3}{2}-H\big)\bigg)^{-1}\bigg[t^{\frac{1}{2}-H}\dot{f}(t)+\bigg(H-\frac{1}{2}\bigg)t^{H-\frac{1}{2}}\int_{0}^{t}\frac{t^{\frac{1}{2}-H}\dot{f}(t)-s^{\frac{1}{2}-H}\dot{f}(s)}{(t-s)^{H+\frac{1}{2}}}ds\bigg].
\end{aligned}
\end{equation}
The interested reader is referred to \cite{budhiraja2020large},\cite{gross1967abstract}, Chapter 5.1.3 of \cite{nualart2006malliavin} and \cite{stroock2010probability} Sections 8.1.2, 8.2.3 for more details on the construction of the Cameron-Martin space of a Gaussian measure.
\begin{rem}\label{CM1/2rem}
	For $H=1/2,$ the vector space $\h_{1/2}$ coincides (with equivalence of norms) 
	with the Sobolev space $H^1_0([0,T])$ of absolutely continuous functions $f$ with a square-integrable weak derivative and $f(0)=0$. The latter is the Cameron-Martin space of a standard Brownian motion.
\end{rem}
The Cameron-Martin space of the fractional Brownian noise $"\dot{B}^H"$, viewed as a random distribution, is denoted by $\mathfrak{H}$. As a set, it consists of distributions $f$ such that $\frac{d}{dt}(K_Hf)\in L^2[0,T]$ (see \cite{pipiras2001classes}, as well as  \cite{pipiras2000integration}). 
In light of Remark 4.2 of \cite{pipiras2001classes} we have the continuous inclusions
\begin{equation}\label{Taqquembed}
L^2\subset L^{\frac{1}{H}}\subset |\mathfrak{H}|\subset \mathfrak{H},
\end{equation}
where $(|\mathfrak{H}|, \|\cdot\|_{|\mathfrak{H}|})$ denotes the Banach space of measurable functions $f$ such that 
\begin{equation*}
\|f\|_{|\mathfrak{H}|}:=H(2H-1)\iint_{[0,T]^{2}}|f(t)||f(s)||t-s|^{2H-2}dsdt<\infty.
\end{equation*}
Throughout this work, integrals with respect to the fBm $B^H$ are Young integrals, in the sense of Definition \ref{Young}. The integrals with respect to the Brownian motion $W$ are standard It\^o integrals. \\
\noindent \subsection{Assumptions}\label{subsecAssumptions} We shall now state our assumptions for system \eqref{model}. As we pointed out in Section \ref{intro}, we work with $\epsilon$ and $\eta=\eta(\epsilon)$  in the asymptotic regime 
\begin{equation}\label{regime}
\lim_{\epsilon\to0}\frac{\sqrt{\eta}}{\sqrt{\epsilon}}=0.
\end{equation}
\noindent Regarding the coefficients of the fast motion we assume:
\begin{customthm}{1}\label{C1} There exists a constant $C_g>0$ such that for all $(x,y)\in\mathcal{X}\times\mathcal{Y},$ \[\big|g(x,y)\big|\leq C_g\big(1+|y|\big).\]
\end{customthm}
\begin{customthm}{2}\label{C2} The matrix-valued function $\tau\in C^1_b$ and $ \tau\tau^T$ is uniformly nondegenerate.
\end{customthm}
\begin{customthm}{3}\label{C3}  For $y\in\mathcal{Y},$ $f(y)=-\Gamma y+\zeta(y)$, where $\Gamma$ is a positive matrix, $\zeta\in C^1_b$ and there exists 
	$ C_f>0$ such that for all $y\in\mathcal{Y}, \langle (\Gamma-L_\zeta\cdot I)y,y\rangle\geq C_f|y|^2$,
	where $L_\zeta:=\|\nabla\zeta\|_{\infty}$. 
\end{customthm}

\noindent Conditions \ref{C2}-\ref{C3} guarantee that the the It\^o diffusion with infinitesimal generator 
\begin{equation}\label{generator}
\mathcal{L}\phi(y)=\frac{1}{2}[D^2\phi:(\tau\tau^T)](y)+\nabla \phi(y)f(y)\;,y\in\mathcal{Y},\;\;\phi\in C^2(\mathcal{Y}),
\end{equation}
where $D^2\phi:(\tau\tau^T)(y):=\text{trace}(D^2\phi(y)(\tau\tau^T)(y)),$ is strongly mixing and has on $\mathcal{Y}$ a unique invariant measure $\mu.$ Our next set of assumptions concerns the coefficients of the slow motion.

\begin{customthm}{4}\label{C4}The functions $c,\sigma_1,\sigma_2$ are Lipschitz continuous and the matrix-valued function $\sigma_2$ is uniformly bounded.
	The function $b$ is differentiable with $\nabla b\in C_b$ and moreover it satisfies the centering condition 
	\begin{equation}\label{centering}
	\bar{b}:=\int_{\mathcal{Y}}bd\mu=0,
	\end{equation}
	where $\mu$ is the invariant measure corresponding to the operator $\mathcal{L}$ \eqref{generator}.
\end{customthm}
\begin{customthm}{5}\label{C5}There exist constants $ K_1>0,\nu\in(0,1)$ such that for all $(x,y)\in\mathcal{X}\times\mathcal{Y},$ $|c(x,y)|\leq K_1(1+|x|^\nu+|y|).$
\end{customthm}

\noindent The following condition concerns the coefficient $\sigma_1$ in \eqref{model}. In particular, we provide two different sets of assumptions on $\sigma_1, H,\epsilon$ and $\eta,$ under which our main results hold.
\begin{customthm}{6}\label{C6} We assume that one of the following holds:\\
	\textbf{(i)}
	$\sigma_1=\sigma_1(x,y)$, $H\in(\frac{3}{4},1)$  and there exists $\beta\in(2(1-H), \frac{1}{2})$ such that 
	\begin{equation}\label{betaregime}
	\lim_{\epsilon\to 0}\frac{\sqrt{\epsilon}}{\eta^\beta}=0.
	\end{equation}
	Moreover, there exist constants $K_2>0,\nu_1\in(0, \frac{1}{2})$ and $\nu_2\in (2(1-H), \frac{1}{2})$, such that for all $(x,y)\in\mathcal{X}\times\mathcal{Y},$ \begin{equation}\label{sigmagrowth1}
	|\sigma_1(x,y)|\leq K_2(1+|x|^{\nu_1}+|y|^{\nu_2}).
	\end{equation}
	\textbf{(ii)} $\sigma_1=\sigma_1(x), H\in(\frac{1}{2},1)$ and there exist constants $K_2>0,\nu_1\in(0,\frac{1}{2})$ such that for all $(x,y)\in\mathcal{X}\times\mathcal{Y},$ \begin{equation}
	\label{sigmagrowth2}|\sigma_1(x)|\leq K_2(1+|x|^{\nu_1}).
	\end{equation}  	 
\end{customthm}

\noindent A few preliminary comments on Conditions \ref{C1}-\ref{C6} are given in the following remark:
\begin{rem} 1) The coefficient $g$ is allowed to depend on both the slow and fast variables and, due to \eqref{regime}, is asymptotically unimportant for the long-time behaviour of the process $Y^{\epsilon, \eta}$. While Condition \ref{C1} is not optimal, an investigation of optimal growth conditions for $g$ is beyond the scope of our work. The assumption that $g$ is bounded in the slow variable is made to simplify our estimates for the fast motion (Lemmas \ref{Ybndlem}, \ref{Yfraclem}) which, in turn, play an important role for the proof of our tightness results (Proposition \ref{tightnessprop}).  2) The assumption that $\sigma_2$ is Lipschitz continuous is not required for the analysis of Section \ref{tightsec} or the proof of the Laplace principle upper bound and will only be used in the proof of the lower bound (see Lemma \ref{tcont}). 3) The growth assumptions of Conditions \ref{C4}, \ref{C5}, \ref{C6}(i), (ii) are used in Section \ref{tightsec}, to prove the tightness estimates of Proposition \ref{tightnessprop}. 4) In  Sections \ref{limsec}-\ref{comparisonsec} we replace the growth assumptions \eqref{sigmagrowth1}, \eqref{sigmagrowth2} of Condition 6 with the stronger condition that $\sigma_1$ is bounded. 5) The asymptotic regime
	$$ \sqrt{\eta}\lesssim \sqrt{\epsilon}\lesssim \eta^\beta\;,\;\;\text{as}\;\epsilon\to 0$$
	of \eqref{regime}, and Condition \ref{C6}$(i)$ allows us to show both that the presence of stochastic controls preserves the ergodic properties of the fast motion and that the stochastic integral term $\sigma_1dB^H$ is uniformly bounded over small values of $\epsilon$ (see also Remarks \ref{ergodicproprem}, \ref{relativeraterem} below).
\end{rem} 

\noindent Conditions \ref{C2}, \ref{C3} and the centering condition \eqref{centering} guarantee that the Poisson equation
\begin{equation}\label{Poisson}
\left \{\begin{aligned}
&	\mathcal{L}\Psi(y)=-b(y)\;,\;\;y\in\mathcal{Y}\\&
\textstyle\int_{\mathcal{Y}}\Psi(y)d\mu(y)=0,
\end{aligned}\right.
\end{equation}
with $\mathcal{L}$ as in \eqref{generator}, has a unique solution $\Psi\in C^2(\mathcal{Y};\mathcal{X})$ in the class of functions that grow at most polynomially.  Moreover, we have for all $y\in\mathcal{Y}$
\begin{equation*}\label{Psilin}
|\Psi(y)|\leq C(1+|y|)
\end{equation*}
for some constant $C>0$ (see e.g. \cite{Pardoux} or \cite{ganguly2021inhomogeneous}, Proposition A.2 for a proof). Such equations have been studied under general assumptions and applied to the theory of non-periodic homogenization; see e.g. \cite{Pardoux}. As we shall see in Section \ref{resultssec}, the solution of \eqref{Poisson} is connected to the asymptotic analysis of the singular term $\tfrac{\sqrt{\epsilon}}{\sqrt{\eta}}bdt.$ 

\begin{customthm}{7}\label{C7} The map $\mathcal{Y}\ni y\longmapsto \nabla\Psi(y)\in\mathscr{L}(\mathcal{Y};\mathcal{X})$ is bounded and Lipschitz continuous.
\end{customthm}

\begin{rem} Condition \ref{C7} is made to simplify the exposition and proofs of the following sections and is by no means optimal for the results of this paper to hold. In Lemmas \ref{psider}, \ref{psiliplem} of Appendix \ref{AppA} we provide sufficient conditions on the coefficients of the fast motion under which Condition \ref{C7} is satisfied.
\end{rem}

\noindent Throughout this work we assume that the following condition is in effect.
\begin{customthm}{8}\label{C8}	For each $\epsilon,\eta, T>0$, \eqref{model} has a unique strong solution $$\{ (X^{\epsilon,\eta}_t,Y^{\epsilon,\eta}_t)\}_{t\in[0,T]}\subset L^0(\Omega; C([0,T];\mathcal{X}\times\mathcal{Y})). $$
\end{customthm}
\begin{rem} The reader is referred to \cite{guerra2008stochastic,kubilius2002existence,mishura2011existence, rascanu2002differential} for existence and uniqueness results for SDEs driven by fractional Brownian motion. Our setup is close to that of \cite{guerra2008stochastic} where the authors consider mixed SDEs like \eqref{model} and $dW, dB^H$ are interpreted as It\^o and Young integrals respectively. In particular, if the coefficients $b,c,f,g,\sigma_2,\tau$ are Lipschitz continuous and grow at most linearly in all the arguments and $\sigma_1\in C^1$ has a bounded derivative that is $\delta$-H\"older continuous, for some $\delta\in(0,1],$ then Theorem 2.2 of \cite{guerra2008stochastic} asserts that Condition \ref{C8} is satisfied. Moreover, the same theorem yields that, for all $\alpha\in(1-H,\frac{1\wedge \delta}{2}),$ $X^{\epsilon,\eta}\in W_0^\alpha([0,T];\mathcal{X})$ with probability $1.$	
\end{rem}
Before we conclude this section, we emphasize that we use the notation 
$$ \bar{\phi}(x):=\int_{\mathcal{Y}}\phi(x,y)d\mu(y)\;,\; x\in\mathcal{X}$$
to denote the integral of a function $\phi$ with respect to the invariant measure $\mu.$

\section{Weak convergence method and main results}\label{resultssec}
In this section we review the weak convergence approach to large deviations for \eqref{model} and then we state our main results on the averaging principle for the controlled process $X^{\epsilon,\eta,u}$ and the LDP for $\{X^{\epsilon,\eta}\}$. To this end, we fix $T>0, H\in(1/2,1)$ and an independent pair $\{(B_t^H, W_t)\}_{t\in[0,T]}$ of an fBm and Brownian motion defined on the filtered probability space $(\Omega,\mathscr{F},\{\mathscr{F}_t\}_{t\geq 0},\pr).$

The starting point of the weak convergence method lies in a variational representation for exponential functionals of the noise. In particular, from Theorem 3.2 of \cite{zhang2009variational} (see also \cite{budhiraja2020large}, Proposition 3.1) we have for any bounded, Borel  $F:C([0,T];\R^d)\rightarrow\R$ 
\begin{equation*}
\label{varrep}
-\log\ex e^{-F(B^H,W)}=\inf_{u=(u_1,u_2)\in\mathcal{A}_b}\ex\bigg[ \frac{1}{2}\|u\|^2_{\h_H\oplus\h_{1/2}}+F\big((B^H,W)+(u_1,u_2)\big)    \bigg],
\end{equation*}
where $\h_H\oplus\h_{1/2}$ is the Cameron-Martin space of the driving noise $(B^H, W),$ see \eqref{CMdef} and Remark \ref{CM1/2rem} above, and $\mathcal{A}_b$ is a family of stochastic controls given by
\begin{equation}\label{Ab}
\mathcal{A}_b=\bigg\{ u\in L^0(\Omega,\h_H\oplus\h_{1/2} ) :\forall t\in[0,T]\; u(t)\;\text{is}\;\mathscr{F}_t-\text{measurable},\; \|u\|_ {\h_H\oplus\h_{1/2}}<\infty\;\text{a.s.}        \bigg\}.
\end{equation}
Moreover, we shall consider the smaller class $\mathcal{A}_N$ of stochastic controls in $\mathcal{A}_b$ that are uniformly bounded by a constant $N>0$ i.e.
\begin{equation}\label{AN}
\mathcal{A}_N=\bigg\{ u\in \mathcal{A}_b: \|u\|_ {\h_H\oplus\h_{1/2}}\leq N\;\text{a.s.}        \bigg\}.
\end{equation}
As we shall explain in Section \ref{lapluppersec}, the class $\mathcal{A}_N$ is sufficient for the proof of the Laplace Principle upper bound as long as $N$ is taken sufficiently large.

In view of Condition \ref{C8}, there exists, for each $\epsilon,\eta>0$ and initial conditions $(x_0,y_0)\in\R^d$ a measurable map $G^{\epsilon,\eta}: C([0,T];\R^d)\rightarrow C([0,T];\mathcal{X})$ such that 
$$ X^{\epsilon, \eta}= G^{\epsilon,\eta}(B^H, W)$$
with probability $1.$ Letting $h\in C_b\big(C([0,T];\mathcal{X});\R)$, we replace $F$ by $h\circ G^{\epsilon,\eta}$ and rescale to obtain the variational formula
\begin{equation}\label{varform}
\begin{aligned}-\epsilon\log\ex_{x_0}\big[ e^{-h(X^{\epsilon,\eta})/\epsilon}\big]=  \inf_{u\in\mathcal{A}_b}\ex_{x_0}\bigg[ \frac{1}{2}\|u\|^2_{\h_H\oplus\h_{1/2}}+h\big(X^{\epsilon,\eta,u}\big) \bigg],   \end{aligned}\end{equation}
where the process $$ X^{\epsilon, \eta,u}:= G^{\epsilon,\eta}\bigg(B^H+\frac{1}{\sqrt{\epsilon}}u_1, W+\frac{1}{\sqrt{\epsilon}}u_2\bigg)$$
corresponds to the controlled slow-fast system
\begin{equation}\label{consys}
\left \{\begin{aligned}
& dX^{\epsilon,\eta,u}_t= \tfrac{\sqrt{\epsilon}}{\sqrt{\eta}}b(Y^{\epsilon,\eta,u}_t)dt+c(X_t^{\epsilon,\eta,u}, Y^{\epsilon,\eta,u}_t)dt+\sigma_1(X^{\epsilon,\eta,u}_t,Y^{\epsilon,\eta,u}_t)du_1(t)\\&\quad\quad\quad+\sigma_2(X^{\epsilon,\eta,u}_t,Y^{\epsilon,\eta,u}_t)du_2(t)+\sqrt{\epsilon}\bigg[\sigma_1(X^{\epsilon,\eta,u}_t,Y^{\epsilon,\eta,u}_t)dB^H_t+\sigma_2(X^{\epsilon,\eta,u}_t,Y^{\epsilon,\eta,u}_t)dW_t\bigg]\\&
dY^{\epsilon,\eta,u}_t= \frac{1}{\eta}f( Y^{\epsilon,\eta,u}_t)dt+\frac{1}{\sqrt{\epsilon\eta}}g(X^{\epsilon,\eta,u}_t,Y^{\epsilon,\eta,u}_t)dt+\frac{1}{\sqrt{\epsilon\eta}}\tau(Y^{\epsilon,\eta,u}_t)du_2(t)\\&\quad\quad\quad+\frac{1}{\sqrt{\eta}}\tau(Y^{\epsilon,\eta,u}_t)dW_t\\&
X^{\epsilon,\eta,u}_0=x_0\in\R^m, Y^{\epsilon,\eta,u}_0=y_0\in\R^{d-m}.
\end{aligned} \right.
\end{equation}
As mentioned in Section \ref{intro}, a Laplace Principle for the family $\{X^{\epsilon,\eta};\epsilon>0\}$ follows by studying the limit of \eqref{varform} as $\epsilon\to 0.$ The Laplace Principle is equivalent to an LDP with the same rate function, provided that the latter has compact sublevel sets, see e.g. \cite{dupuis2011weak}, Theorems 1.2.1, 1.2.3 for a proof. In order to understand the limiting behavior of the controlled process $X^{\epsilon,\eta,u}$ one needs to keep track of both the stochastic controls $(u_1,u_2)$ and the long-time behavior of the controlled fast motion $Y^{\epsilon,\eta,u}$. For this reason, we introduce a family of random occupation measures $\{P^{\epsilon}\;;\epsilon\in(0,1)\}$ on $\mathscr{B}([0,T]\times\mathcal{U}_1\times\mathcal{U}_2\times\mathcal{Y})$ given by
\begin{equation}
\label{occupation}
\begin{aligned}
P^{\epsilon}(A_1\times A_2\times A_3\times A_4)&:=
\int_{A_1}\mathds{1}_{A_2}\big( K^{-1}_{H}u^{\epsilon}_1(s) \big)\mathds{1}_{A_3}\big( \dot{u}^{\epsilon}_2(s) \big)\mathds{1}_{A_4}\big( Y^{\epsilon,\eta,u^\epsilon}_s \big)ds.
\end{aligned}
\end{equation}
Here, for each $\epsilon>0,$ $Y^{\epsilon,\eta,u^\epsilon}$ is controlled by  $(u^{\epsilon}_1, u^{\epsilon}_2)\in\mathcal{A}_b, \dot{u}^\epsilon_2$ is the time-derivative of $u^\epsilon_2$ and $K^{-1}_H$ is the operator defined in \eqref{Kinverse}. Note that since $(u^\epsilon_1,u^\epsilon_2)\in\h_H\oplus \h_{1/2} $ almost surely, $(K^{-1}_Hu^\epsilon_1,\dot{u}^\epsilon_2) \in L^2([0,T];\mathcal{U}_1)\oplus L^2([0,T];\mathcal{U}_2) $ is well-defined. The reason why we define $P^\epsilon$ with respect to a family of controls $u^\epsilon=\{(u^{\epsilon}_1, u^{\epsilon}_2);\epsilon>0\}$ instead of a single control $u=(u_1, u_2)$ shall become clear in our proof of the Laplace Principle upper bound (see Section \ref{lapluppersec} below).

Assuming for the moment that,  as $\epsilon\to 0,$  the sequence $(X^{\epsilon,\eta,u^\epsilon}, P^\epsilon)$ converges in distribution to a pair $(\psi, P)\in C([0,T];\mathcal{X})\times\mathscr{P}([0,T]\times\mathcal{U}_1\times\mathcal{U}_2\times\mathcal{Y}), $ the next step of the method is to characterize the law of $(\psi,P).$  To this end, note that the analysis of the second and fourth terms in the first equation of \eqref{consys} is straightforward. In particular, for each $t\in[0,T]$ we have
\begin{equation*}
\begin{aligned}
&\int_{0}^{t}c(X_s^{\epsilon,\eta,u}, Y^{\epsilon,\eta,u}_s)ds+\int_{0}^{t}\sigma_2(X^{\epsilon,\eta,u}_s,Y^{\epsilon,\eta,u}_s)du_2(s)\\&=\int_{0}^{t}\bigg[c(X_s^{\epsilon,\eta,u}, Y^{\epsilon,\eta,u}_s)+\sigma_2(X^{\epsilon,\eta,u}_s,Y^{\epsilon,\eta,u}_s)\dot{u}_2(s)\bigg]ds\\&
=\int_{[0,t]\times\mathcal{U}_1\times\mathcal{U}_2\times\mathcal{Y}}\bigg[c(X_s^{\epsilon,\eta,u},y)+\sigma_2(X^{\epsilon,\eta,u}_s,y)u_2\bigg]dP^{\epsilon}(s,u_1,u_2,y)
\end{aligned}
\end{equation*}
and thus we expect that, as $\epsilon\to 0,$ the latter will converge in distribution to 
\begin{equation*}
\begin{aligned}
\int_{[0,t]\times\mathcal{U}_1\times\mathcal{U}_2\times\mathcal{Y}}\bigg[c(\psi_s,y)+\sigma_2(\psi_s,y)u_2\bigg]dP(s,u_1,u_2,y).
\end{aligned}
\end{equation*}

\noindent From an application of It\^o's formula to the process $\{\Psi(Y^{\epsilon,\eta,u}_t)\}_{t\in[0,T]},$ where $\Psi$ is the solution of the Poisson equation \eqref{Poisson}, it is then possible to show that the limit of the singular term $\tfrac{\sqrt{\epsilon}}{\sqrt{\eta}}bdt$ is captured by
$$    \int_{[0,t]\times\mathcal{U}_1\times\mathcal{U}_2\times\mathcal{Y}}\nabla\Psi(y)\big[\tau(y)u_2+g(\psi_s,y)\big]dP(s,u_1,u_2,y). $$

The terms considered up to this point also appear in the case $H=1/2$ and their limiting behavior is the subject of Lemma \ref{locave} below. In our setting, the essential difference in the asymptotic analysis of the controlled slow motion is related to the term $\sigma_1du_1$ in \eqref{consys}. In order to treat this term, we first note that $K^{-1}_Hu_1\in L^2([0,T];\mathcal{U}_1)$ almost surely and furthermore, in view of \eqref{KH} and \eqref{Idef}, $u_1$ has a square-integrable weak derivative, defined almost everywhere on $[0,T].$ Thus we can write $$\dot{u}_1=\frac{d}{dt}K_H\big(K^{-1}_Hu_1)=\dot{K}_H\big(K^{-1}_Hu_1),$$
where, for any $v\in L^2([0,T];\mathcal{U}_1),$
\begin{equation}\label{Kdot}
\dot{K}_Hv(s):=\frac{d}{dt}\circ  K_Hv(s)=\frac{c_H}{\Gamma(H-\frac{1}{2})}s^{H-\frac{1}{2}}\int_{0}^{s}z^{\frac{1}{2}-H}(s-z)^{H-\frac{3}{2}}v(z)dz\;,\;\;s\in[0,T].
\end{equation} Combining the last two displays, we derive the following expression in terms of the occupation measures \eqref{occupation}:
\begin{equation}\label{sigma1comp}
\begin{aligned}
&\int_{0}^{t}\sigma_1( X_s^{\epsilon,\eta,u},Y^{\epsilon,\eta,u}_s)du_1(s)=\int_{0}^{t}\sigma_1( X_s^{\epsilon,\eta,u},Y^{\epsilon,\eta,u}_s)\dot{u}_1(s)ds\\&= \frac{c_H}{\Gamma(H-\frac{1}{2})} \int_{0}^{t}\int_{0}^{s}s^{H-\frac{1}{2}}z^{\frac{1}{2}-H}(s-z)^{H-\frac{3}{2}} \sigma_1(X_s^{\epsilon,\eta,u},Y^{\epsilon,\eta,u}_s)[K^{-1}_Hu_1(z)]dzds\\&
= \frac{c_H}{\Gamma(H-\frac{1}{2})}\int_{[0,t]\times\mathcal{Y} }\int_{[0,s]\times\mathcal{U}_1}\\&\quad\quad\quad s^{H-\frac{1}{2}}z^{\frac{1}{2}-H}(s-z)^{H-\frac{3}{2}} \sigma_1(X_s^{\epsilon,\eta,u},y_2)v_1dP^{\epsilon, \otimes 2}(z,v_1,v_2,y_1,s,u_1,u_2,y_2).
\end{aligned}\end{equation}
Thus, as $\epsilon\to 0$, we expect that the latter converges in distribution to
\begin{equation}\label{nonlocal}
\frac{c_H}{\Gamma(H-\frac{1}{2})}\int_{[0,t]\times\mathcal{Y} }\int_{[0,s]\times\mathcal{U}_1}s^{H-\frac{1}{2}}z^{\frac{1}{2}-H}(s-z)^{H-\frac{3}{2}} \sigma_1(\psi_s,y_2)v_1dP(z,v_1,v_2,y_1)dP(s,u_1,u_2,y_2),
\end{equation}
as we will show in Proposition \ref{sigmaave} below. 

\begin{rem}\label{rem:occupationDomainConvention} Notice that in \eqref{sigma1comp}, \eqref{nonlocal} we have omitted the domains of integration for the variables that do not appear on the integrand of the occupation measures $P^\epsilon$ and limiting measures $P.$ This slight abuse of notation is made for the sake of lighter notation and will be used in several places throughout the rest of this article.
\end{rem}

The arguments above conclude the analysis of the limiting dynamics of $X^{\epsilon,\eta,u}$ and motivate the following definition of viable pairs:

\begin{dfn}\label{viable} Let $T>0, x_0\in\mathcal{X}$, $\Lambda_1:\mathcal{X}\times\mathcal{U}_1\times\mathcal{U}_2\times\mathcal{Y}\rightarrow\mathcal{X}$, $\Lambda_2:  \mathcal{X}\times([0,T]\times\mathcal{U}_1\times\mathcal{U}_2\times\mathcal{Y})^2\rightarrow\mathcal{X}$ and set $\Lambda=(\Lambda_1,\Lambda_2)$.     A pair $(\psi,P)\in C([0,T];\mathcal{X})\times\mathscr{P}([0,T]\times\mathcal{U}_1\times\mathcal{U}_2\times\mathcal{Y})$ is called \textit{viable} with respect to $\Lambda$ if the following hold:
	(i) $P$ has a finite second moment i.e.
	$$\int_{[0,T]\times\mathcal{U}_1\times\mathcal{U}_2\times\mathcal{Y}}\big(|u|^2+|v|^2+|y|^2\big)dP(t,u,v,y)<\infty$$
	(ii) For all $ h\in C_b([0,T]\times\mathcal{U}_1\times\mathcal{U}_2\times\mathcal{Y})$ we have the decomposition
	\begin{equation}\label{Pdec}
	\int_{[0,T]\times\mathcal{U}_1\times\mathcal{U}_2\times\mathcal{Y}}hdP=\int_{0}^{T}\int_{\mathcal{Y}}\int_{\mathcal{U}_1\times\mathcal{U}_2}h(s,u, v,y)d\Theta(u,v|y,s)d\mu(y)ds,
	\end{equation}
	where $\Theta(\cdot|\cdot)$ is a stochastic kernel on $\mathcal{U}_1\times\mathcal{U}_2$ given $\mathcal{Y}\times[0,T]$ (see Appendix A.5 in \cite{dupuis2011weak} for stochastic kernels) and $\mu$ is the unique invariant measure corresponding to $\mathcal{L}$ \eqref{generator}.\\
	(iii) $\forall t\in[0,T]$
	\begin{equation}\label{limdynamics}
	\begin{aligned}
	\psi&(t)=x_0+\int_{[0,t]\times\mathcal{U}_1\times\mathcal{U}_2\times\mathcal{Y}}\bigg[\Lambda_1( \psi(s),u_2,v_2,y_2)\\&+\int_{[0,s]\times\mathcal{U}_1\times\mathcal{U}_2\times\mathcal{Y}}\Lambda_2( \psi(s),s,u_2,v_2,y_2,z,u_1,v_1,y_1)dP(z,u_1,v_1,y_1)\bigg]dP(s,u_2,v_2,y_2).
	\end{aligned}
	\end{equation}
	The set of viable pairs with respect to $\Lambda$ is denoted by $\mathscr{V}_{\Lambda,x_0}$ and for each $\psi\in C([0,T];\mathcal{X}),$ $\mathscr{V}_{\Lambda,x_0,\psi}$ denotes the $\psi$-section $\{ P\in\mathscr{P}([0,T]\times\mathcal{U}_1\times\mathcal{U}_2\times\mathcal{Y}) :(\psi, P)\in\mathscr{V}_{\Lambda,x_0}\}.$
\end{dfn}
\begin{rem} A notion of viable pairs has been used in the study of large deviations for the case $H=1/2,$ see e.g. \cite{dupuis2012large}, Definition 2.7. In order to extend the method to the case $H>1/2,$ we have modified the definition to account for the non-local term \eqref{nonlocal} that appears in the limiting dynamics. This term is captured by the term $\Lambda_2$ in \eqref{limdynamics}.
\end{rem}
\noindent With Definition \ref{viable} at hand, we are ready to state our main results. The next theorem provides a characterization of the limit points of the family $\{ (X^{\epsilon,\eta,u}, P^\epsilon); \epsilon\ll1, u\in\mathcal{A}_N      \},$ in the sense of convergence in distribution. The reason why we can restrict the analysis to the smaller class of controls $\mathcal{A}_N$ is explained in Section \ref{lapluppersec} below.

\begin{thm}\label{main1} Let $T>0, 
	u\in\mathcal{A}_N$ as in \eqref{AN} and $(X^{\epsilon,\eta,u}, Y^{\epsilon,\eta,u})$ solve the controlled system \eqref{consys} with initial conditions $(x_0,y_0)\in\mathcal{X}\times\mathcal{Y}$. Moreover, let $\Lambda_1:\mathcal{X}\times\mathcal{U}_1\times\mathcal{U}_2\times\mathcal{Y}\rightarrow\mathcal{X}$, $\Lambda_2:  \mathcal{X}\times([0,T]\times\mathcal{U}_1\times\mathcal{U}_2\times\mathcal{Y})^2\rightarrow\mathcal{X}$ with 
	\begin{equation}\label{Lambda1}
	\Lambda_1(x,u_1,u_2,y)=c\big(x, y\big)+\nabla\Psi(y)g\big(x, y\big)+\big[\nabla\Psi(y)\tau(y)+\sigma_2(x,y)\big]u_2,
	\end{equation}
	\begin{equation}\label{Lambda2}
	\Lambda_2(x,t,u_1,v_1,y_1, s,u_2,v_2,y_2 )= \frac{c_H}{\Gamma(H-\frac{1}{2})}t^{H-\frac{1}{2}}s^{\frac{1}{2}-H}(t-s)^{H-\frac{3}{2}} \sigma_1(x,y_2)u_1,
	\end{equation}
	where $\Psi$ is the unique strong solution of \eqref{Poisson}, $c_H$ as in \eqref{cH} and set $\Lambda=(\Lambda_1,\Lambda_2).$ Under Conditions \ref{C1}-\ref{C7} and assuming that $\sigma_1$ is bounded, there exists a sufficiently small $\epsilon_0>0$ such that the family of processes  $\{ X^{\epsilon,\eta, u} ;\epsilon<\epsilon_0, u\in\mathcal{A}_N  \}$ is tight in $C([0,T];\mathcal{X})$ and the family $\{ P^\epsilon;  \epsilon<\epsilon_0 \}$ of occupation measures \eqref{occupation} is tight in $\mathscr{P}([0,T]\times\mathcal{U}_1\times\mathcal{U}_2\times\mathcal{Y}).$ Then for any sequence in $\{(X^{\epsilon,\eta, u}, P^{\epsilon} ),\;\epsilon<\epsilon_0, u\in\mathcal{A}_N\}$ there exists a subsequence that converges in distribution with limit $(\psi,P)$. With probability $1$, $(\psi,P)\in\mathscr{V}_{\Lambda,x_0}$, per Definition \ref{viable}.
\end{thm}

\begin{rem}\label{ergodicproprem} The importance of \eqref{Pdec} lies in the fact that the $y$-marginal of the measure $P$ does not depend on the variables $u_1,u_2$ and is given by the invariant measure $\mu.$ This is a consequence of the asymptotic regime \eqref{regime} which guarantees that the ergodic properties of $Y^{\epsilon,\eta}$ are preserved under perturbations by the stochastic control $u_2$ (see \eqref{consys}).
\end{rem}

\noindent The LDP for the slow motion $X^{\epsilon,\eta}$ is stated in the following theorem:
\begin{thm}\label{main2} Let $T>0, \Lambda $ as in Theorem \ref{main1}, $\mathscr{V}_{\Lambda,x_0,\phi}$ as in Definition \ref{viable}  and $(X^{\epsilon,\eta}, Y^{\epsilon,\eta})$ be the unique strong solution of \eqref{model} with initial conditions $(x_0,y_0)\in\mathcal{X}\times\mathcal{Y}$. Define a functional $S_{x_0}:C([0,T];\mathcal{X})\rightarrow[0,\infty]$ by
	\begin{equation}\label{ratefun}
	S_{x_0}(\phi)= \inf_{P\in\mathscr{V}_{\Lambda,x_0,\phi}}\frac{1}{2}\int_{[0,T]\times\mathcal{U}_1\times\mathcal{U}_2\times\mathcal{Y}}\big[|u_1|^2+|u_2|^2\big]dP(t,u_1,u_2,y)
	\end{equation}
	with the convention that $\inf\varnothing=\infty.$ Under Conditions \ref{C1}-\ref{C8} and assuming that $\sigma_1$ is bounded we have, for all bounded, continuous $h:C([0,T];\mathcal{X})\rightarrow\R,$
	\begin{equation*}
	\lim_{\epsilon\to 0}\epsilon\log\ex_{x_0}\big[ e^{-h(X^{\epsilon,\eta})/\epsilon}      \big]= -\inf_{\phi\in C([0,T];\R^m) }\big[ S_{x_0}(\phi) +h(\phi) \big],
	\end{equation*}
	where $\ex_{x_0}$ denotes that the initial condition is given by $x_0.$ In particular, $\{X^{\epsilon,\eta}\}$ satisfies a Large Deviation Principle in $C([0,T];\mathcal{X})$ with rate function $S_{x_0}$.
\end{thm}

\noindent Theorems \ref{main1}, \ref{main2} are proved in Sections \ref{limsec}, \ref{LDPsec} respectively. As we discuss in Section \ref{comparisonsec}, in certain cases, the rate function $S_{x_0}$ admits the explicit form
\begin{equation}\label{nonvarrateRev}
\begin{aligned}
S_{x_0}(\phi)&\equiv S^{H}_{x_0}(\phi)\\&=\frac{1}{2}\int_{0}^{T}\bigg\langle \dot{\phi}_t-\bar{c}(\phi_t)-\overline{\nabla\Psi g}(\phi_t), [\mathcal{Q}_H(\phi)\mathcal{Q}_H(\phi)^*]^{-1}\big[ \dot{\phi}-\bar{c}(\phi)-\overline{\nabla\Psi g}(\phi)      \big](t)\bigg\rangle_{\mathcal{X}}dt,
\end{aligned}
\end{equation}
for all $\phi\in C([0,T];\mathcal{X})$ such that $\phi_0=x_0$ and $\dot{\phi}-\bar{c}(\phi)-\overline{\nabla\Psi g}(\phi)\in L^2([0,T];\mathcal{X}),$ and $S_{x_0}(\phi)=\infty$ otherwise. For each $H\in(1/2,1)$ and $\phi\in C([0,T];\mathcal{X}),$  $\mathcal{Q}_H(\phi)\mathcal{Q}_H(\phi)^*:L^2([0,T];\mathcal{X})\rightarrow L^2([0,T];\mathcal{X})$ and $\mathcal{Q}_H(\phi):L^2([0,T];\mathcal{U}_1)\oplus L^2([0,T]\times\mathcal{Y}, dt\otimes d\mu;\mathcal{U}_2)\rightarrow L^2([0,T];\mathcal{X})$ denotes the "effective diffusivity" operator
\begin{equation}\label{QH1}
\begin{aligned}
\mathcal{Q}_H(\phi)[(u_1,u_2)](t)&:=\bar\sigma_1(\phi_t)\dot{K}_Hu_1(t)+\int_{\mathcal{Y}}Q(\phi_t,y)u_2(t,y)d\mu(y)\;,\;\; t\in[0,T],
\end{aligned}		
\end{equation}
with $\dot{K}_H$ as in \eqref{Kdot} and 
\begin{equation}\label{Qdef1}
\mathcal{X}\times\mathcal{Y}\ni (x,y)\mapsto Q(x,y):=[\nabla\Psi(y)\tau(y)+\sigma_2(x,y)]\in\mathscr{L}(\mathcal{U}_2;\mathcal{X}).
\end{equation}

\begin{rem}
	In Section \ref{LPlbsec} we prove the Laplace Principle lower bound in two cases with different degrees of generality. First we consider the case $b=\sigma_2=0, \bar{\sigma}_1\neq 0$ which essentially reduces to the setting of a small-noise LDP for a fractional SDE with an averaged diffusion coefficient. Then we prove the lower bound for the full model \eqref{model}, in the case where the averaged matrix-valued function $$\overline{QQ^T}(x)=\int_{\mathcal{Y}}QQ^T(x,y)d\mu(y)$$ is uniformly non-degenerate. For more details on this condition see Section \ref{assumptionsection} below.
\end{rem}

The  term $\bar{\sigma}_1$ on the right-hand side of \eqref{QH1} reflects the na\"ive averaging of the diffusion coefficient proved in \cite{hairer2020averaging}, Theorem A (see also Remark 1.1 therein). The latter is very different from the case $H=1/2$ and is mainly attributed to the pathwise interpretation of the integral $\sigma_1dB^H.$ Roughly speaking, when $H>1/2$ the corresponding equations behave more like ODEs, while in the case $H=1/2$ one also has to average the quadratic variation of $B^{1/2}$ (see \cite{brehier2021averaging} for an explanation based on the quadratic variation of Brownian motion). 

From the perspective of Large Deviations, a closely related fact is that the LDP rate function $S^H_{x_0}$ turns out to be discontinuous at $H=1/2$ when $\sigma_1=\sigma_1(x,y).$ A rigorous proof of this is given in Section \ref{comparisonsec}, Proposition \ref{Hlimprop} below. The latter means that, in general, the two cases are rather different in terms of both typical dynamics and tail behavior.

\section{Tightness estimates for the controlled slow dynamics}\label{tightsec}
\noindent In this section we prove estimates for the process $X^{\epsilon,\eta,u}$ that are uniform over small values of $\epsilon$ and controls $u\in\mathcal{A}_N.$  We work in the topology of the fractional Sobolev space $W^{a,\infty}_{0}$ \eqref{W0norm} which, in view of Definition \ref{Young}, comes as a natural choice. The estimates proved here will then be used in Section \ref{limsec} to show that the family $\{(X^{\epsilon,\eta,u}, P^\epsilon) ;\epsilon, u\}$ of processes and occupation measures is tight. The following is the main result of this section:

\begin{prop}\label{tightnessprop}
	Let $T>0$ and $\mathcal{A}_N$ as in \eqref{AN}. Under Conditions \ref{C1}-\ref{C7} the following hold:\\ \textbf{(i)} Under Condition \ref{C6}$(i)$ and with $\nu_2$ as in \eqref{sigmagrowth1}  we have that, for all $a\in(1-H,\frac{1}{4}\wedge\frac{\nu_2}{2}),$ there exists $\epsilon_0>0$
	such that
	\begin{equation*}
	\sup_{\epsilon<\epsilon_0, u\in\mathcal{A}_N}\ex\big\|X^{\epsilon,\eta,u}\big\|_{0,a,\infty}<\infty.
	\end{equation*}
	In view of \eqref{pmembeds} the latter implies that for any $\theta<a$
	\begin{equation}\label{tightnessX}
	\sup_{\epsilon<\epsilon_0, u\in\mathcal{A}_N}\ex\big\|X^{\epsilon,\eta,u}\big\|_{C^\theta([0,T];\mathcal{X})}<\infty.
	\end{equation}
	\textbf{(ii)} Under Condition \ref{C6}(ii), the conclusion of (i) holds for all $a\in(1-H,\frac{1}{2})$. 
\end{prop}

\noindent The proof of Proposition \ref{tightnessprop} is preceded by several auxiliary lemmas and is deferred to the end of this section. In order to avoid repetition, we shall only provide the complete proof for the first part of Proposition \ref{tightnessprop}. The second part treats the case $\sigma_1(x,y)=\sigma_1(x),$ follows by similar arguments and its proof is in fact simpler. For the sake of completeness, we will provide comments on the differences between the two cases when necessary. 

\begin{rem}\label{Yintrem}
	One of the challenges in obtaining the estimates that follow lies in that the weak derivatives $\dot{u}_i, i=1,2$  are only known to be square integrable. This, along with the recurrence condition (Condition \ref{C3}), only allows us to obtain uniform estimates for $Y^{\epsilon,\eta,u}$ in $L^p(\Omega, L^2([0,T];\mathcal{Y})), p\geq 1$, see Lemma \ref{Ybndlem}$(i)$.
\end{rem}
Our strategy for proving \eqref{Xbnd} is based on the following lemma, which
provides pathwise estimates for the Young integral $\sigma_1dB^H$ via the integration-by-parts formula \eqref{ibpformula}.
\begin{lem}\label{Youngbnd} Let $a\in(1-H, \frac{1}{2})$, $T>0$, $0\leq s<t\leq T$, $|\Delta_a|$ as in \eqref{Deltapmabs}, $\nu_1$ as in \eqref{sigmagrowth1} 
	and $L_{\sigma_1}$ denote the Lipschitz constant of $\sigma_1$. Under Condition \ref{C6} the following hold: \\
	(i) There exists a constant $C>0$ that depends on $a,L_{\sigma_1}, \nu_2,T$ such that, with probability $1,$ 
	\begin{equation}\label{preYoung}
	\begin{aligned}
	\bigg|\int_{s}^{t}\sigma_1&\big(X_r^{\epsilon,\eta,u}, Y_r^{\epsilon,\eta,u}\big)dB^H_r\bigg|\\&\leq C\|B^H\|_{T,1-a,\infty}\bigg[(t-s)^{\frac{1}{2}-a}\bigg(1+\sup_{r\in[0,t]}\big| X_r^{\epsilon,\eta,u}\big|^{\nu_1}+\big\|Y^{\epsilon,\eta,u}\big\|_{L^2([0,T];\mathcal{Y})}     \bigg)\\&
	+\int_{s}^{t}|\Delta_a|X_{s,r}^{\epsilon,\eta,u}dr+\int_{s}^{t}|\Delta_a|Y_{s,r}^{\epsilon,\eta,u}dr \bigg].
	\end{aligned}
	\end{equation}
	(ii) There exists a constant $C>0$ such that, with probability 1,
	\begin{equation*}
	\begin{aligned}
	|\Delta_a&|\bigg(\int_{0}^{\cdot}  \sigma_1\big(X_r^{\epsilon,\eta,u}, Y_r^{\epsilon,\eta,u}\big)dB^H_r\bigg)_{0,t}\\&\leq  C\|B^H\|_{T,1-a,\infty}\bigg[  1+\sup_{r\in[0,t]}\big| X_r^{\epsilon,\eta,u}\big|^{\nu_1}+\big\|Y^{\epsilon,\eta,u}\big\|_{L^2([0,T];\mathcal{Y})}+    \int_{0}^{t}   (t-r)^{-a}|\Delta_a|X_{0,r}^{\epsilon,\eta,u}dr\\&+\int_{0}^{t}(t-s)^{-a-1}\int_{s}^{t}|\Delta_a|Y_{s,r}^{\epsilon,\eta,u}dr ds  \bigg].
	\end{aligned}
	\end{equation*}      	  		
\end{lem}
\begin{proof} We shall only prove the estimates under Condition \ref{C6}$(i).$   It is straightforward to verify that similar estimates hold under Condition \ref{C6}$(ii)$ and their proof is in fact simpler since $\sigma_1$ only depends on $x.$ In this case, the right-hand side of the estimates no longer depends on $Y^{\epsilon,\eta,u}$.\\
	(i) In view of Section \ref{fbmsec}, we can fix a version of $B^H$ with paths in $W^{1-a,\infty}_{T}[0,T].$ For this version, \eqref{Weyl-} and \eqref{-norm} yield the estimate 
	\begin{equation*}
	\sup_{0\leq r<t\leq T}\big|D^{1-a}_{t^-}B^H_{t^-}(r)\big|\leq C_a\|B^H\|_{T,1-a,\infty}\;,
	\end{equation*}
	where the positive constant $C_a$ does not depend on $t.$ Recalling \eqref{ibpformula} and using the Lipschitz continuity and growth of $\sigma_1$ the latter implies
	\begin{equation*}
	\begin{aligned}
	\bigg|&\int_{s}^{t} D^a_{s^+}\sigma_1(X_r^{\epsilon,\eta,u},Y^{\epsilon,\eta,u}_r)D^{1-a}_{t^-}B^H_{t^-}(r)dr\bigg|\\&\leq \sup_{0\leq r<t\leq T}\big|D^{1-a}_{t^-}B^H_{t^-}(r)\big|\int_{s}^{t} \big|D^a_{s^+}\sigma_1(X_r^{\epsilon,\eta,u},Y^{\epsilon,\eta,u}_r)\big|dr
	\\& \leq C_{a}\|B^H\|_{T,1-a,\infty}\int_{s}^{t}\bigg[ \frac{\big|\sigma_1(X_r^{\epsilon,\eta,u},Y^{\epsilon,\eta,u}_r)\big|}{(r-s)^a} \\&+\int_{s}^{r}\frac{\big|\sigma_1(X_r^{\epsilon,\eta,u},Y^{\epsilon,\eta,u}_r)-\sigma_1(X_q^{\epsilon,\eta,u},Y^{\epsilon,\eta,u}_q)\big|}{(r-q)^{a+1}} dq \bigg]dr\\&
	\leq C_{a,\sigma_1}(K_2+L_{\sigma_1})\|B^H\|_{T,1-a,\infty}\bigg[\int_{s}^{t} (r-s)^{-a}\big(1+\big|X_r^{\epsilon,\eta,u}|^{\nu_1}+ \big|Y_r^{\epsilon,\eta,u}\big|^{\nu_2}\big)dr\\&+\int_{s}^{t}\int_{s}^{r}\frac{\big|X_r^{\epsilon,\eta,u}-X_q^{\epsilon,\eta,u}\big|+\big|Y^{\epsilon,\eta,u}_r-Y^{\epsilon,\eta,u}_q\big|}{(r-q)^{a+1}}dqdr\bigg]\\&
	\leq C_{a,\sigma_1}\|B^H\|_{T,1-a,\infty}\bigg[ (t-s)^{1-a}\bigg(1+\sup_{r\in[0,t]}\big| X_r^{\epsilon,\eta,u}\big|^{\nu_1}\bigg)+(t-s)^{\frac{1}{2}-a} \bigg(\int_{0}^{t}\big|Y^{\epsilon,\eta,u}_r\big|^{2\nu_2}dr\bigg)^{\frac{1}{2}}\\&+ \int_{s}^{t}|\Delta_a|X_{s,r}^{\epsilon,\eta,u}dr+\int_{s}^{t}|\Delta_a|Y_{s,r}^{\epsilon,\eta,u}dr  \bigg],
	\end{aligned}
	\end{equation*}
	where the last line follows from the Cauchy-Schwarz inequality and the definition of $|\Delta_a|$. The estimate follows by noting that $2\nu_2<1$ and applying Young's inequality for products.

	\noindent$ (ii)$ From \eqref{preYoung} we have
	\begin{equation*}\label{young1}
	\begin{aligned}
	|\Delta_a|\bigg(\int_{0}^{\cdot} & \sigma_1\big(X_r^{\epsilon,\eta,u}, Y_r^{\epsilon,\eta,u}\big)dB^H_r\bigg)_{0,t}= \int_{0}^{t}(t-s)^{-a-1}\bigg|\int_{s}^{t}\sigma_1\big(X_r^{\epsilon,\eta,u}, Y_r^{\epsilon,\eta,u}\big)dB^H_r\bigg|ds\\&
	\leq C\|B^H\|_{T,1-a,\infty}\bigg[\bigg(1+ \sup_{r\in[0,t]}\big| X_r^{\epsilon,\eta,u}\big|^{\nu_1}+\big\|Y^{\epsilon,\eta,u}\big\|_{L^2([0,T];\mathcal{Y})}\bigg)\int_{0}^{t}(t-s)^{-\frac{1}{2}-a}ds\\&+\int_{0}^{t}(t-s)^{-a-1}\int_{s}^{t}|\Delta_a|Y_{s,r}^{\epsilon,\eta,u}dr ds +    \int_{0}^{t}(t-s)^{-a-1}\int_{s}^{t}|\Delta_a|X_{s,r}^{\epsilon,\eta,u}drds \bigg].
	\end{aligned}
	\end{equation*}
	Since $a<1/2,$ the singularity in the first integral of the last display is integrable. Finally, an application of Fubini's theorem on the last term yields
	\begin{equation*}
	\begin{aligned}
	\int_{0}^{t}(t-s)^{-a-1}\int_{s}^{t}|\Delta_a|X_{s,r}^{\epsilon,\eta,u}drds&=\int_{0}^{t}\int_{s}^{t}\int_{s}^{r}(t-s)^{-a-1}\frac{\big|X_{r}^{\epsilon,\eta,u}-X_{q}^{\epsilon,\eta,u}\big|}{(r-q)^{a+1}}dqdrds\\&= \int_{0}^{t}   \int_{0}^{r} \bigg(\int_{0}^{q}(t-s)^{-a-1}    ds\bigg)\frac{\big| X_r^{\epsilon,\eta,u}-X_q^{\epsilon,\eta,u} \big|    }{(r-q)^{a+1}}dqdr\\&
	\leq \int_{0}^{t}   \int_{0}^{r} \bigg(\int_{0}^{r}(t-s)^{-a-1}    ds\bigg)\frac{\big| X_r^{\epsilon,\eta,u}-X_q^{\epsilon,\eta,u} \big|    }{(r-q)^{a+1}}dqdr
	\\&
	=\frac{1}{a}\int_{0}^{t}   \int_{0}^{r} \big[(t-r)^{-a}-t^{-a}\big]\frac{\big| X_r^{\epsilon,\eta,u}-X_q^{\epsilon,\eta,u} \big|    }{(r-q)^{a+1}}dqdr\\&\leq 
	\frac{1}{a}\int_{0}^{t}   (t-r)^{-a}\int_{0}^{r}\frac{\big| X_r^{\epsilon,\eta,u}-X_q^{\epsilon,\eta,u} \big|    }{(r-q)^{a+1}}dqdr\\&
	=\frac{1}{a}\int_{0}^{t}   (t-r)^{-a}|\Delta_a|X_{0,r}^{\epsilon,\eta,u}dr.
	\end{aligned}
	\end{equation*}
	\noindent The proof is complete.
\end{proof}
\begin{rem} In view of \eqref{preYoung} we see that, in order to estimate the Young integral, one has to control both the sup-norm $\|X^{\epsilon,\eta,u}\|_{C([0,T];\mathcal{X})},$ and the integrated difference ratio $|\Delta_a|X^{\epsilon,\eta,u}_{0,t}$ of some order $a.$ This is our goal for the rest of this section, as well as the reason for working in the space $W^{a,\infty}_0.$
	
\end{rem}
\noindent The following lemma collects some preliminary estimates for the Stieltjes integrals $du_i, i=1,2.$ 
\begin{lem}\label{u1bnd} Let $|\Delta_a|$ as in \eqref{Deltapmabs}, $u=(u_1,u_2)\in\mathcal{A}_N$ \eqref{AN}, $\nu_1,\nu_2$ as in \eqref{sigmagrowth1}, $T>0$ and $0<a<\frac{\nu_2}{2}.$ There exists a constant $C>0$ such that for $i=1,2$ the following hold:
	\begin{equation}\label{supcon}
	(i)\sup_{t\in[0,T]}\bigg|\int_{0}^{t}\sigma_i\big(X_r^{\epsilon,\eta,u}, Y_r^{\epsilon,\eta,u}\big)du_i(r)\bigg|\leq  C\bigg(1+\int_{0}^{T}\sup_{s\in[0,r]}\big|X_s^{\epsilon,\eta,u}\big|dr+ \big\|Y^{\epsilon,\eta,u}\big\|_{L^2([0,T];\mathcal{Y})}\bigg).
	\end{equation}
	\begin{equation*}
	(ii)|\Delta_a|\bigg(\int_{0}^{\cdot}\sigma_i\big(X_r^{\epsilon,\eta,u}, Y_r^{\epsilon,\eta,u}\big)du_i(r)\bigg)_{0,t}\leq C\bigg( 1+\sup_{t\in[0,T]}\big|X_t^{\epsilon,\eta,u}\big|^{2\nu_1}+\|Y^{\epsilon,\eta,u}\|_{L^2([0,T];\mathcal{Y})}\bigg).
	\end{equation*}
\end{lem}

\begin{proof}$ (i)$ Since $u=(u_1,u_2)\in\h_H\oplus\h_{1/2},$ it has $\pr$-almost surely classically differentiable paths. Thus, for $i=1,2,$ $du_i$ is in fact a classical Stieltjes integral. Starting with $i=1$ and using \eqref{sigmagrowth1}, Lemma \ref{udotL2} and the Cauchy-Schwarz inequality we obtain
	\begin{equation*}
	\begin{aligned}
	\bigg|\int_{0}^{t}\sigma_1\big(X_r^{\epsilon,\eta,u}, Y_r^{\epsilon,\eta,u}\big)du_1(r)\bigg|&\leq K_2\int_{0}^{t}\big(1+\big|X_r^{\epsilon,\eta,u}\big|^{\nu_1}+ \big|Y_r^{\epsilon,\eta,u}\big|^{\nu_2}\big)\big|\dot{u}_1(r)\big|dr\\&
	\leq C_{\nu_1, \nu_2}\bigg(T+\int_{0}^{t}\big|X_r^{\epsilon,\eta,u}\big|^{2\nu_1}dr+ \int_{0}^{t}\big|Y_r^{\epsilon,\eta,u}\big|^{2\nu_2}dr\bigg)^{\frac{1}{2}}\|\dot{u}_1\|_{L^2}\\&
	\leq C_{\nu_1,\nu_2}\bigg(1+\int_{0}^{t}\big|X_r^{\epsilon,\eta,u}\big|dr+ \big\|Y^{\epsilon,\eta,u}\big\|_{L^2([0,T];\mathcal{Y})}\bigg),
	\end{aligned}
	\end{equation*}
	where we used that $\nu_2,\nu_1\leq 1/2$ and Lemma \ref{udotL2} to deduce that  $\|\dot{u}_1\|_{L^2([0,T];\mathcal{U}_1)}\leq C\|u_1\|_{\h_H}\leq C N$ with probability $1.$ For $i=2,$ the estimate for $\sigma_2du_2$ follows similarly and is in fact simpler since $\sigma_2$ is uniformly bounded (see Condition \ref{C4}).\\
	$(ii)$ Let $0\leq s<t\leq T$.  In view of \eqref{sigmagrowth1}, Lemma \ref{udotL2}, Fubini's theorem and the Cauchy-Schwarz inequality we have 
	\begin{equation*}
	\begin{aligned}
	\int_{0}^{t}&(t-s)^{-a-1}\bigg|\int_{s}^{t}\sigma_1\big(X_r^{\epsilon,\eta,u}, Y_r^{\epsilon,\eta,u}\big)du_1(r)\bigg|ds\\&\leq 	\int_{0}^{t}(t-s)^{-a-1}\int_{s}^{t}\big|\sigma_1\big(X_r^{\epsilon,\eta,u}, Y_r^{\epsilon,\eta,u}\big)\dot{u}_1(r)\big|drds\\&
	\leq K_2\int_{0}^{t}(t-s)^{-a-1}\int_{s}^{t}\big(1+\big|X_r^{\epsilon,\eta,u}\big|^{\nu_1}+ \big|Y_r^{\epsilon,\eta,u}\big|^{\nu_2}\big)\big|\dot{u}_1(r)\big|drds\\&	=K_2\int_{0}^{t}\big(1+\big|X_r^{\epsilon,\eta,u}\big|^{\nu_1}+ \big|Y_r^{\epsilon,\eta,u}\big|^{\nu_2}\big)|\dot{u}_1(r)|\int_{0}^{r}(t-s)^{-a-1}dsdr\\&
	\leq \frac{K_2}{a}\int_{0}^{t}(t-r)^{-a}\big(1+\big|X_r^{\epsilon,\eta,u}\big|^{\nu_1}+ \big|Y_r^{\epsilon,\eta,u}\big|^{\nu_2}\big)|\dot{u}_1(r)|dr\\&
	\leq C_{a}\|\dot{u}_1\|_{L^2}\bigg[T^{1-2a}\bigg(1+\sup_{r\in[0,T]}\big|X_r^{\epsilon,\eta,u}\big|^{2\nu_1}\bigg)+ \bigg(\int_{0}^{t}(t-r)^{-2a}\big|Y_r^{\epsilon,\eta,u}\big|^{2\nu_2}dr\bigg)\bigg]^{\frac{1}{2}}\\&
	\leq  C_{\nu_2,a,T}\bigg[1+\sup_{t\in[0,T]}\big|X_t^{\epsilon,\eta,u}\big|^{2\nu_1}+\bigg(\int_{0}^{t}(t-r)^{-a/\nu_2} \big|Y_r^{\epsilon,\eta,u}\big|dr\bigg)\bigg],
	\end{aligned} 
	\end{equation*}
	where we applied Young's inequality for products with exponent $\tfrac{1}{2\nu_2}\geq 1$ to obtain the last line. Yet another application of the Cauchy-Schwarz inequality yields
	\begin{equation*}
	\begin{aligned}
	|\Delta_a|\bigg(\int_{0}^{\cdot}\sigma_i&\big(X_r^{\epsilon,\eta,u}, Y_r^{\epsilon,\eta,u}\big)du_1(r)\bigg)_{0,t}\\&
	\leq  C\bigg[ 1+\sup_{t\in[0,T]}\big|X_t^{\epsilon,\eta,u}\big|^{2\nu_1}+\bigg(\int_{0}^{t}(t-r)^{-2a/\nu_2}dr\bigg)^{\frac{1}{2}}\|Y^{\epsilon,\eta,u}\|_{L^2([0,T];\mathcal{Y})}\bigg]%\\&
	\end{aligned} 
	\end{equation*}
	The latter concludes the estimate since $a<\nu_2/2$ implies that the singularity is integrable.
	The estimate for $\sigma_2du_2$ follows from an identical argument and is in fact simpler due to the uniform boundedness of $\sigma_2.$
\end{proof}

\begin{rem} The restriction $a<\nu_2/2$ in Lemma \ref{u1bnd} is related to the integrability properties of the controlled fast process $Y^{\epsilon,\eta,u}$ (see Remark \ref{Yintrem} above) and is only necessary in the case $\sigma_1=\sigma_1(x,y)$ i.e. under Condition \ref{C6}(i).
\end{rem}

\begin{lem}\label{cbnd} Let $|\Delta_a|$ as in \eqref{Deltapmabs}, $a\in(1-H, \frac{1}{2})$, $T>0$. There exists a constant $C>0$ such that the following hold with probability $1$
	\begin{equation}\label{supc}
	(i) \sup_{t\in[0,T]}\bigg|\int_{0}^{t}  c\big(X_r^{\epsilon,\eta,u}, Y_r^{\epsilon,\eta,u}\big)dr\bigg|\leq C\bigg(  1+  \int_{0}^{T}\sup_{s\in[0,r]}\big|X_s^{\epsilon,\eta,u}    \big| dr+\big\|Y^{\epsilon,\eta,u}\big\|_{L^2([0,T];\mathcal{Y})}     \bigg),
	\end{equation}   
	\begin{equation*}
	(ii)\big|\Delta_a\big|\bigg(\int_{0}^{\cdot}  c\big(X_r^{\epsilon,\eta,u}, Y_r^{\epsilon,\eta,u}\big)dr\bigg)_{ 0,t}\leq C\bigg(  1+  \int_{0}^{t}\frac{\sup_{s\in[0,r]}\big|X_s^{\epsilon,\eta,u} \big|^{\nu}}{(t-r)^{a}}   dr+\big\|Y^{\epsilon,\eta,u}\big\|_{L^2([0,T];\mathcal{Y})}     \bigg).
	\end{equation*}        	
\end{lem}
\begin{proof} $(i) $ The inequality follows from Condition \ref{C5} along with the Cauchy-Schwarz and Young's product inequalities.
	$(ii)$	From Condition \ref{C5} and Fubini's theorem we have
	\begin{equation*}
	\begin{aligned}
	|\Delta_a\big|\bigg(&\int_{0}^{\cdot}  c\big(X_r^{\epsilon,\eta,u}, Y_r^{\epsilon,\eta,u}\big)dr\bigg)_{0,t}\leq\int_{0}^{t}(t-s)^{-a-1}\int_{s}^{t}\big|c\big(X_r^{\epsilon,\eta,u}, Y_r^{\epsilon,\eta,u}\big)\big|drds\\&
	\leq K_1\int_{0}^{t}(t-s)^{-a-1}\int_{s}^{t}\big(1+\big|X_r^{\epsilon,\eta,u}\big|^{\nu}+\big| Y_r^{\epsilon,\eta,u}\big|\big)drds\\&
	= K_1\int_{0}^{t}\big(1+\big|X_r^{\epsilon,\eta,u}\big|^{\nu}+\big| Y_r^{\epsilon,\eta,u}\big|\big)\int_{0}^{r}(t-s)^{-a-1}dsdr\\&
	\leq K_1\int_{0}^{t}(t-r)^{-a}\big(1+\sup_{s\in[0,r]}\big|X_s^{\epsilon,\eta,u}\big|^{\nu}+\big| Y_r^{\epsilon,\eta,u}\big|\big)dr\\&     		
	\leq  C_T\bigg[1+\int_{0}^{t}(t-r)^{-a}\sup_{s\in[0,r]}\big|X_s^{\epsilon,\eta,u}    \big|^{\nu} dr+\big\|Y^{\epsilon,\eta,u}\big\|_{L^2([0,T];\mathcal{Y})}\bigg(\int_{0}^{t}(t-r)^{-2a}dr\bigg)^\frac{1}{2}\bigg],
	\end{aligned}
	\end{equation*}  
	where we applied the Cauchy-Schwarz inequality to obtain the last line. The estimate follows.
\end{proof}
\noindent Next, we consider the term of order $\sqrt{\epsilon/\eta}$ in the slow dynamics. To this end let $0\leq s<t\leq T$ and $\Psi$ be the unique strong solution of the Poisson equation \eqref{Poisson}. An application of It\^o's formula yields
\begin{equation*}
\begin{aligned}
\Psi(Y^{\epsilon,\eta,u}_t)-\Psi(Y^{\epsilon,\eta,u}_s)&=\int_{s}^{t}\nabla\Psi(Y^{\epsilon,\eta,u}_r)dY^{\epsilon,\eta,u}_r+\frac{1}{2}\int_{s}^{t}D^2\Psi(dY^{\epsilon,\eta,u}_r):d\langle Y^{\epsilon,\eta,u}\rangle_r\\&
=\frac{1}{\sqrt{\epsilon\eta}}\int_{s}^{t}\nabla\Psi(Y^{\epsilon,\eta,u}_r)\big[g(X^{\epsilon,\eta,u}_r,Y^{\epsilon,\eta,u}_r)+\tau(Y^{\epsilon,\eta,u}_r)\dot{u}_2(r)\big]dr\\&+\frac{1}{\eta}\int_{s}^{t}\mathcal{L}\Psi(Y^{\epsilon,\eta,u}_r)dr+\frac{1}{\sqrt{\eta}}\int_{s}^{t}\nabla\Psi(Y^{\epsilon,\eta,u}_r)\tau(Y^{\epsilon,\eta,u}_r)dW_r\\&=\frac{1}{\sqrt{\epsilon\eta}}\int_{s}^{t}\nabla\Psi(Y^{\epsilon,\eta,u}_r)\big[g(X^{\epsilon,\eta,u}_r,Y^{\epsilon,\eta,u}_r)+\tau(Y^{\epsilon,\eta,u}_r)\dot{u}_2(r)\big]dr\\&-\frac{1}{\eta}\int_{s}^{t}b\big(Y^{\epsilon,\eta,u}_r\big)dr+\frac{1}{\sqrt{\eta}}\int_{s}^{t}\nabla\Psi(Y^{\epsilon,\eta,u}_r)\tau(Y^{\epsilon,\eta,u}_r)dW_r.
\end{aligned}	
\end{equation*}
Therefore, 
\begin{equation}\label{bterm}
\begin{aligned}
\frac{\sqrt{\epsilon}}{\sqrt\eta}\int_{s}^{t}b(Y^{\epsilon,\eta,u}_r)dr=&-\sqrt{\epsilon\eta}[\Psi(Y^{\epsilon,\eta,u}_t)-\Psi(Y^{\epsilon,\eta,u}_s)]+\sqrt{\epsilon}\int_{s}^{t}\nabla\Psi(Y^{\epsilon,\eta,u}_r)\tau(Y^{\epsilon,\eta,u}_r)dW_r\\&+
\int_{s}^{t}\nabla\Psi(Y^{\epsilon,\eta,u}_r)\big[g(X^{\epsilon,\eta,u}_r,Y^{\epsilon,\eta,u}_r)+\tau(Y^{\epsilon,\eta,u}_r)\dot{u}_2(r)\big]dr.
\end{aligned}	
\end{equation}

\noindent Some preliminary estimates for this term are collected in the next lemma.

\begin{lem}\label{bbnd} Let $|\Delta_a|$ as in \eqref{Deltapmabs}, $a\in(1-H, \frac{1}{2})$, $\theta\in(a,\frac{1}{2})$, $T>0$. There exists a constant $C$ such that the following hold with probability $1;$
	\begin{equation}\label{supb}
\begin{aligned}
		(i)\; \sup_{t\in[0,T]}\bigg|\frac{\sqrt{\epsilon}}{\sqrt\eta}\int_{0}^{t}&b(Y^{\epsilon,\eta,u}_r)dr\bigg|\leq C\bigg[1+\|Y^{\epsilon,\eta,u}\|_{L^2([0,T];\mathcal{Y})}\\&+\sqrt{\epsilon\eta}\big[Y^{\epsilon,\eta,u}\big]_{C^\theta}+\sqrt{\epsilon}\sup_{t\in[0,T]}\bigg|\int_{0}^{t}\nabla\Psi(Y^{\epsilon,\eta,u}_r)\tau(Y^{\epsilon,\eta,u}_r)dW_r\bigg|\bigg],
	\end{aligned}	
	\end{equation} 
	\begin{equation*}
	 \begin{aligned}
	(ii)\;\big|\Delta_a\big|\bigg(\frac{\sqrt{\epsilon}}{\sqrt{\eta}}\int_{0}^{\cdot}  &b\big( Y_r^{\epsilon,\eta,u}\big)dr\bigg)_{0,t}\leq C\bigg[1+\|Y^{\epsilon,\eta,u}\|_{L^2([0,T];\mathcal{Y})}+\sqrt{\epsilon\eta}\big[Y^{\epsilon,\eta,u}\big]_{C^\theta}\\&+\sqrt{\epsilon}\big|\Delta_a\big|\bigg(\int_{0}^{\cdot}\nabla\Psi(Y^{\epsilon,\eta,u}_r)\tau(Y^{\epsilon,\eta,u}_r)dW_r\bigg)_{0,t}\bigg].
	\end{aligned}	
	\end{equation*} 
\end{lem}
\begin{proof}$ (i) $ From \eqref{bterm}, the mean value inequality and Conditions \ref{C1}, \ref{C2} we obtain
	\begin{equation}\label{4.4.1}
	\begin{aligned}
	\frac{\sqrt{\epsilon}}{\sqrt\eta}&\bigg|\int_{s}^{t}b(Y^{\epsilon,\eta,u}_r)dr\bigg|\leq \sqrt{\epsilon\eta}\big|\Psi(Y^{\epsilon,\eta,u}_t)-\Psi(Y^{\epsilon,\eta,u}_s)\big|+\sqrt{\epsilon}\bigg|\int_{s}^{t}\nabla\Psi(Y^{\epsilon,\eta,u}_r)\tau(Y^{\epsilon,\eta,u}_r)dW_r\bigg|\\&+
	\int_{s}^{t}\big|\nabla\Psi(Y^{\epsilon,\eta,u}_r)\tau(Y^{\epsilon,\eta,u}_r)\dot{u}_2(r)\big|dr
	+\int_{s}^{t}\big|\nabla\Psi(Y^{\epsilon,\eta,u}_r)g(X^{\epsilon,\eta,u}_r,Y^{\epsilon,\eta,u}_r)\big|dr\\&\leq C\sqrt{\epsilon\eta}\|\nabla\Psi\|_{\infty}[Y^{\epsilon,\eta,u}]_{C^\theta}(t-s)^\theta+\sqrt{\epsilon}\bigg|\int_{s}^{t}\nabla\Psi(Y^{\epsilon,\eta,u}_r)\tau(Y^{\epsilon,\eta,u}_r)dW_r\bigg|\\&+\|\nabla\Psi\tau\|_{\infty}\|\dot{u}_2\|_{L^2}(t-s)^{1/2}+c_g\|\nabla\Psi\|_{\infty}(t-s)^{1/2}\big(1+\big\|Y^{\epsilon,\eta,u}\big\|_{L^2}\big)
	\end{aligned}
	\end{equation}
	where we also used the Cauchy-Schwarz term for the last term of the right-hand side. We remark that, in view of Condition \ref{C7}, $\nabla\Psi$ is bounded.  The estimate follows by setting $s=0.$

	$(ii) $ 	Continuing from \eqref{4.4.1} we have 
	\begin{equation*}
	\begin{aligned}
	\big|\Delta_a\big|\bigg(\frac{\sqrt{\epsilon}}{\sqrt{\eta}}\int_{0}^{\cdot}  &b\big( Y_r^{\epsilon,\eta,u}\big)dr\bigg)_{0,t}\leq    C\sqrt{\epsilon\eta}\|\nabla\Psi\|_{\infty}[Y^{\epsilon,\eta,u}]_{C^\theta}\int_{0}^{t}(t-s)^{\theta-a-1}ds\\&+\sqrt{\epsilon}\int_{0}^{t}(t-s)^{-a-1}\bigg|\int_{s}^{t}\nabla\Psi(Y^{\epsilon,\eta,u}_r)\tau(Y^{\epsilon,\eta,u}_r)dW_r\bigg|ds\\&
	+C_{\Psi,\tau}\big(1+\|\dot{u}_2\|_{L^2}+\big\|Y^{\epsilon,\eta,u}\big\|_{L^2}\big)\int_{0}^{t}(t-s)^{-a-\frac{1}{2}}ds,
	\end{aligned}	
	\end{equation*}
	where the Riemann integrals are finite since $\theta\in(a,1/2).$ The proof is complete.	
\end{proof}

\noindent Combining the previous bounds, we obtain the following preliminary estimate for the integrated difference ratio $|\Delta_a|X^{\epsilon,\eta,u},$ which depends on the sup-norm of $X^{\epsilon,\eta,u}.$

\begin{lem}\label{DeltaX1} Let $|\Delta_a|$ as in \eqref{Deltapmabs}, $\nu,\nu_1, \nu_2$ as in Conditions \ref{C5}, \ref{C6} respectively, $T>0, a\in(1-H,\frac{1}{4}\wedge\frac{\nu_2}{2})$ and 
	\begin{equation}\label{Feps}
	F_\epsilon:=(1\vee\|B^H\|_{T,1-a,\infty})\exp\big( CT^{1-a}\sqrt{\epsilon}\big\| B^H\big\|_{T,1-a,\infty} \big),
	\end{equation}
	for some constant $C>0$ that does not  depend on $\epsilon$. Under Condition \ref{C6}(i), there exists a constant $C'>0$ that does not depend on $\epsilon$ such that for all $\epsilon<1$ we have
	\begin{equation}\label{integralprebnd}
	\begin{aligned}
	\sup_{t\in[0,T]}\big|\Delta_a\big|X^{\epsilon,\eta,u}_{0,t}\leq C' F_\epsilon&\bigg[1+\sup_{t\in[0,T]}\big|X_t^{\epsilon,\eta,u}\big|^{2\nu_1\vee\nu}+\big\|Y^{\epsilon,\eta,u}\big\|_{L^2([0,T];\mathcal{Y})}\\&+\sqrt{\epsilon\eta}\big[Y^{\epsilon,\eta,u}\big]_{C^\theta}+\sqrt{\epsilon}\sup_{t\in[0,T]}\int_{0}^{t}(t-s)^{-a-1}\int_{s}^{t}\big|\Delta_a\big|Y_{s,r}^{\epsilon,\eta,u}dr ds \\&+\sqrt{\epsilon}\sup_{t\in[0,T]}\big|\Delta_a\big|\bigg(\int_{0}^{\cdot}\nabla\Psi(Y^{\epsilon,\eta,u}_r)\tau(Y^{\epsilon,\eta,u}_r)dW_r\bigg)_{0,t}\\&+\sqrt{\epsilon}\sup_{t\in[0,T]}\big|\Delta_a\big|\bigg(\int_{0}^{\cdot}\sigma_2(X^{\epsilon,\eta,u}_r,Y^{\epsilon,\eta,u}_r)dW_r\bigg)_{0,t}\bigg],
	\end{aligned}
	\end{equation}
	with probability $1$.

\end{lem}
\begin{proof} 
	From Lemmas \ref{Youngbnd}$(ii),$  \ref{u1bnd}$(ii),$   \ref{cbnd}$(ii),$ \ref{bbnd}$(ii),$ we have 
	\begin{equation*}\label{Deltaxprebnd}
	\begin{aligned}
	&\big|\Delta_a\big|X^{\epsilon,\eta,u}_{0,t}\leq\\&\big|\Delta_a\big|\bigg(\frac{\sqrt{\epsilon}}{\sqrt{\eta}}\int_{0}^{\cdot}  b\big( Y_r^{\epsilon,\eta,u}\big)dr\bigg)_{0,t}+\big|\Delta_a\big|\bigg(\int_{0}^{\cdot}  c\big(X_r^{\epsilon,\eta,u}, Y_r^{\epsilon,\eta,u}\big)dr\bigg)_{0,t}\\&+\sum_{i=1}^{2}|\Delta_a|\bigg(\int_{0}^{\cdot}\sigma_i\big(X_r^{\epsilon,\eta,u}, Y_r^{\epsilon,\eta,u}\big)du_i(r)\bigg)_{0,t}\\&
	+\sqrt{\epsilon}|\Delta_a|\bigg(\int_{0}^{\cdot}  \sigma_1\big(X_r^{\epsilon,\eta,u}, Y_r^{\epsilon,\eta,u}\big)dB^H_r\bigg)_{0,t}+\sqrt{\epsilon}\big|\Delta_a\big|\bigg(\int_{0}^{\cdot}\sigma_2\big(X_r^{\epsilon,\eta,u}, Y_r^{\epsilon,\eta,u}\big)dW_r\bigg)_{0,t}\\&
	\lesssim 1+\|Y^{\epsilon,\eta,u}\|_{L^2([0,T];\mathcal{Y})}+\sqrt{\epsilon\eta}\big[Y^{\epsilon,\eta,u}\big]_{C^\theta}+\sqrt{\epsilon}\big|\Delta_a\big|\bigg(\int_{0}^{\cdot}\nabla\Psi(Y^{\epsilon,\eta,u}_r)\tau(Y^{\epsilon,\eta,u}_r)dW_r\bigg)_{0,t}\\&+\sqrt{\epsilon}\big|\Delta_a\big|\bigg(\int_{0}^{\cdot}\sigma_2\big(X_r^{\epsilon,\eta,u}, Y_r^{\epsilon,\eta,u}\big)dW_r\bigg)_{0,t}\\&+\bigg(   1+  \int_{0}^{t}(t-r)^{-a}\sup_{s\in[0,r]}\big|X_s^{\epsilon,\eta,u}    \big|^{\nu} dr+\big\|Y^{\epsilon,\eta,u}\big\|_{L^2([0,T];\mathcal{Y})}   \bigg)
	\\&+\bigg( 1+\sup_{t\in[0,T]}\big|X_t^{\epsilon,\eta,u}\big|^{2\nu_1}+\|Y^{\epsilon,\eta,u}\|_{L^2}\bigg)
	\\&+\sqrt{\epsilon}\|B^H\|_{T,1-a,\infty}\bigg(  1+\sup_{r\in[0,t]}\big| X_r^{\epsilon,\eta,u}\big|^{2\nu_1}+\big\|Y^{\epsilon,\eta,u}\big\|_{L^2([0,T];\mathcal{Y})}\\&+    \int_{0}^{t}   (t-r)^{-a}\big|\Delta_a\big|X_{0,r}^{\epsilon,\eta,u}dr+\int_{0}^{t}(t-s)^{-a-1}\int_{s}^{t}\big|\Delta_a\big|Y_{s,r}^{\epsilon,\eta,u}dr ds  \bigg)\\&
	\lesssim \sqrt{\epsilon}\|B^H\|_{T,1-a,\infty}\int_{0}^{t}   (t-r)^{-a}\big|\Delta_a\big|X_{0,r}^{\epsilon,\eta,u}dr\\&+\sqrt{\epsilon}\|B^H\|_{T,1-a,\infty}\int_{0}^{t}(t-s)^{-a-1}\int_{s}^{t}\big|\Delta_a\big|Y_{s,r}^{\epsilon,\eta,u}dr ds\\&
	+\big(1\vee\big\|B^H\|_{T,1-a,\infty}\big)\bigg[1+\sup_{t\in[0,T]}\big|X_t^{\epsilon,\eta,u}|^{\nu_1\vee\nu}+\big\|Y^{\epsilon,\eta,u}\big\|_{L^2([0,T];\mathcal{Y})}+\sqrt{\epsilon\eta}\big[Y^{\epsilon,\eta,u}\big]_{C^\theta}\\&+\sqrt{\epsilon}\big|\Delta_a\big|\bigg(\int_{0}^{\cdot}\nabla\Psi(Y^{\epsilon,\eta,u}_r)\tau(Y^{\epsilon,\eta,u}_r)dW_r\bigg)_{0,t}+\sqrt{\epsilon}\big|\Delta_a\big|\bigg(\int_{0}^{\cdot}\sigma_2\big(X_r^{\epsilon,\eta,u}, Y_r^{\epsilon,\eta,u}\big)dW_r\bigg)_{0,t}\bigg]. 
	\end{aligned}
	\end{equation*}
	\noindent The proof is complete upon invoking Gr\"onwall's inequality.\end{proof}

\noindent Up to this point we have only proved pathwise estimates for the slow motion. The following maximal inequality for the stochastic integral $\sigma_1dB^H$ is the last ingredient needed for the proof of Proposition \ref{tightnessprop}. Lemmas \ref{Ybndlem}, \ref{Yfraclem} provide probabilistic bounds for the fast motion and play a key role in the proof of the maximal inequality.

\begin{prop}\label{maxprop} Let $T>0, p\geq 1,$ $a\in(1-H, \frac{1}{4}\wedge\frac{\nu_2}{2})$ and $\theta\in(a,\frac{1}{2})$. Under Condition \ref{C6}(i) and for all $\epsilon$ sufficiently small there exists a constant $C$ such that 
	\begin{equation}\label{maxbnd}
	\begin{aligned}
	\ex\bigg[\sup_{t\in[0,T]}\bigg|\int_{0}^{t}\sigma_1\big(X_r^{\epsilon,\eta,u}, Y_r^{\epsilon,\eta,u}\big)dB^H_r\bigg|^p\bigg]&\leq C\bigg(1+ \ex\sup_{t\in[0,T]}\big| X_t^{\epsilon,\eta,u}\big|^p+\eta^{-p\beta}        \bigg).
	\end{aligned}
	\end{equation}
\end{prop}
\begin{proof}
	With $F_\epsilon$ as in \eqref{Feps},  let $Z_\epsilon:=(1+F_\epsilon)\|B^H\|_{T,1-a,\infty}.$  The pathwise estimates \eqref{preYoung} (with $s=0$) and \eqref{integralprebnd}, along with Young's product inequality with exponents $p_1=1/(\nu\vee 2\nu_1), p_2=1/(1-\nu\vee 2\nu_1)$,  furnish 
	\begin{equation*}\label{supH}
	\begin{aligned}
	\bigg|\int_{0}^{t}&\sigma_1\big(X_r^{\epsilon,\eta,u}, Y_r^{\epsilon,\eta,u}\big)dB^H_r\bigg|  \\&\leq C\|B^H\|_{T,1-a,\infty}
	\bigg(1+\sup_{t\in[0,T]}\big| X_r^{\epsilon,\eta,u}\big|^{\nu_1}+\big\|Y^{\epsilon,\eta,u}\big\|_{L^2([0,T];\mathcal{Y})}    	\\&+\sup_{t\in[0,T]}\big|\Delta_a\big|X_{0,t}^{\epsilon,\eta,u}+\int_{0}^{t}\big|\Delta_a\big|Y_{0,r}^{\epsilon,\eta,u}dr \bigg)\\&\leq c_1Z_\epsilon^{p_2}+c_2\sup_{t\in[0,T]}\big| X_t^{\epsilon,\eta,u}\big|\\&+C Z_\epsilon\bigg[ 1+\big\|Y^{\epsilon,\eta,u}\big\|_{L^2([0,T];\mathcal{Y})}+\int_{0}^{T}\big|\Delta_a\big|Y_{0,r}^{\epsilon,\eta,u}dr+     \sqrt{\epsilon\eta}\big[Y^{\epsilon,\eta,u}\big]_{C^\theta}\\&+\sqrt\epsilon\sup_{t\in[0,T]}\int_{0}^{t}(t-s)^{-a-1}\int_{s}^{t}\big|\Delta_a\big|Y_{s,r}^{\epsilon,\eta,u}drds\\&+  \sqrt{\epsilon}\sup_{t\in[0,T]}\big|\Delta_a\big|\bigg(\int_{0}^{\cdot}\nabla\Psi(Y^{\epsilon,\eta,u}_r)\tau(Y^{\epsilon,\eta,u}_r)dW_r\bigg)_{0,t}\\&+\sqrt{\epsilon}\sup_{t\in[0,T]}\big|\Delta_a\big|\bigg(\int_{0}^{\cdot}\sigma_2(X^{\epsilon,\eta,u}_r,Y^{\epsilon,\eta,u}_r)dW_r\bigg)_{0,t}\bigg]\\&
	=: c_1Z_\epsilon^{p_2}+c_2\sup_{t\in[0,T]}\big| X_t^{\epsilon,\eta,u}\big|+C Z_\epsilon Y_\epsilon.
	\end{aligned}
	\end{equation*}
	By virtue of Fernique's theorem, the $\mathscr{F}_T^H$-measurable random variable $Z_\epsilon$ has finite moments of all orders provided that $\epsilon$ is sufficiently small (see \eqref{fbmexpint}). The It\^o integral terms can be treated similarly so we only sketch the argument for $\int_{0}^{t}\nabla\Psi_r\tau_r dW_r.$ To this end, let $\theta\in(a,1/2).$ In view of Conditions \ref{C5}, \ref{C7} $\nabla\Psi(Y^{\epsilon,\eta,u}_\cdot)\tau(Y^{\epsilon,\eta,u}_\cdot)$ is uniformly bounded. From the Burkholder-Davis-Gundy (BDG) inequality and the Kolmogorov  continuity criterion it follows that the H\"older seminorm $[ \int_{0}^{\cdot}\nabla\Psi(Y^{\epsilon,\eta,u}_r)\tau(Y^{\epsilon,\eta,u}_r)dW_r    ]_{C^\theta}$ has moments of all orders that are uniformly bounded over $\epsilon<1$. Thus, the continuity of the first inclusion in the first line of \eqref{pmembeds} yields $\sup_{\epsilon<1}\sqrt{\epsilon}\ex[ \int_{0}^{\cdot}\nabla\Psi(Y^{\epsilon,\eta,u}_r)\tau(Y^{\epsilon,\eta,u}_r)dW_r    ]^p_{a,\infty}$ $<\infty.$ Finally, from Lemmas \ref{Ybndlem}, \ref{Yfraclem} and the Cauchy-Schwarz inequality we have 
	\begin{equation*}
	\ex\big[( Z_\epsilon Y_\epsilon)^p\big]^2 \leq  \ex\big[Z^{2p}_\epsilon]\ex\big[Y^{2p}_\epsilon]\leq C(1+\eta^{-2p\beta}).
	\end{equation*}
	The proof is complete.
\end{proof}

\begin{rem}\label{relativeraterem} Both Lemma \ref{DeltaX1} and Proposition \ref{maxprop} continue to hold under Condition \ref{C6}(ii). In this setting the proofs are simpler and the statements hold for all $a\in(1-H, 1/2).$ Moreover, if $\sigma=\sigma_1(x),$ it is possible to obtain the nonsingular bound
	\begin{equation*}\label{maxbnd2}
	\begin{aligned}
	\ex\bigg[\sup_{t\in[0,T]}\bigg|\int_{0}^{t}\sigma_1\big(X_r^{\epsilon,\eta,u}\big)dB^H_r\bigg|^p\bigg]&\leq C\bigg(1+ \ex\sup_{t\in[0,T]}\big| X_t^{\epsilon,\eta,u}\big|^p \bigg).
	\end{aligned}
	\end{equation*}
	Indeed, the  singular term $\eta^{-\beta}$ on the right-hand side of \eqref{maxbnd} accounts for the feedback of the fast motion which enters the calculations through the term $\int_{0}^{T}\big|\Delta_a\big|Y_{0,r}^{\epsilon,\eta,u}dr$ (see also Lemma \eqref{Yfraclem}$(i)$). The latter only appears in the case where $\sigma_1$ depends on $y.$
\end{rem}

\noindent We conclude this section with the proof of Proposition \ref{tightnessprop}.

\noindent \textit{Proof of Proposition \ref{tightnessprop}}.  We shall only prove $(i)$ since $(ii)$ is simpler and follows from identical arguments.  Combining \eqref{supb}, \eqref{supc}, \eqref{supcon} we have 
\begin{equation*}
\begin{aligned}
\big|X^{\epsilon,\eta,u}_t\big|^p&\lesssim |x_0|^p+\bigg(\frac{\sqrt{\epsilon}}{\sqrt\eta}\sup_{t\in[0,T]}\bigg|\int_{0}^{t}b(Y^{\epsilon,\eta,u}_s)ds\bigg|\bigg)^p+\sup_{t\in[0,T]}\bigg|\int_{0}^{t}  c\big(X_r^{\epsilon,\eta,u}, Y_r^{\epsilon,\eta,u}\big)dr\bigg|^p\\&+\sum_{i=1}^{2}\sup_{t\in[0,T]}\bigg|\int_{0}^{t}\sigma_i\big(X_r^{\epsilon,\eta,u}, Y_r^{\epsilon,\eta,u}\big)du_i(r)\bigg|^p\\&+\bigg(\sqrt{\epsilon}\sup_{t\in[0,T]}\bigg|\int_{0}^{t}\sigma_1\big(X_r^{\epsilon,\eta,u}, Y_r^{\epsilon,\eta,u}\big)dB^H_r\bigg|\bigg)^p\\&+\bigg(\sqrt{\epsilon}\sup_{t\in[0,T]}\bigg|\int_{0}^{t}\sigma_2\big(X_r^{\epsilon,\eta,u}, Y_r^{\epsilon,\eta,u}\big)dW_r\bigg|\bigg)^p\\&
\leq  C\bigg[1+\bigg(\sqrt{\epsilon}\sup_{t\in[0,T]}\bigg|\int_{0}^{t}\nabla\Psi(Y^{\epsilon,\eta,u}_s)\tau(Y^{\epsilon,\eta,u}_s)dW_s\bigg|\bigg)^p+  \int_{0}^{T}\sup_{s\in[0,r]}\big|X_s^{\epsilon,\eta,u}    \big|^p dr\\&+\big\|Y^{\epsilon,\eta,u}\big\|^p_{L^2([0,T];\mathcal{Y})}+\big(\sqrt{\epsilon\eta}\big[Y^{\epsilon,\eta,u}\big]_{C^\theta}\big)^p\\&
+\bigg(\sqrt{\epsilon}\sup_{t\in[0,T]}\bigg|\int_{0}^{t}\sigma_1\big(X_r^{\epsilon,\eta,u}, Y_r^{\epsilon,\eta,u}\big)dB^H_r\bigg|\bigg)^p\\&+\bigg(\sqrt{\epsilon}\sup_{t\in[0,T]}\bigg|\int_{0}^{t}\sigma_2\big(X_r^{\epsilon,\eta,u}, Y_r^{\epsilon,\eta,u}\big)dW_r\bigg|\bigg)^p\bigg].
\end{aligned}
\end{equation*}
Taking expectation and applying the Burkholder-Davis-Gundy inequality and \eqref{maxbnd} for the It\^o and Young integrals respectively, as well as Lemma \ref{Ybndlem} for the $L^2$ and H\"older norms of the fast process, we obtain
\begin{equation*}
\begin{aligned}
\ex\sup_{t\in[0,T]}\big|X^{\epsilon,\eta,u}_t\big|^p&
\leq  C\bigg[1+\int_{0}^{T}\ex\sup_{s\in[0,r]}\big|X_s^{\epsilon,\eta,u}    \big|^p dr\\&+\epsilon^{\frac{p}{2}}\bigg(\int_{0}^{T}\ex\big|\nabla\Psi(Y^{\epsilon,\eta,u}_s)\tau(Y^{\epsilon,\eta,u}_s)\big|^2ds\bigg)^\frac{p}{2}
\\&+\epsilon^{\frac{p}{2}} \ex\sup_{t\in[0,T]}\big| X_t^{\epsilon,\eta,u}\big|^p+\bigg(\frac{\sqrt{\epsilon}}{\eta^{\beta} }\bigg)^p\\&+\epsilon^{\frac{p}{2}}\bigg(\int_{0}^{T}\ex\big|\sigma_2(X^{\epsilon,\eta,u}_s,Y^{\epsilon,\eta,u}_s)\big|^2ds\bigg)^\frac{p}{2}\bigg].
\end{aligned}
\end{equation*}
For $\epsilon$ sufficiently small and due to the uniform boundedness of the It\^o integrands (Conditions \ref{C2}, \ref{C4}, \ref{C7}) we can rearrange and apply Gr\"onwall's inequality to deduce that
\begin{equation}\label{supbnd}
\ex\sup_{t\in[0,T]}\big|X^{\epsilon,\eta,u}_t\big|^p\leq C_{p,T,\Psi,\tau,\sigma_2,x_0}e^{CT}\bigg[1+\bigg(\frac{\sqrt{\epsilon}}{\eta^{\beta} }\bigg)^p\bigg]\leq C,
\end{equation}
where the last inequality holds for $\epsilon$ small from Condition \ref{C6}$(i)$.
It remains to estimate $|\Delta_a|X^{\epsilon}_{0,t}$. To this end, we return to the pathwise estimate \eqref{integralprebnd} and apply the Cauchy-Schwarz inequality for expectation to obtain	
\begin{equation*}
\begin{aligned}
\bigg(\ex& \sup_{t\in[0,T]}\big(\big|\Delta_a\big| X^{\epsilon,\eta,u}_{0,t}\big)^p\bigg)^2\\&\leq   C \ex [F_\epsilon^{2p}]\ex\bigg[1+\sup_{t\in[0,T]}\big|X_t^{\epsilon,\eta,u}|^{2p(2\nu_1\vee\nu)}+\big\|Y^{\epsilon,\eta,u}\big\|^{2p}_{L^2([0,T];\mathcal{Y})}+\big(\sqrt{\epsilon\eta}\big[Y^{\epsilon,\eta,u}\big]_{C^\theta}\big)^{2p}\\&+\sup_{t\in[0,T]}\big|\Delta_a\big|\bigg(\sqrt{\epsilon}\int_{0}^{\cdot}\nabla\Psi(Y^{\epsilon,\eta,u}_r)\tau(Y^{\epsilon,\eta,u}_r)dW_r\bigg)_{0,t}^{2p}\\&+\bigg(\sqrt{\epsilon}\sup_{t\in[0,T]}\int_{0}^{t}(t-s)^{-a-1}\int_{s}^{t}\big|\Delta_a\big|Y_{s,r}^{\epsilon,\eta,u}dr ds\bigg)^{2p}\\&
+\sup_{t\in[0,T]}\big|\Delta_a\big|\bigg(\sqrt{\epsilon}\int_{0}^{\cdot}\sigma_2(X^{\epsilon,\eta,u}_r,Y^{\epsilon,\eta,u}_r)dW_r\bigg)_{0,t}^{2p}\bigg]
\end{aligned}
\end{equation*}
Appealing once again to Fernique's theorem \eqref{fbmexpint}  along with \eqref{supbnd} and Lemmas \ref{Ybndlem}, \ref{Yfraclem} we obtain
\begin{equation}\label{Xbnd}
\begin{aligned}
\ex \sup_{t\in[0,T]}\big(\big|\Delta_a \big|X^{\epsilon,\eta,u}_{0,t}\big)^p&\leq   C\bigg[1+\bigg(\frac{\sqrt{\epsilon}}{\eta^\beta}\bigg)^p\bigg],
\end{aligned}
\end{equation}
and, due to Condition \eqref{C6}(i), the latter is finite for $\epsilon$ small. Note that the It\^o integrals can been treated as in the proof of Proposition \ref{maxprop}. Combining \eqref{supbnd} and \eqref{Xbnd} we conclude that 
\begin{equation*}
\sup_{\epsilon<\epsilon_0, u\in\mathcal{A}_N}\ex\big\|X^{\epsilon,\eta,u}\big\|^p_{a,\infty}\leq \sup_{\epsilon<\epsilon_0, u\in\mathcal{A}_N}\bigg(\ex\sup_{t\in[0,T]}\big|X^{\epsilon,\eta,u}_t\big|^p	+			\ex \sup_{t\in[0,T]}\big(\big|\Delta_a\big|X^{\epsilon,\eta,u}_{0,t}\big)^p\bigg)<\infty.
\end{equation*}

\section{Limiting behavior of the controlled dynamics}\label{limsec}
\noindent This section is devoted to the proof of Theorem \ref{main1}. To this end, we first show in Lemma \ref{tightlem} that, for some sufficiently small $\epsilon_0>0,$ the family \begin{equation}\label{Tfam}
\mathscr{T}=\{ (X^{\epsilon,\eta,u},P^\epsilon); \epsilon<\epsilon_0, u\in\mathcal{A}_N\}
\end{equation} of controlled slow processes \eqref{consys} and occupation measures \eqref{occupation} is tight. Then, we characterize the limiting behavior of $P^\epsilon$ in Lemma \ref{occulem} and the limiting dynamics of  $X^{\epsilon,\eta,u}$ in Lemma \ref{locave} and Proposition \ref{sigmaave}.  We emphasize that, from this point on and throughout the rest of this work, we shall replace the growth assumptions \eqref{sigmagrowth1}, \eqref{sigmagrowth2} of Conditions \ref{C6}$(i)$, $(ii)$ with the stronger assumption that $\sigma_1$ is bounded. This assumption simplifies the proofs of convergence (Proposition \ref{sigmaave}) and the Laplace lower bound (Section \ref{LPlbsec}). An investigation of the optimal growth rates of $\sigma_1$ under which our results still hold is beyond the scope of this paper and is left for future work.

Before we move on to the main body of this section let us recall the notion of tightness for a family of probability measures as well as the classical theorem of Prokhorov.
\begin{dfn} Let $\mathcal{E}$ be a Polish space. A family  $\Pi\subset\mathscr{P}(\mathcal{E})$ of probability measures is called tight if for any $\delta>0$ there exists a compact set $K_\delta\subset\mathcal{E}$ such that  \[\sup_{P\in\Pi}P(\mathcal{E}\setminus K_\delta)<\delta.\]
\end{dfn}

\noindent Prokhorov's theorem asserts that  the notions of tightness and relative weak sequential compactness on $\mathscr{P}(\mathcal{E})$ are equivalent, provided that $\mathcal{E}$ is a Polish space.

\begin{thm}(Prokhorov) Let $\Pi\subset\mathscr{P}(\mathcal{E})$ be a family of probability measures on the Polish space $\mathcal{E}.$ Every sequence in $\Pi$ has a convergent subsequence in the topology of weak convergence of measures if and only if $\Pi$ is tight.
\end{thm}

\begin{lem}\label{tightlem} The family $\mathcal{T} \eqref{Tfam}$ is tight in $C([0,T];\mathcal{X})\times\mathscr{P}([0,T]\times\mathcal{U}_1\times\mathcal{U}_2\times\mathcal{Y}).$	
\end{lem}
\begin{proof}
	We first show that the laws of the occupation measures $\{P^{\epsilon}\}_{\epsilon}$ form a tight family in   $\mathscr{P}(\mathscr{P}([0,T]\times\mathcal{U}_1\times\mathcal{U}_2\times\mathcal{Y})).$ To this end, let $\phi: [0,T]\times\mathcal{U}_1\times\mathcal{U}_2\times\mathcal{Y}\rightarrow[0,\infty]$ with 
	$\phi(t,u_1,u_2,y)=|u_1|+|u_2|+|y|$ and note that, for each $M>0$, the sub-level set $\{ \phi\leq M      \}$ is compact. In view of \cite{dupuis2011weak}, Theorem A.3.17, $F: \mathscr{P}([0,T]\times\mathcal{U}_1\times\mathcal{U}_2\times\mathcal{Y})\rightarrow[0,\infty]$, $$F(\theta)=\int_{[0,T]\times\mathcal{U}_1\times\mathcal{U}_2\times\mathcal{Y}}\phi d\theta$$
	is a tightness function and using the Cauchy-Schwarz inequality,  
	the definitions of $\mathcal{A}_N$ \eqref{AN} and the Cameron-Martin space $\h_H$ \eqref{CMdef} along with Lemma \ref{Ybndlem}(i),
	\begin{equation}\label{occutight}
	\begin{aligned}
	\sup_{\epsilon\in(0,1)}\ex F(P^{\epsilon})&=\sup_{\epsilon\in(0,1)}\int_{0}^{T}\ex\bigg[|K^{-1}_{H}u^{\epsilon}_1(s)|+|\dot{u}^{\epsilon}_2(s)|+|Y^{\epsilon,\eta,u^\epsilon}_s|\bigg]ds\\&\leq C_{T}\sup_{\epsilon\in(0,1)}\int_{0}^{T}\ex|K^{-1}_{H}u^{\epsilon}_1(s)|^2+\ex|\dot{u}^{\epsilon}_2(s)|^2+\ex|Y^{\epsilon,\eta,u^\epsilon}_s|^2ds<\infty.
	\end{aligned}
	\end{equation}
	An application of Chebyshev's inequality concludes the arguments. Turning to the controlled slow processes $\{X^{\epsilon,\eta,u};\epsilon,u\},$ the estimate \eqref{tightnessX} along with the Arzel\`a-Ascoli criterion (see Theorem 7.2 in \cite{94060261}) yields the desired conclusion.\end{proof}

\noindent The following uniform integrability property is a byproduct of tightness and will be used in the identification of the limiting slow dynamics.
\begin{cor}\label{UI} The occupation measures $P^\epsilon$ are uniformly integrable in the sense that
	\begin{equation*}
	\lim_{M\to\infty}\sup_{\epsilon<1}\ex\int_{[0,T]\times\{|u_1|> M\}\times\{|u_2|> M\}\times\{|y|> M\}}\big(|u_1|+|u_2|+|y|\big)dP^\epsilon(t,u_1,u_2,y)=0.
	\end{equation*}
\end{cor}
\begin{proof} Let $M,\epsilon>0$. From \eqref{occutight} and Chebyshev's inequality we have 
	\begin{equation*}
	\ex\int_{[0,T]\times\{|u_1|> M\}\times\{|u_2|> M\}\times\{|y|> M\}}\big(|u_1|+|u_2|+|y|\big)dP^\epsilon(t,u_1,u_2,y)\leq \frac{1}{M}\sup_{\epsilon\in(0,1)}\ex F(P^\epsilon)
	\end{equation*}
	The proof is complete upon taking $M\to\infty$.
\end{proof}

\noindent As we showed, there exists $\epsilon_0>0$ such that the family $\mathscr{T}$ is tight. From Prokhorov's theorem, any sequence of elements in $\mathscr{T}$ has a subsequence that converges in distribution to a probability measure  $\mu_\ell$ on  $C([0,T];\mathcal{X})\times\mathscr{P}([0,T]\times\mathcal{U}_1\times\mathcal{U}_2\times\mathcal{Y})$. At this point we invoke the Skorokhod representation theorem which allows us to assume that the subsequence converges to $(\psi, P)$ almost surely in $C([0,T];\mathcal{X})\times\mathscr{P}([0,T]\times\mathcal{U}_1\times\mathcal{U}_2\times\mathcal{Y})$. Skorokhod's theorem introduces a new probability space which, for the sake of simplicity, will not be reflected in our notation. Our goal is to provide a characterization of the limit points $(\psi, P)$ as viable pairs, according to Definition \ref{viable}.

The next lemma shows that the limiting measures $P$ satisfy properties $(i)$ and $(ii)$ of Definition \ref{viable}.
\begin{lem}\label{occulem} Let $T>0,$ $\mathcal{L}$ as in \eqref{generator} and assume that $P^\epsilon\rightarrow P$ almost surely in $\mathscr{P}([0,T]\times\mathcal{U}_1\times\mathcal{U}_2\times\mathcal{Y})$. The following hold:\\
	\noindent (i) \begin{equation*}
	\pr\bigg[\int_{[0,T]\times\mathcal{U}_1\times\mathcal{U}_2\times\mathcal{Y}}\big(|u_1|^2+|u_2|^2+|y|^2\big)dP(s,u_1,u_2,y)<\infty\bigg]=1,
	\end{equation*} 	
	\noindent (ii)
	\begin{equation}\label{mumarg}
	\pr\bigg[ \forall t\in[0,T], h\in C_c^2(\mathcal{Y}),	\int_{[0,t]\times\mathcal{U}_1\times\mathcal{U}_2\times\mathcal{Y}}\mathcal{L}h(y)dP(s,u_1,u_2,y)=0\bigg]=1,
	\end{equation}
	(iii) for all $t\in[0,T]$, \begin{equation}\label{Lebmarg}
	P( [0,t]\times \mathcal{U}_1\times\mathcal{U}_2\times\mathcal{Y} )=t.
	\end{equation}
\end{lem}
\begin{proof} $(i)$ In view of \eqref{occutight} we have 
	\begin{equation*}
	\begin{aligned}
	\sup_{\epsilon\in(0,1)}&\ex\bigg[\int_{[0,T]\times\mathcal{U}_1\times\mathcal{U}_2\times\mathcal{Y}}\big(|u_1|^2+|u_2|^2+|y|^2\big)dP^\epsilon(t,u_1,u_2,y)\bigg]\\&=
	\sup_{\epsilon\in(0,1)}\int_{0}^{T}\ex|K^{-1}_{H}u^{\epsilon}_1(s)|^2+\ex|\dot{u}^{\epsilon}_2(s)|^2+\ex|Y^{\epsilon,\eta,u^\epsilon}_s|^2ds<\infty.
	\end{aligned}
	\end{equation*}
	In view of the lower semicontinuity of the map $\mathcal{U}_1\times\mathcal{U}_2\times\mathcal{Y}\ni(u_1,u_2,y)\mapsto |u_1|^2+|u_2|^2+|y|^2\in[0,\infty]$ and the weak convergence of $P^\epsilon$ to $P$, the Portmanteau lemma (Theorem A.3.4 in \cite{dupuis2011weak}) concludes the proof.	
	
	\noindent $(ii)$ Let $\epsilon>0$, $D_1\subset [0,T], D_2\subset C_c^2(\mathcal{Y}) $ be countable, dense sets and $(t,h)\in D_1\times D_2$. An application of It\^o's formula yields
	\begin{equation*}	
	\begin{aligned}
	h(Y_t^{\epsilon,\eta,u})&-h(y_0)\\&=\frac{1}{\eta}\int_{0}^{t}\mathcal{L}h(Y_s^{\epsilon,\eta,u})ds+\tau(Y^{\epsilon,\eta,u}_s)\dot{u_2}(s)\big]ds+\frac{1}{\sqrt{\eta}}\int_{0}^{t}\nabla h(Y^{\epsilon,\eta,u}_s)\tau(Y^{\epsilon,\eta,u}_s)dW_s\\&+\frac{1}{\sqrt{\epsilon\eta}}\int_{0}^{t}\nabla h(Y^{\epsilon,\eta,u}_s)\big[g(X^{\epsilon,\eta,u}_s,Y^{\epsilon,\eta,u}_s)+\tau(Y^{\epsilon,\eta,u}_s)\dot{u_2}(s)\big]ds.
	\end{aligned}
	\end{equation*}
	Therefore,
	\begin{equation*}
	\begin{aligned}
	\int_{[0,t]\times\mathcal{U}_1\times\mathcal{U}_2\times\mathcal{Y}}\mathcal{L}&h(y)dP^\epsilon(s,u,v,y)=\int_{0}^{t}\mathcal{L}h(Y_s^{\epsilon,\eta,u})ds
	\\&= \eta \big[h(Y_t^{\epsilon,\eta,u})-h(y_0)\big]-\sqrt{\eta}\int_{0}^{t}\nabla h(Y^{\epsilon,\eta,u}_s)\tau(Y^{\epsilon,\eta,u}_s)dW_s\\&		
	-\frac{\sqrt{\eta}}{\sqrt{\epsilon}}\int_{0}^{t}\nabla h(Y^{\epsilon,\eta,u}_s)\big[g(X^{\epsilon,\eta,u}_s,Y^{\epsilon,\eta,u}_s)+\tau(Y^{\epsilon,\eta,u}_s)\dot{u_2}(s)\big]ds
	\end{aligned}
	\end{equation*}
	Since $h,\nabla h$ and $\tau$ (Condition \ref{C2}) are bounded, $g$ only grows at most linearly in $y$ (Condition \ref{C1}),  $Y^{\epsilon,\eta,u}$ is uniformly bounded in $L^2(\Omega\times[0,T]),$ $\dot{u_2}$ is square integrable with probability $1$ and $\sqrt{\eta/\epsilon}\rightarrow 0$, the right-hand side of the equality converges to $0$ in probability. Turning to the left-hand side, let $M>0$ and write
	\begin{equation*}
	\begin{aligned}
	\int_{[0,t]\times\mathcal{U}_1\times\mathcal{U}_2\times\mathcal{Y}}\mathcal{L}h(y)dP^\epsilon(s,u,v,y)&=	\int_{[0,t]\times\mathcal{U}_1\times\mathcal{U}_2\times\{|y|\leq M\}}\mathcal{L}h(y)dP^\epsilon(s,u,v,y)\\&+	\int_{[0,t]\times\mathcal{U}_1\times\mathcal{U}_2\times\{|y|>M\}}\mathcal{L}h(y)dP^\epsilon(s,u,v,y).
	\end{aligned}
	\end{equation*}
	Note that $y\mapsto\mathcal{L}h(y)=\frac{1}{2}[D^2h:(\tau\tau^T)](y)+\nabla h(y)f(y)$ is a continuous bounded function on $[0,t]\times\mathcal{U}_1\times\mathcal{U}_2\times\{|y|\leq M\}$. 
	Since $P^\epsilon\rightarrow P$ almost surely in the topology of weak convergence of measures, the first term on the right-hand side converges to $\int_{[0,t]\times\mathcal{U}_1\times\mathcal{U}_2\times\mathcal{Y}}\mathcal{L}h(y)dP(s,u,v,y)$ almost surely. The second term vanishes in $L^1(\Omega)$ as $M\to\infty$, uniformly in $\epsilon$, in light of the uniform integrability of  $\{P^\epsilon;\epsilon>0\}$ (Corollary \ref{UI}). It follows that \eqref{mumarg} holds on the complement of a $\pr-$null set $N_{t,h}\subset\Omega$. Using the continuity in $t,h$ and the density of $D_1,D_2$ we deduce that \eqref{mumarg} holds in the complement of the  $\pr-$null set $\cup_{(t,h)\in D_1\times D_2}N_{t,h}$.
	
	\noindent $(iii)$  Let $t\in(0,T].$ Since the set $[0, t]\times\mathcal{U}_1\times\mathcal{U}_2\times\mathcal{Y}$ is closed, the Portmanteau lemma furnishes
    $P([0, t]\times\mathcal{U}_1\times\mathcal{U}_2\times\mathcal{Y})\geq \limsup_{\epsilon\to0}P^{\epsilon}([0, t]\times\mathcal{U}_1\times\mathcal{U}_2\times\mathcal{Y})=t.$ To prove the reverse inequality, let $n\in\N$ large enough so that $t>\tfrac{2}{n}$ and apply the Portmanteau lemma to the open set $(\tfrac{1}{n}, t-\tfrac{1}{n})$ to obtain 
    $$  P( (\tfrac{1}{n}, t-\tfrac{1}{n})\times \mathcal{U}_1\times\mathcal{U}_2\times\mathcal{Y})\leq \liminf_{\epsilon\to 0}P^\epsilon( (\tfrac{1}{n}, t-\tfrac{1}{n})\times \mathcal{U}_1\times\mathcal{U}_2\times\mathcal{Y})=t-\tfrac{2}{n}.$$
    Taking $n\to\infty$ in both sides of the last inequality and noting that $(\tfrac{1}{n}, t-\tfrac{1}{n})\downarrow[0, t]$ we conclude that $ P( [0, t]\times\mathcal{U}_1\times\mathcal{U}_2\times\mathcal{Y})\leq t.$ The proof is complete.
    \end{proof}

\begin{rem} 
	The previous lemma implies that any weak limit point $P$ of the occupation measures $P^\epsilon$ satisfies properties (i), (ii) of Definition \ref{viable} with probability $1$. Indeed, from \eqref{Lebmarg} we see that $P$ agrees with Lebesgue measure on any closed sub-interval of $[0,T]$. Thus, with probability $1$, the last marginal of $P$ coincides with Lebesgue measure on $\mathscr{B}([0,T])$. Next, consider the differential operator $\mathcal{L}: Dom(\mathcal{L})\subset C_c(\mathcal{Y})\rightarrow C_c(\mathcal{Y})$ on the dense subspace $Dom(\mathcal{L})$ of  $C_c(\mathcal{Y})$ which contains $C^2_c(\mathcal{Y})$. In view of the Riesz-Markov-Kakutani theorem, the topological dual of $C_c(\mathcal{Y})$ can be identified with the space $\mathscr{M}(\mathcal{Y})$ of real-valued Radon measures on $\mathcal{Y}.$ Thus, the transpose operator can be viewed as a linear map $\mathcal{L}^*:\mathscr{M}(\mathcal{Y})\rightarrow \mathscr{D}^{' 2}(\mathcal{Y}),$ the latter being the space of Schwartz distributions of order at most $2$. Letting $\Pi_{\mathcal{Y}}:[0,T]\times\mathcal{U}_1\times\mathcal{U}_2\times\mathcal{Y}\rightarrow\mathcal{Y}$ be the projection map, \eqref{mumarg} implies that $P\circ\Pi_{\mathcal{Y}}^{-1}\in Ker(\mathcal{L}^*)$. Since the Markov process $Y^{1}$ is uniquely ergodic, the latter implies that $P\circ\Pi_{\mathcal{Y}}^{-1}=\mu$ (see e.g. \cite{bellet2006ergodic}, Chapters 7, 8). Finally, since $[0,T]\times\mathcal{U}_1\times\mathcal{U}_2\times\mathcal{Y}$ is a Polish space, we can disintegrate $dP(u,v,y,t)$ into $d\Theta(u,v|y,t)d\mu(y)dt$ where $\Theta$ is a regular conditional probability (or stochastic kernel) with probability $1$ (see Appendix A.5 of \cite{dupuis2011weak} for definition and properties of stochastic kernels). 
\end{rem}

It remains to show that any limiting pair $(\psi,P)$ satisfies property (iii) of Definition \ref{viable} with $\Lambda=(\Lambda_1, \Lambda_2)$ as in Theorem \ref{main1}. 	In view of \eqref{consys},\eqref{bterm} we have
\begin{equation}\label{Xlimeq}
\begin{aligned}
X_t^{\epsilon,\eta,u}&=x_0+\int_{0}^{t}\Lambda_1\big(X_s^{\epsilon,\eta,u},\dot{u}_1(s),\dot{u}_2(s),Y_s^{\epsilon,\eta,u}\big)ds+\int_{0}^{t}\sigma_1\big(X_s^{\epsilon,\eta,u}, Y_s^{\epsilon,\eta,u}\big)du_1(s)\\&-\sqrt{\epsilon\eta}[\Psi(Y^{\epsilon,\eta,u}_t)-\Psi(y)]
+\sqrt{\epsilon}\int_{0}^{t}\sigma_1\big(X_s^{\epsilon,\eta,u}, Y_s^{\epsilon,\eta,u}\big)dB^H_s\\&+\sqrt{\epsilon}\int_{0}^{t}\nabla\Psi(Y^{\epsilon,\eta,u}_s)\tau(Y^{\epsilon,\eta,u}_s)dW_s.
\end{aligned}  
\end{equation}
From the estimates of the previous section it follows that each term on the second line of the last display vanishes as $\epsilon\to 0$. In particular, Proposition \ref{maxprop}, Proposition \ref{tightnessprop}, Lemmas \ref{Ybndlem}, \ref{Yfraclem}(ii) and Conditions \ref{C6}, \ref{C7} imply that 
\begin{equation*}
\begin{aligned}
\bigg\|\sqrt{\epsilon}\int_{0}^{\cdot}\sigma_1\big(X_s^{\epsilon,\eta,u}, Y_s^{\epsilon,\eta,u}\big)dB^H_s
&-\sqrt{\epsilon\eta}[\Psi(Y^{\epsilon,\eta,u}_\cdot)-\Psi(y)]\\&+\sqrt{\epsilon}\int_{0}^{\cdot}\nabla\Psi(Y^{\epsilon,\eta,u}_s)\tau(Y^{\epsilon,\eta,u}_s)dW_s\bigg\|_{L^1(\Omega;C([0,T];\mathcal{X}))}
\end{aligned}
\end{equation*}
$=O(\sqrt{\epsilon}/\eta^\beta)$, as $\epsilon\to 0.$ In fact, under Condition \ref{C6}(ii), the latter is $O(\sqrt{\epsilon})$. The next lemma addresses the limiting behavior of the second summand on the right-hand side of \eqref{Xlimeq}.
\begin{lem}\label{locave}
	Assume that $(X^{\epsilon,\eta,u},P^\epsilon)\rightarrow (\psi, P)$ in $C([0,T];\mathcal{X})\times\mathscr{P}([0,T]\times\mathcal{U}_1\times\mathcal{U}_2\times\mathcal{Y})$ almost surely and let $\Lambda_1$ as in \eqref{Lambda1}. For all $t\in[0,T],$ the following limit is valid in distribution:
	\begin{equation*}
	\begin{aligned}
	\lim_{\epsilon\to 0 } \int_{0}^{t}\Lambda_1\big(X_s^{\epsilon,\eta,u},\dot{u}_1(s),\dot{u}_2(s),Y_s^{\epsilon,\eta,u} \big)ds=\int_{[0,t]\times\mathcal{U}_1\times\mathcal{U}_2\times\mathcal{Y}}\Lambda_1\big(\psi_s,u_1,u_2,y \big)dP(s,u_1,u_2,y).
	\end{aligned}       
	\end{equation*}
\end{lem}
\begin{proof}
	The proof follows along the lines of \cite{dupuis2012large}, Lemma 3.2. 	The difference here is that the fast motion evolves in the unbounded domain $\mathcal{Y}$. First note that from the Lipschitz continuity of $\Lambda_1$ in $x,y,$ the affine dependence of $\Lambda_1$ in the last argument, the square integrability of the control $\dot{u}_2$  and the $\pr$-almost-sure uniform convergence of $X^{\epsilon,\eta,u}$ to $\psi,$ it suffices to study the term $\int_{0}^{t}\Lambda_1\big(\psi_s,\dot{u}_1(s),\dot{u}_2(s),Y_s^{\epsilon,\eta,u} \big)ds.$ By definition of $\Lambda_1$ and the occupation measures $P^\epsilon,$ the latter is equal to
	$\int_{[0,t]\times\mathcal{U}_1\times\mathcal{U}_2\times\mathcal{Y}}\Lambda_1\big(\psi_s,u_1,u_2,y\big) dP^\epsilon(s,u_1,u_2,y)$.
	Now let $M>0$ and define a cutoff function $F_M:\mathcal{Y}\rightarrow\mathcal{Y} 
	$ by \begin{equation*}
	F_M(y)=\begin{cases}
	&y\;,\;\;|y|\leq M\\&
	\frac{M}{|y|}y\;,\;\;|y|> M.
	\end{cases}
	\end{equation*}
	Decomposing the domain of integration we have 
	\begin{equation*}
	\begin{aligned}
	&\bigg|\int_{[0,t]\times\mathcal{U}_1\times\mathcal{U}_2\times\mathcal{Y}}\Lambda_1\big(\psi_s,u_1,u_2,y\big) d[P^\epsilon(s,u_1,u_2,y)-P(s,u_1,u_2,y)]\bigg|\\&
	\leq \bigg|\int_{[0,t]\times\mathcal{U}_1\times\mathcal{U}_2\times\{|y_2|\leq M\}}\Lambda_1\big(\psi_s,u_1,u_2,y\big) d[P^\epsilon(s,u_1,u_2,y)-P(s,u_1,u_2,y)]\bigg|\\&+
	\bigg|  \int_{[0,t]\times\mathcal{U}_1\times\mathcal{U}_2\times  \{|y_2|> M\}}\Lambda_1\big(\psi_s,u_1,u_2,y\big) d[P^\epsilon(s,u_1,u_2,y)-P(s,u_1,u_2,y)]\bigg|\\&
	\leq \bigg|\int_{[0,t]\times\mathcal{U}_1\times\mathcal{U}_2\times\mathcal{Y}}\Lambda_1\big(\psi_s,u_1,F_M(u_2),F_M(y)\big) d[P^\epsilon(s,u_1,u_2,y)-P(s,u_1,u_2,y)]\bigg|\\&+
	C_{\Psi,\tau}  \int_{[0,t]\times\mathcal{U}_1\times\{|v_2|> M\}\times  \{|y_2|> M\}}\big[1+|\psi_s|+ |y|+|u_2|\big]d[P^\epsilon(s,u_1,u_2,y)+P(s,u_1,u_2,y)],
	\end{aligned}
	\end{equation*}
	where the last inequality holds since $\Lambda_1$ grows at most linearly in all its arguments. Regarding the first term in the last display, note that for each $M>0, t\in[0,T]$, the map
	\begin{equation*}
	[0,T]\times\mathcal{U}_1\times\mathcal{U}_2\times\mathcal{Y} \ni(s,u_1,u_2,y)\longmapsto \Lambda_1\big(\psi_s,u_1,F_M(u_2),F_M(y)\big)\mathds{1}_{[0,t]}(s)\in\mathcal{X}
	\end{equation*}
	is bounded with probability $1$ and its discontinuity set is contained in $\{t\}\times\mathcal{U}_1\times\mathcal{U}_2\times\mathcal{Y}$.
	Since $P^\epsilon\rightarrow P$ almost surely in the topology of weak convergence of measures and, by Lemma \ref{occulem}(iii), the last marginal of $P$ is Lebesgue measure we can invoke \cite{dupuis2011weak}, Theorem A.3.10  and the dominated convergence theorem to deduce that this term converges to $0$, as $\epsilon\to0$, in $L^1(\Omega).$ Finally, the second term is bounded in expectation by
	\begin{equation*}
	C \sup_{\epsilon>0}\ex\int_{[0,T]\times\mathcal{U}_1\times\{|v_2|> M\}\times  \{|y_2|> M\}}\big[1+\sup_{t\in[0,T]}|\psi_t|+ |y|+|u_2|\big]dP^\epsilon(s,u_1,u_2,y),
	\end{equation*}
	which converges to $0$, as $M\to\infty$, from the uniform integrability of $P^\epsilon$. The proof is thus complete upon taking limits first as $\epsilon\to 0$ and then as $M\to \infty$.
\end{proof}

\noindent We conclude the proof of Theorem \ref{main1} by identifying the averaging limit of the integral $\sigma_1du_1.$ As pointed out in Section \ref{resultssec}, this term is essentially different from the case $H=1/2$ since here one needs to account for the memory kernel of the fBm $B^H.$ Before we proceed, we remind the reader that Remark \ref{rem:occupationDomainConvention} is in effect, i.e. we shall omit domains of integration for variables that do not appear in the integrands of the occupation and limiting measures $P^\epsilon, P.$

\begin{prop}\label{sigmaave}
	(Averaging of $\sigma_1 du_1$) Assume that $(X^{\epsilon,\eta,u},P^\epsilon)\rightarrow (\psi, P)$ in $C([0,T];\mathcal{X})\times\mathscr{P}([0,T]\times\mathcal{U}_1\times\mathcal{U}_2\times\mathcal{Y})$ almost surely. With $\Lambda_2$ as in \eqref{Lambda2} and for all $t\in[0,T],$ the following limit is valid in distribution up to a subsequence: 
	\begin{equation*}
	\begin{aligned}&\lim_{\epsilon\to 0 } \int_{0}^{t}\sigma_1\big(X_r^{\epsilon,\eta,u}, Y_r^{\epsilon,\eta,u}\big)du_1(r)\\&=\int_{[0,t]\times\mathcal{Y} }\int_{[0,s]\times\mathcal{U}_1}         \Lambda_2(\psi_s,s,u_1,u_2,y_2, z,v_1,v_2,y_1 )dP(z,v_1,v_2,y_1)dP(s,u_1,u_2,y_2).
	\end{aligned}
	\end{equation*}
	
\end{prop}
\begin{proof} 	Since $u_1\in\mathcal{H}_H$ \eqref{CMdef}, there exists $v_1\in L^2([0,T];\mathcal{Y})$ such that $u_1=K_Hv_1$. In view of Condition \ref{C4},  Lemma \ref{udotL2} and the weak convergence of $X^{\epsilon,\eta,u}$ to $\psi,$ we have
	\begin{equation*}
	\begin{aligned}
	\sup_{t\in[0,T]}\bigg|&\int_{0}^{t}\sigma_1\big(X_r^{\epsilon,\eta,u}, Y_r^{\epsilon,\eta,u}\big)du_1(r)-\int_{0}^{t}\sigma_1\big(\psi_r, Y_r^{\epsilon,\eta,u}\big)du_1(r)\bigg|\\&\leq \int_{0}^{t}\big|\sigma_1\big(X_r^{\epsilon,\eta,u}, Y_r^{\epsilon,\eta,u}\big)-\sigma_1\big(\psi_r, Y_r^{\epsilon,\eta,u}\big)\big|\big|\dot{K}_Hv_1(r)\big|dr\\&
	\leq L_{\sigma_1}\sup_{r\in[0,T]}\big|X_r^{\epsilon,\eta,u}- \psi_r\big|\int_{0}^{T}\big|\dot{K}_Hv_1(r)\big|dr\\&
	\leq C\big\|\dot{K}_Hv_1\big\|_{L^2} \sup_{r\in[0,T]}\big|X_r^{\epsilon,\eta,u}- \psi_r\big|\longrightarrow 0,
	\end{aligned}	
	\end{equation*}
	as $\epsilon\to 0$. Here $L_{\sigma_1}$ is the Lipschitz constant of $\sigma_1$ and $\dot{K}_H$ is given in $\eqref{Kdot}.$ From \eqref{sigma1comp} we can then write
	\begin{equation*}
	\begin{aligned}
	&\int_{0}^{t}\sigma_1(\psi_s,Y^{\epsilon,\eta,u}_s)du_1(s)= \frac{c_H}{\Gamma(H-\frac{1}{2})}\cdot\\&\cdot\int_{[0,t]\times\mathcal{Y} }\int_{[0,s]\times\mathcal{U}_1}s^{H-\frac{1}{2}}z^{\frac{1}{2}-H}(s-z)^{H-\frac{3}{2}} \sigma_1(\psi_s,y_2)v_1dP^{\epsilon}(z,v_1,v_2,y_1)dP^{\epsilon}(s,u_1,u_2,y_2).
	\end{aligned}\end{equation*}
	Note that the integrand on the right-hand side of the last display is singular at $s=z$ and $z=0.$ In order to apply a weak convergence argument  we decompose the domain of integration as follows: let $\lambda\in(0,\frac{1}{2}),$ $\rho:=\lambda t$ and set $C_H=c_H/(\Gamma(H-1/2)).$ Since  the last marginal of $P$ is Lebesgue measure on $[0,T]$ we can write
	\begin{equation*}
	\begin{aligned}
	&C_H^{-1}\bigg[\int_{0}^{t}\sigma_1(\psi_s,Y^{\epsilon,\eta,u}_s)du_1(s)\\&-\int_{[0,t]\times\mathcal{Y} }\int_{[0, s]\times\mathcal{U}_1}  \Lambda_2(\psi_s,s,u_1,u_2,y_2, z,v_1,v_2,y_1 )dP^{\otimes 2}(z,v_1,v_2,y_1,s,u_1,u_2,y_2)\bigg]\\&
	=\int_{[2\rho,t]\times\mathcal{Y} }\int_{[\rho, s-\rho]\times\mathcal{U}_1}s^{H-\frac{1}{2}}z^{\frac{1}{2}-H}(s-z)^{H-\frac{3}{2}}\\&\quad\quad\times \sigma_1(\psi_s,y_2)v_1d\big[\big(P^{\epsilon,\otimes 2}-P^{\otimes 2}\big)(z,v_1,v_2,y_1,s,u_1,u_2,y_2)\big]\\&
	+\int_{[0,\rho]\times\mathcal{U}_1 }\int_{[z,t] \times\mathcal{Y}}	+\int_{[\rho, t-\rho]\times\mathcal{U}_1 }\int_{[z, z+\rho]\times\mathcal{Y}}+\int_{[t-\rho, t]\times\mathcal{U}_1 }\int_{[z,t]\times\mathcal{Y}}\\&\;\bigg(s^{H-\frac{1}{2}}z^{\frac{1}{2}-H}(s-z)^{H-\frac{3}{2}} \sigma_1(\psi_s,y_2)v_1\bigg)d\big[\big(P^{\epsilon,\otimes 2}-P^{\otimes 2}\big)(s,u_1,u_2,y_2,z,v_1,v_2,y_1)\big]\\&
	=:\sum_{i=0}^{3}R_i(\epsilon,\lambda,t).
	\end{aligned}
	\end{equation*}
	We start by treating the term $R_0$. To this end, let $\xi,M>0$ and recall the cutoff function $F_M:\mathcal{U}_1\rightarrow \mathcal{U}_1$ from the proof of Lemma \ref{locave}.
	An application of Chebyshev's inequality yields
	\begin{equation*}
	\begin{aligned}
	&\xi^{-1}\pr[|R_0(\epsilon,\lambda,t)|\geq\xi ]\leq\\&  \ex\bigg|\int_{[2\rho,t]\times\mathcal{Y} }\int_{[\rho, s-\rho]\times\{|v_1|\leq M\}}s^{H-\frac{1}{2}}z^{\frac{1}{2}-H}(s-z)^{H-\frac{3}{2}}\\&\times \sigma_1(\psi_s,y_2)v_1d\big[\big(P^{\epsilon,\otimes 2}-P^{\otimes 2}\big)(z,v_1,v_2,y_1,s,u_1,u_2,y_2)\big]\bigg|\\&
	+ \ex\int_{[2\rho,t]\times\mathcal{Y} }\int_{[\rho, s-\rho]\times\{|v_1|> M\}}s^{H-\frac{1}{2}}z^{\frac{1}{2}-H}(s-z)^{H-\frac{3}{2}}\\&\times \big|\sigma_1(\psi_s,y_2)\big||v_1|d\big[\big(P^{\epsilon,\otimes 2}+P^{\otimes 2}\big)(z,v_1,v_2,y_1,s,u_1,u_2,y_2)\big]\\&
	\leq \ex\bigg|\int_{[2\rho,t]\times\mathcal{Y} }\int_{[\rho, s-\rho]\times\mathcal{U}_1}s^{H-\frac{1}{2}}z^{\frac{1}{2}-H}(s-z)^{H-\frac{3}{2}} \\&\times\sigma_1(\psi_s,y_2)F_M(v_1)d\big[\big(P^{\epsilon,\otimes 2}-P^{\otimes 2}\big)(z,v_1,v_2,y_1,s,u_1,u_2,y_2)\big]\bigg|\\&
	+ \frac{\|\sigma_1\|_\infty T^{H-\frac{1}{2}}(T-2\rho)}{\rho  }\ex\int_{[\rho, s-\rho]\times\{|v_1|> M\}} |v_1|d\big[P^{\epsilon}(z,v_1,v_2,y_1)+P(z,v_1,v_2,y_1)\big]\\&=:R_{0,1}(M,\lambda,t,\epsilon)+R_{0,2}(M,\lambda,t,\epsilon).
	\end{aligned}	
	\end{equation*}
	\noindent Regarding $R_{0,1}$,  fix $M>0,\lambda\in(0,1/2), t\in[0,T]$ and let $E_{\lambda,t}:=\{ (z,s) \in[0,T]^2 :  z\in [\rho, s-\rho] , s\in[2\rho,t]   \}$.  Since $\lim_{\epsilon\to 0}P^\epsilon=P$ a.s. in $\mathscr{P}([0,T]\times\mathcal{U}_1\times\mathcal{U}_2\times\mathcal{Y}),$ it follows that $\lim_{\epsilon\to 0}P^{\epsilon,\otimes 2}=P^{\otimes 2}$ a.s. in $\mathscr{P}(([0,T]\times\mathcal{U}_1\times\mathcal{U}_2\times\mathcal{Y})^2)$ (see e.g. Theorem 2.8,\cite{94060261}). Furthermore, the map
$\big([0,T]\times\mathcal{U}_1\times\mathcal{U}_2\times\mathcal{Y}\big)^2\ni (z,v_1,v_2,y_1,s,u_1,u_2,y_2)\longmapsto s^{H-\frac{1}{2}}z^{\frac{1}{2}-H}(s-z)^{H-\frac{3}{2}} \sigma_1(\psi_s,y_2)F_M(v_1)\mathds{1}_{E_{\lambda,t}}(s,z)\in\mathcal{X}$
	is bounded with probability $1$ and its discontinuity set $\mathcal{C}_{M,\lambda,t}$ is contained in the triangle $\partial E_{\lambda,t}=\{z=\rho\}\cup\{s=t\}\cup\{ z=s-\rho\}.$ From $\eqref{Lebmarg}$, it follows that $P(\mathcal{C}_{M,\lambda,t})\leq P(\partial E_{\lambda,t})=Leb_{[0,T]^2}(\partial E_{\lambda,t})=0$
	with probability $1$. Thus, from \cite{dupuis2011weak},Theorem A.3.10 along with the dominated convergence theorem we deduce that $\lim_{\epsilon\downarrow 0}R_{0,1}(M,\lambda,t,\epsilon)=0$. From an application of the Portmanteau lemma and the uniform integrability of $P^\epsilon$ (Corollary \ref{UI}) we also have
	\begin{equation*}
	\begin{aligned}
	\lim_{M\to\infty}R_{0,2}(M,\lambda,t,\epsilon)\leq \frac{C_{H,\sigma_1,T}}{\xi\lambda t}\lim_{M\to\infty}\sup_{\epsilon>0}\ex\int_{[0,T]\times\{|v_1|> M\}} |v_1|dP^{\epsilon}(z,v_1,v_2,y_1)=0.
	\end{aligned}
	\end{equation*}
	Therefore, for each $t\in[0,T],\lambda\in(0,\frac{1}{2})$, $\lim_{M\to\infty}\lim_{\epsilon\to0}R_0(\epsilon,\lambda,t)=0$ in probability. Turning to  $R_1$ we have 
	\begin{equation}\label{R1}
	\begin{aligned}\hspace*{-0.2cm}
	&\ex|R_1(\epsilon,\lambda,t)|\\&\leq\ex\int_{[0,\rho]\times\mathcal{U}_1 }\int_{[z,t] \times\mathcal{Y}}s^{H-\frac{1}{2}}z^{\frac{1}{2}-H}(s-z)^{H-\frac{3}{2}}\\&\times \big|\sigma_1(\psi_s,y_2)\big||v_1|d\big[\big(P^{\epsilon,\otimes 2}+P^{\otimes 2}\big)(s,u_1,u_2,y_2,z,v_1,v_2,y_1)\big]
	\\&\leq  C_{\sigma_1}T^{H-\frac{1}{2}}\ex\int_{[0,\rho]\times\mathcal{U}_1 }|v_1|z^{\frac{1}{2}-H}\int_{[z,t]}(s-z)^{H-\frac{3}{2}}dsd\big[\big(P^{\epsilon}+P\big)(z,v_1,v_2,y_1)\big]\\&
	\leq C_{H,\sigma_1,T}\ex\int_{[0,\rho]\times\mathcal{U}_1 }|v_1|z^{\frac{1}{2}-H}(t-z)^{H-\frac{1}{2}}d\big[\big(P^{\epsilon}+P\big)(z,v_1,v_2,y_1)\big]\\&
	\leq C_{H,\sigma_1,T}T^{H-\frac{1}{2}} ( [z^{1+1-2H} ]_0^\rho )^{\frac{1}{2}}
	\rho\ex\bigg(\int_{[0,\rho]\times\mathcal{U}_1 }|v_1|^2d\big[\big(P^{\epsilon}+P\big)(z,v_1,v_2,y_1)\big]\bigg)^{\frac{1}{2}}\\&
	\leq C_{H,\sigma_1,T}\rho^{2-H}\sup_{\epsilon>0}\ex\bigg[\int_{[0,T]\times\mathcal{U}_1 }|v_1|^2dP^{\epsilon}(z,v_1,v_2,y_1)\bigg]^{\frac{1}{2}}\\&\leq  C_{H,\sigma_1,T}T^{2-H}\lambda^{2-H}\sup_{\epsilon>0}\ex\bigg[\int_{[0,T]\times\mathcal{U}_1 }|v_1|^2dP^{\epsilon}(z,v_1,v_2,y_1)\bigg]^{\frac{1}{2}},
	\end{aligned}
	\end{equation}
	where we used the Cauchy-Schwarz inequality and Portmanteau lemma to obtain the last line.
	As for $R_2$,
	\begin{equation}\label{R2}
	\begin{aligned}
	&\ex|R_2(\epsilon,\lambda,t)|\\&\leq \ex\int_{[\rho, t-\rho]\times\mathcal{U}_1 }\int_{[z, z+\rho]\times\mathcal{Y}}s^{H-\frac{1}{2}}z^{\frac{1}{2}-H}(s-z)^{H-\frac{3}{2}}\\&\times \big|\sigma_1(\psi_s,y_2)\big||v_1|d\big[\big(P^{\epsilon,\otimes 2}+P^{\otimes 2}\big)(s,u_1,u_2,y_2,z,v_1,v_2,y_1)\big]\\&
	\leq C_{H,\sigma_1}T^{H-\frac{1}{2}}\ex\int_{[\rho, t-\rho]\times\mathcal{U}_1 }z^{\frac{1}{2}-H}|v_1|
	[(s-z)^{H-1/2}]_{z}^{z+\rho}d\big[\big(P^{\epsilon}+P\big)(z,v_1,v_2,y_1)\big]\\&
	\leq C_{H,\sigma_1,T}\rho^{H-\frac{1}{2}} \bigg(\int_{\rho}^{t-\rho}z^{1-2H}dz\bigg)^{\frac{1}{2}} \ex\bigg(\int_{[0,T]\times\mathcal{U}_1 }|v_1|^2d\big[\big(P^{\epsilon}+P\big)(z,v_1,v_2,y_1)\big]\bigg)^{\frac{1}{2}}\\&
	\leq  C_{H,\sigma_1,T}\lambda^{H-\frac{1}{2}}(T-\rho)^{1-H}\sup_{\epsilon>0}\ex\bigg(\int_{[0,T]\times\mathcal{U}_1 }|v_1|^2dP^{\epsilon}(z,v_1,v_2,y_1)\bigg)^{\frac{1}{2}},
	\end{aligned}
	\end{equation}
	where we applied the Cauchy-Schwarz inequality along with the Portmanteau lemma once again to obtain the last line. Similarly,
	\begin{equation}\label{R3}
	\begin{aligned}
	&\ex|R_3(\epsilon,\lambda,t)|\\&\leq\int_{[t-\rho, t]\times\mathcal{U}_1 }\int_{[z,t]\times\mathcal{Y}}s^{H-\frac{1}{2}}z^{\frac{1}{2}-H}(s-z)^{H-\frac{3}{2}}\\&\times \big|\sigma_1(\psi_s,y_2)\big||v_1|d\big[\big(P^{\epsilon,\otimes 2}+P^{\otimes 2}\big)(s,u_1,u_2,y_2,z,v_1,v_2,y_1)\big]\\&
	\leq C_{H,\sigma_1}T^{H-\frac{1}{2}}\int_{[t-\rho, t]\times\mathcal{U}_1}|v_1|z^{\frac{1}{2}-H}(t-z)^{H-\frac{1}{2}}d\big[\big(P^{\epsilon}+P\big)(z,v_1,v_2,y_1)\big]\\&
	\leq C_{H,\sigma_1,T}\lambda^{H-\frac{1}{2}}(T-\rho)^{1-H}\sup_{\epsilon>0}\ex\bigg[\int_{[0,T]\times\mathcal{U}_1 }|v_1|^2dP^{\epsilon}(z,v_1,v_2,y_1)\bigg]^{\frac{1}{2}}.
	\end{aligned}
	\end{equation}
	In view of the previous arguments, estimates \eqref{R1}-\eqref{R3} and the uniform moment bound in \eqref{occutight} we can take limits, first as as $\epsilon\to 0,$ then as $M\to\infty$, and lastly as $\lambda\to 0,$ to obtain the desired convergence in probability . 
\end{proof}

In view of Lemma \ref{locave} and Proposition \ref{sigmaave}, it follows that, along any convergent subsequence in distribution in $C([0,T];\mathcal{X})$, the dynamics converge pointwise in distribution to the law of the solution of \eqref{limdynamics}. Due to the tightness of the individual terms in $C([0,T];\mathcal{X})$ and uniqueness of the pointwise limit we conclude that this convergence takes place in $C([0,T];\mathcal{X}).$ Thus the limiting pair satisfies $(iii)$ of Definition \ref{viable} with $\Lambda$ as in Theorem \ref{main1}.

\section{Proof of the Large Deviation Principle}\label{LDPsec} In this section we prove the second main result of this paper, Theorem \ref{main2}. Our proof of the LDP for the slow process $X^{\epsilon}$ proceeds in the following steps: First we prove the Laplace Principle upper bound in Section \ref{lapluppersec}. This is a straightforward consequence of Theorem \ref{main1} along with the Portmanteau lemma. In Section \ref{compsec}, we show that the candidate rate function \eqref{ratefun} admits two equivalent ordinary control formulations and that its sublevel sets are compact. These equivalent formulations are then used in Section \ref{LPlbsec}, where we prove the Laplace Principle lower bound in two different cases. Each case comes with different assumptions for the coefficients of the slow-fast system \eqref{model} and our presentation follows an increasing order of generality. As is well known, the Laplace Principle (LP) for a rate function with compact sublevel sets is equivalent to an LDP with the same rate function. The reader is referred to \cite{dupuis2011weak}, Theorems 1.2.1 , 1.2.3 for a proof. Finally, we discuss the assumption for our second proof of the lower bound in Section \ref{assumptionsection} and provide a few examples in which it is satisfied. 	

\subsection{The Laplace Principle upper bound}\label{lapluppersec} \noindent We aim to show that for any continuous, bounded function $h:C([0,T];\mathcal{X})\rightarrow \R$ the following asymptotics hold:
\begin{equation*}\label{laplupper}
\limsup_{\epsilon\to 0}\epsilon\log\ex_{x_0}\big[ e^{-h(X^\epsilon)/\epsilon}      \big]\leq -\inf_{\phi\in C([0,T];\mathcal{X}) }\big[ S_{x_0}(\phi) +h(\phi) \big],
\end{equation*}
where $S_{x_0}:C([0,T];\mathcal{X})\rightarrow[0,\infty]$ is the rate function given in \eqref{ratefun}.  It suffices to verify the above limit along any convergent sequence in $\epsilon$. Such a sequence exists since, for $\epsilon$ small enough,
$$\bigg|\epsilon\log\ex_{x_0}\big[ e^{-h(X^{\epsilon})/\epsilon}\big]\bigg|\leq \sup_{\phi\in C([0,T];\mathcal{X})}\big|h(\phi)\big|<\infty. $$
In light of the variational representation \eqref{varform} there exists, for each fixed $\epsilon>0$, a pair of controls $u^\epsilon=(u^\epsilon_1,u^\epsilon_2)\in\mathcal{A}_b$ such that
\begin{equation*}
\begin{aligned}
\epsilon\log\ex\big[ e^{-h(X^\epsilon)/\epsilon}      \big]&\leq  -\bigg(\ex\bigg[   \frac{1}{2}\|u^\epsilon_1\|^2_{\h_H}+\frac{1}{2}\|u^\epsilon_2\|^2_{\h_{1/2}} +h\big(X^{\epsilon,\eta,u^\epsilon}\big)\bigg]-\epsilon\bigg)\\&= -\bigg(\ex\bigg[   \frac{1}{2}\int_{0}^{T}\big|K_H^{-1}u^\epsilon_1(t)\big|^2dt+\frac{1}{2}\int_{0}^{T}\big|\dot{u}^\epsilon_2(t)\big|^2dt +h\big(X^{\epsilon,\eta,u^\epsilon}\big)\bigg]-\epsilon\bigg).
\end{aligned}	
\end{equation*} 
In fact, from a standard approximation argument we can assume, without loss of generality, that $u^\epsilon\in\mathcal{A}_N$ for $N$ sufficiently large. The reader is referred to the proof of Theorem 4.3 of \cite[pp. 1655-1656, Proof of the lower bound]{boue1998variational} for a sketch of this argument which relies on truncation of the controls and Chebyshev's inequality.

Using this family of controls and associated controlled processes $X^{\epsilon,\eta,u^\epsilon}$ we construct occupation measures $P^\epsilon$ as in \eqref{occupation}. From Theorem \ref{main1} and Prokhorov's theorem, there exists a further subsequence $(X^{\epsilon,\eta,u}, P^\epsilon)$ that converges in distribution in $C([0,T];\mathcal{X})\times\mathscr{P}([0,T]\times\mathcal{U}_1\times\mathcal{U}_2\times\mathcal{Y})$ as $\epsilon\to 0$ to a viable pair $(\psi,P)\in\mathscr{V}_{\Lambda,x_0}$ (see Definition \ref{viable}). The Portmanteau lemma thus furnishes
\begin{equation*}
\begin{aligned}
\limsup_{\epsilon\to 0}&\epsilon\log\ex\big[ e^{-h(X^\epsilon)/\epsilon}      \big]\\&\leq \limsup_{\epsilon\to 0}
-\bigg(\ex\bigg[   \frac{1}{2}\int_{0}^{T}\big|K_H^{-1}u^\epsilon_1(t)\big|^2dt+\frac{1}{2}\int_{0}^{T}\big|\dot{u}^\epsilon_2(t)\big|^2dt+h\big(X^{\epsilon,\eta,u^\epsilon}\big)\bigg]-\epsilon\bigg)
\\&
=-\liminf_{\epsilon\to 0}\ex\bigg[   \frac{1}{2}\int_{[0,T]\times\mathcal{U}_1\times\mathcal{U}_2\times\mathcal{Y}}\big[|u_1|^2+|u_2|^2\big]dP^\epsilon(t,u_1,u_2,y)+h\big(X^{\epsilon,\eta,u^\epsilon}\big)\bigg]\\&
\leq -\ex\bigg[   \frac{1}{2}\int_{[0,T]\times\mathcal{U}_1\times\mathcal{U}_2\times\mathcal{Y}}\big[|u_1|^2+|u_2|^2\big]dP(t,u_1,u_2,y)+h\big(\psi\big)\bigg]\\&
\leq -\ex\bigg[\inf_{P\in\mathscr{V}_{\Lambda,x_0,\psi}}\frac{1}{2}\int_{[0,T]\times\mathcal{U}_1\times\mathcal{U}_2\times\mathcal{Y}}\big[|u_1|^2+|u_2|^2\big]dP(t,u_1,u_2,y)+  h\big(\psi\big) \bigg]\\&
\leq -\inf_{\phi\in C([0,T];\mathcal{X})}\bigg[  \inf_{P\in\mathscr{V}_{\Lambda,x_0,\phi}}\frac{1}{2}\int_{[0,T]\times\mathcal{U}_1\times\mathcal{U}_2\times\mathcal{Y}}\big[|u_1|^2+|u_2|^2\big]dP(t,u_1,u_2,y)+  h\big(\phi\big)         \bigg]\\&
= -\inf_{\phi\in C([0,T];\mathcal{X})}\bigg[  S_{x_0}(\phi)+  h\big(\phi\big)         \bigg].
\end{aligned}	
\end{equation*} 
The proof is complete.
\subsection{An equivalent form of the rate function and compactness of sublevel sets}\label{compsec} We prepare the proof of the Laplace Principle lower bound by showing that the rate function can be represented in an ordinary control formulation. In particular, we show that the relaxed infimization problem \eqref{ratefun} over the space of probability measures is equivalent to an infimization problem over square integrable controls in feedback form.
To this end, let 
\begin{equation*}
\begin{aligned}
&\mathscr{A}^{o}_{\Lambda,\phi,x_0}:=\bigg\{ u=(u_1,u_2):[0,T]\times\mathcal{Y}\rightarrow\mathcal{U}_1\times\mathcal{U}_2:\\&  \int_{0}^{T}\int_{\mathcal{Y}}\big[|u_1(t,y)|^2+u_2(t,y)|^2\big]d\mu(y)dt<\infty, \forall t\in[0,T]
\\&
\phi(t)=x_0+\int_{0}^{t}\int_{\mathcal{Y}}\Lambda_1( \phi(s),u_1(s,y_2),u_2(s,y_2),y_2)\\&+\bigg(\int_{0}^{s}\int_{\mathcal{Y}}\Lambda_2( \phi(s),s,u_1(s,y_2),u_2(s,y_2),y_2,z,u_1(z,y_1),u_2(z,y_1),y_1)d\mu(y_1)dz\bigg)d\mu(y_2)ds
\bigg\}
\end{aligned}
\end{equation*}
and 
\begin{equation}\label{Abar}
\begin{aligned}
\overline{\mathscr{A}^{o}}_{\phi,x_0}=\bigg\{ u_1&\in L^2([0,T];\mathcal{U}_1), u_2\in L^2([0,T]\times\mathcal{Y},dt\otimes d\mu;\mathcal{U}_2): \phi(0)=x_0,
\\& \dot{\phi}_t= \bar{c}(\phi_t)+\overline{\nabla\Psi g}(\phi_t)+\overline{  \big[ \nabla\Psi \tau+\sigma_2(\phi_t,\cdot\big)\big]u_2(t,\cdot)}+\bar{\sigma}_1( \phi_t)\dot{K}_Hu_1(t)\bigg\}.
\end{aligned}
\end{equation}

\begin{lem}\label{ordform} Let $T>0, x_0\in\mathcal{X}$. For all $\phi\in C([0,T];\mathcal{X})$ we have 
	\begin{equation*}
	\begin{aligned}
	(i)\quad	S_{x_0}(\phi)&=\inf_{P\in\mathscr{V}_{\Lambda,x_0,\phi}}\frac{1}{2}\int_{[0,T]\times\mathcal{U}_1\times\mathcal{U}_2\times\mathcal{Y}}\big[|u_1|^2+|u_2|^2\big]dP(t,u_1,u_2,y)\\&=\inf_{(u_1,u_2)\in \mathscr{A}^{o}_{\Lambda,\phi,x_0}}\frac{1}{2}\int_{0}^{T}\int_{\mathcal{Y}}\big[|u_1(t,y)|^2+|u_2(t,y)|^2\big]d\mu(y)dt.
	\end{aligned}
	\end{equation*}
	\begin{equation}\label{ordform2}
	(ii)\quad\quad\quad	S_{x_0}(\phi)=\inf_{(u_1,u_2)\in \overline{\mathscr{A}^{o}}_{\phi,x_0}}\frac{1}{2}\int_{0}^{T}\bigg[|u_1(t)|^2+\int_{\mathcal{Y}}|u_2(t,y)|^2d\mu(y)\bigg]dt,
	\end{equation}
\end{lem}	 

\begin{proof}
	$(i)$ Let $P\in\mathscr{V}_{\Lambda,x_0,\phi}$. In view of \eqref{Pdec} we have $dP(u,v,y,s)=d\Theta(u,v|y,s)d\mu(y)ds$. Hence, we can define $(u_1,u_2):[0,T]\times\mathcal{Y}\rightarrow\mathcal{U}_1\times\mathcal{U}_2$ 
	by $u_i(t,y)=\int_{\mathcal{U}_1\times\mathcal{U}_2} u_i d\Theta(u_1,u_2|y,t),$ $ i=1,2$ and use Jensen's inequality to verify that
	\begin{equation}\label{Jensen}
	\begin{aligned}
	\int_{0}^{T}\int_{\mathcal{Y}}\big[|u_1(t,y)|^2+|u_2(t,y)|^2\big]&d\mu(y)dt\leq \int_{0}^{T}\int_{\mathcal{Y}}\big[|u_1|^2+|u_2|^2\big]d\Theta(u_1,u_2|y,t)d\mu(y)dt\\&=\int_{[0,T]\times\mathcal{U}_1\times\mathcal{U}_2\times\mathcal{Y}}\big[|u_1|^2+|u_2|^2\big]dP(t,u_1,u_2,y)<\infty.
	\end{aligned}
	\end{equation}
	From the definition of $\Lambda_i$ and \eqref{limdynamics} we have for all $t\in[0,T]$ 
	\begin{equation}\label{limdec}
	\begin{aligned}
	&\int_{[0,t]\times\mathcal{U}_1\times\mathcal{U}_2\times\mathcal{Y}}\bigg[\Lambda_1( \phi_s,u_2,v_2,y_2)\\&+\int_{[0,s]\times\mathcal{U}_1\times\mathcal{U}_2\times\mathcal{Y}}\Lambda_2( \phi_s,s,u_2,v_2,y_2,z,u_1,v_1,y_1)dP(z,u_1,v_1,y_1)\bigg]dP(s,u_2,v_2,y_2)\\&
	=\int_{[0,t]\times\mathcal{U}_1\times\mathcal{U}_2\times\mathcal{Y}}\bigg[c\big(\phi_s, y_2\big)+\nabla\Psi(y_2)g\big(\phi_s, y_2\big)\\&+\bigg(\nabla\Psi(y_2)\tau(y_2)+\sigma_2(\phi_s,y_2)\bigg)v_2\bigg]dP(s,u_2,v_2,y_2)\\&+   \frac{c_H}{\Gamma(H-\frac{1}{2})}\int_{[0,t]\times\mathcal{U}_1\times\mathcal{U}_2\times\mathcal{Y}}\int_{[0,s]\times\mathcal{U}_1\times\mathcal{U}_2\times\mathcal{Y}} s^{H-\frac{1}{2}}z^{\frac{1}{2}-H}(s-z)^{H-\frac{3}{2}} \\&\times\sigma_1(\phi_s,y_2)u_1dP(z,u_1,v_1,y_1)dP(s,u_2,v_2,y_2)\\&
	=\int_{0}^{t}\int_{\mathcal{Y}}\bigg[c\big(\phi_s, y_2\big)+\nabla\Psi(y_2)g\big(\phi_s, y_2\big)+\bigg(\nabla\Psi(y_2)\tau(y_2)+\sigma_2(\phi_s,y_2)\bigg)u_2(s,y_2)\bigg]d\mu(y_2)ds\\&
	+\frac{c_H}{\Gamma(H-\frac{1}{2})}\int_{0}^{t}\int_{\mathcal{Y}}\bigg(\int_{0}^{s}\int_{\mathcal{Y}} s^{H-\frac{1}{2}}z^{\frac{1}{2}-H}(s-z)^{H-\frac{3}{2}} \sigma_1(\phi_s,y_2)u_1(z,y_1)d\mu(y_1)dz\bigg)d\mu(y_2)ds\\&
	=\int_{0}^{t}\bigg[\bar{c}(\phi_s)+\overline{\nabla\Psi g}\big(\phi_s\big)+\overline{   \nabla\Psi \tau u_2}(s)+\overline{\sigma_2(\phi_s,\cdot)u_2(s,\cdot)}\bigg]ds
	\\&+\int_{0}^{t}\int_{\mathcal{Y}}\sigma_1(\phi_s,y_2)\bigg(\frac{c_H}{\Gamma(H-\frac{1}{2})}\int_{0}^{s} s^{H-\frac{1}{2}}z^{\frac{1}{2}-H}(s-z)^{H-\frac{3}{2}} \bar{u}_1(z)dz\bigg)d\mu(y_2)ds\\&
	=\int_{0}^{t}\bigg[\bar{c}(\phi_s)+\overline{\nabla\Psi g}(\phi_s)+\overline{  \big( \nabla\Psi \tau+\sigma_2(\phi_s,\cdot\big)\big)u_2(s,\cdot)}\bigg]ds+\int_{0}^{t}\bar{\sigma}_1( \phi_s)\dot{K}_H\bar{u}_1(s)ds\\&
	=\int_{0}^{t}\int_{\mathcal{Y}}\Lambda_1( \phi_s,s,u_1(s,y_2),u_2(s,y_2),y_2)\\&+\bigg(\int_{0}^{s}\int_{\mathcal{Y}}\Lambda_2( \phi_s,s,u_1(s,y_2),u_2(s,y_2),y_2,z,u_1(z,y_1),u_2(z,y_1),y_1)d\mu(y_1)dz\bigg)d\mu(y_2)ds.
	\end{aligned}
	\end{equation}
	Thus, 
	\begin{equation*}
	\begin{aligned}
	\inf_{(u_1,u_2)\in\mathscr{A}^{o}_{\Lambda,\phi,x_0}}\frac{1}{2}\int_{0}^{T}\int_{\mathcal{Y}}\big[|u_1(t,y)|^2+|u_2(t,y)|^2\big]d\mu(y)dt&\leq S_{x_0}(\phi)
	\end{aligned}
	\end{equation*}
	follows by taking the infimum in \eqref{Jensen}. To prove the reverse inequality let $(u_1,u_2)\in \mathscr{A}^{o}_{\Lambda,\phi,x_0}$ and define a measure $$P(A_1\times A_2\times A_3\times A_4)=\int_{A_1}\int_{A_4} \mathds{1}_{A_2}(u_1(t,y))\mathds{1}_{A_3}(u_2(t,y))d\mu(y)dt$$
	for any $A_1\times A_2\times A_3\times A_4\in\mathscr{B}([0,T]\times \mathcal{U}_1\times\mathcal{U}_2\times\mathcal{Y}).$
	Since \[\Theta(A_2\times A_3|y,t):=\mathds{1}_{A_2}(u_1(t,y))\mathds{1}_{A_3}(u_2(t,y))\] is a stochastic kernel on $\mathcal{U}_1\times\mathcal{U}_2$ given $\mathcal{Y}\times[0,T]$ it follows that \eqref{Pdec} holds. Thus, in view of \eqref{limdec}, we deduce that $P\in
	\mathscr{V}_{\Lambda,x_0,\phi}$ and $$S_{x_0}(\phi)\leq \inf_{(u_1,u_2)\in\mathscr{A}^{o}_{\Lambda,\phi,x_0}}\frac{1}{2}\int_{0}^{T}\int_{\mathcal{Y}}\big[|u_1(t,y)|^2+|u_2(t,y)|^2\big]d\mu(y)dt.$$
	$(ii)$ We proceed once again by proving an upper and a lower bound. First, let $(u_1,u_2)\in\mathscr{A}^{o}_{\Lambda,\phi,x_0}$ and define the averaged control $v_1(t)=\int_{\mathcal{Y}}u_1(t,y)d\mu(y).$ By Jensen's inequality it follows that $v_1\in L^2([0,T];\mathcal{U}_1).$ From the 7th line of \eqref{limdec} we see that  
	$$\phi_t= x_0+\int_{0}^{t}\bigg[\bar{c}(\phi_s)+\overline{\nabla\Psi g}(\phi_s)+\overline{  \big( \nabla\Psi \tau+\sigma_2(\phi_s,\cdot)\big) u_2(s,\cdot)}\bigg]ds+\int_{0}^{t}\bar{\sigma}_1( \phi_s)\dot{K}_Hv_1(s)ds.$$
	Hence, $(v_1,u_2)\in\overline{\mathscr{A}^{o}}_{\phi,x_0}$ and 
	\begin{equation*}
	\begin{aligned}
	\frac{1}{2}\int_{0}^{T}\int_{\mathcal{Y}}\big[|u_1(t,y)|^2+&|u_2(t,y)|^2\big]d\mu(y)dt\geq \frac{1}{2}\int_{0}^{T}\bigg[|v_1(t)|^2+\int_{\mathcal{Y}}|u_2(t,y)|^2d\mu(y)\bigg]dt\\&\geq \inf_{(u_1,u_2)\in \overline{\mathscr{A}^{o}}_{\phi,x_0}}\frac{1}{2}\int_{0}^{T}\bigg[|u_1(t)|^2+\int_{\mathcal{Y}}|u_2(t,y)|^2d\mu(y)\bigg]dt,
	\end{aligned}
	\end{equation*}
	where we used (i) and Jensen's inequality. Taking infimum over $\mathscr{A}^{o}_{\phi,x_0}$ concludes the argument.  The reverse inequality follows trivially from  the inclusion  $\overline{\mathscr{A}^{o}}_{\phi,x_0}\subset\mathscr{A}^{o}_{\Lambda,\phi,x_0}.$  The latter holds by identifying any $u\in L^2([0,T];\mathcal{U}_1)$ with a feedback control in $L^2([0,T]\times\mathcal{Y};\mathcal{U}_1)$ that is constant with respect to its second argument.
\end{proof}
\begin{rem} The ordinary controls $u_1,u_2$ correspond to the noises $B^H, W$ respectively and are given a-priori in feedback form, i.e. they depend on both time and the fast variable. However, Lemma \ref{ordform}$(ii)$ shows that it is sufficient to consider controls $u_1$ that only depend on the time variable. This simplification is not possible in the case $H=1/2$ (see e.g. \cite{dupuis2012large}, Theorem 5.2) and is related to the pathwise interpretation of the integral $dB^H$ (see also the discussion in the end of Section \ref{resultssec} above).	 	
\end{rem}
\noindent We conclude this subsection with the following lemma on the compactness of the rate function's sublevel sets. 
\begin{lem}
	Let $x_0\in\mathcal{X}.$ For all $M>0$ the set $S_M=\{ \phi\in C([0,T];\mathcal{X}) :  S_{x_0}(\phi)\leq M\}$ is compact in the topology of uniform convergence.
\end{lem}
\begin{proof} We will use the equivalent form \eqref{ordform2}. Let $\{\phi^n\}_{n\in\N}\subset S_M$ and for each $n\in\N$ choose  $u_n=(u^n_1,u^n_2)\in\overline{\mathscr{A}^{o}}_{\phi^n,x_0}$ such that 
	\begin{equation}\label{L2weak}
	\frac{1}{2}\int_{0}^{T}\bigg[|u^n_1(t)|^2+\int_{\mathcal{Y}}|u^n_2(t,y)|^2d\mu(y)\bigg]dt\leq M+\frac{1}{n}.
	\end{equation}
	It follows that $\{u_n\}_{n\in\N}\subset L^2([0,T];\mathcal{U}_1)\oplus L^2([0,T]\times\mathcal{Y},dt\otimes d\mu;\mathcal{U}_2) $ is relatively compact in the weak $L^2$ topology. Thus, up to subsequences, there exists a weak $L^2$-limit $u=(u_1,u_2)$ that also satisfies \eqref{L2weak} due to the lower semi-continuity of the norm.  Since $\phi^n$ satisfies
	\begin{equation*}
	\phi^n_t=x_0+\int_{0}^{t}\bigg(\bar{c}(\phi^n_s)+\overline{\nabla\Psi g}(\phi^n_s)+\overline{  \big[ \nabla\Psi \tau+\sigma_2(\phi^n_s,\cdot\big)\big]u^n_2(s,\cdot)}+\bar{\sigma}_1( \phi^n_s)\dot{K}_Hu^n_1(s)\bigg) ds,\; t\in[0,T],
	\end{equation*}
	the coefficients are Lipschitz, $\nabla\Psi \tau,\sigma_2,\sigma_1$ are bounded and the linear operator $\dot{K}_H:L^2\rightarrow L^2$ is continuous in the norm topology (hence weakly continuous; see also Lemma \ref{udotL2} below), we can apply Gr\"onwall's inequality along with an Arzel\`a-Ascoli argument to show that the family $\{\phi^n\}\subset C([0,T];\mathcal{X})$ is relatively compact. Passing, if necessary, to a further subsequence it follows from uniqueness of solutions that any limit point $\phi$ of $\{\phi^n\}_{n\in\N}$ satisfies
	\begin{equation*}
	\phi_t=x_0+\int_{0}^{t}\bigg(\bar{c}(\phi_s)+\overline{\nabla\Psi g}(\phi_s)+\overline{  \big[ \nabla\Psi \tau+\sigma_2(\phi_s,\cdot\big)\big]u_2(s,\cdot)}+\bar{\sigma}_1( \phi_s)\dot{K}_Hu_1(s)\bigg) ds.
	\end{equation*}
	This implies that $(u_1,u_2)\in\overline{\mathscr{A}^{o}}_{\phi,x_0}$
	and 
	
	\begin{equation*}\label{L2weak2}
	S_{x_0}(\phi)\leq \frac{1}{2}\int_{0}^{T}\bigg[|u_1(t)|^2+\int_{\mathcal{Y}}|u_2(t,y)|^2d\mu(y)\bigg]dt\leq M.
	\end{equation*}
	At this point we have shown that the sublevel set $S_M$ is relatively compact. In order to conclude it remains to show that it is closed. Hence, it suffices to show that the rate function $S_{x_0}$ is lower semicontinuous. To this end, let $\{\phi^n\}_{n\in\N}\subset C([0,T];\mathcal{X})$ and $u^n\in \overline{\mathscr{A}^{o}}_{\phi^n,x_0}$ and assume that $(\phi^n,u^n)\rightarrow (\phi,u)$ in $C([0,T];\mathcal{X})\times L^2([0,T];\mathcal{U}_1)\oplus L^2([0,T]\times\mathcal{Y},dt\otimes d\mu;\mathcal{U}_2)$ where the latter is endowed with the weak topology. Assuming without loss of generality (otherwise there is nothing to prove) that $\liminf_{n\to\infty}S_{x_0}(\phi^n)=M<\infty,$ we can pass to a subsequence that satisfies \eqref{L2weak} and 
	\begin{equation*}
	S_{x_0}(\phi^n)\geq \frac{1}{2}\int_{0}^{T}\bigg[|u^n_1(t)|^2+\int_{\mathcal{Y}}|u^n_2(t,y)|^2d\mu(y)\bigg]dt-\frac{1}{n}.
	\end{equation*}
	From our previous discussion, there exists a subsequence $(\phi^n,u^n)$ that converges to a pair $(\phi', u')$ such that $u'\in\overline{\mathscr{A}^{o}}_{\phi',x_0}$ and by uniqueness of the limit we must have $(\phi', u')=(\phi, u).$ Thus, from the lower semicontinuity of the  norms we obtain 
	\begin{equation*}
	\begin{aligned}
	\liminf_{n\to\infty}S_{x_0}(\phi^n)&\geq \liminf_{n\to\infty}\frac{1}{2}\int_{0}^{T}\bigg[|u^n_1(t)|^2+\int_{\mathcal{Y}}|u^n_2(t,y)|^2d\mu(y)\bigg]dt\\&
	\geq \frac{1}{2}\int_{0}^{T}\bigg[|u_1(t)|^2+\int_{\mathcal{Y}}|u_2(t,y)|^2d\mu(y)\bigg]dt\\&
	\geq \inf_{u\in\overline{\mathscr{A}^{o}}_{\phi,x_0}} \frac{1}{2}\int_{0}^{T}\bigg[|u_1(t)|^2+\int_{\mathcal{Y}}|u_2(t,y)|^2d\mu(y)\bigg]dt\\&
	=S_{x_0}(\phi).
	\end{aligned}
	\end{equation*}
	The proof is complete.	\end{proof}

\subsection{The Laplace Principle lower bound}\label{LPlbsec}  We wish to show that for all continuous, bounded $h:C([0,T];\mathcal{X})\rightarrow\R,$ 

\begin{equation}\label{laplower}
\liminf_{\epsilon\to 0}\epsilon\log\ex_{x_0}\big[ e^{-h(X^\epsilon)/\epsilon}      \big]\geq -\inf_{\phi\in C([0,T];\mathcal{X}) }\big[ S_{x_0}(\phi) +h(\phi) \big].
\end{equation}
We provide the proof in two cases that depend upon the generality of the model and require different arguments.
\subsubsection{\textbf{Case 1}: $\mathbf{b=\sigma_2=0, \bar{\sigma}_1\neq 0}$} In view of \eqref{ordform2} and \eqref{Abar}, the rate function takes the form
\begin{equation}\label{smallnoiseform}
S_{x_0}(\phi)=\inf_{u_1\in \overline{\mathscr{A}^{o}}_{\phi,x_0}}\frac{1}{2}\int_{0}^{T}|u_1(t)|^2dt=\inf_{u\in \mathfrak{A}^{o}_{\phi,x_0}}\frac{1}{2}\|u\|^2_{\h_H},
\end{equation} 
where
\begin{equation}\label{smallnoiscon}
\begin{aligned}
\mathfrak{A}^{o}_{\phi,x_0}=\bigg\{ u\in \h_H: \phi(0)=x_0,\; \dot{\phi}_t= \bar{c}\big(\phi_t\big)+\bar{\sigma}_1\big( \phi_t\big)\dot{u}(t)\bigg\}.
\end{aligned}
\end{equation}
Note that, in this case,  $S_{x_0}$ coincides with the Large Deviations rate function for the family $\{X^\epsilon\}_{\epsilon}$ of solutions to the small-noise averaged Young SDE
\begin{equation*}
\left \{\begin{aligned}
& dX^\epsilon_t=\bar{c}\big(X^\epsilon_t\big)dt+\sqrt{\epsilon}\bar{\sigma}_1\big(X^\epsilon_t\big)dB^H_t,\\&
X^{\epsilon}_0=x_0.
\end{aligned} \right.
\end{equation*}
In order to prove the lower bound let $\delta>0$ and choose $\phi\in C([0,T];\mathcal{X})$ such that 
\begin{equation*}
S_{x_0}(\phi)+h(\phi)<\inf_{\psi\in C([0,T];\mathcal{X}) }\big[ S_{x_0}(\psi) +h(\psi) \big]+\frac{\delta}{2}.
\end{equation*}
Next let $v\in\mathfrak{A}^{o}_{\phi,x_0}$ such that 
\begin{equation*}
\frac{1}{2}\|v\|^2_{\h_H}<S_{x_0}(\phi)+\frac{\delta}{2}.
\end{equation*}
Considering the deterministic control $\tilde{u}=(v,0)$ and the slow process $X^{\epsilon,\eta,\tilde{u}}$ controlled by $\tilde{u},$ it follows from uniqueness of the limiting deterministic dynamics and standard averaging theory that $X^{\epsilon,\eta,\tilde{u}}\rightarrow \phi,$ as $\epsilon\to 0,$ in distribution in $C([0,T];\mathcal{X}).$ Thus, in view of the variational representation \eqref{varrepcon} we have 
\begin{equation*}
\begin{aligned}
\limsup_{\epsilon\to 0} -\epsilon\log\ex\big[ e^{-h(X^\epsilon)/\epsilon}      \big]&\leq \limsup_{\epsilon\to 0}  \ex\bigg[   \frac{1}{2}\|v\|^2_{\h_H} +h\big(X^{\epsilon,\eta,\tilde{u}}\big)\bigg]\\&
= \frac{1}{2}\|v\|^2_{\h_H}+\limsup_{\epsilon\to 0}\ex\big[ h\big(X^{\epsilon,\eta,\tilde{u}}\big)\big]\\&
=\frac{1}{2}\|v\|^2_{\h_H}+h(\phi)\\&\leq S_{x_0}(\phi)+h(\phi)+\frac{\delta}{2}<\inf_{\psi\in C([0,T];\mathcal{X}) }\big[ S_{x_0}(\psi) +h(\psi) \big]+\delta.
\end{aligned}
\end{equation*}
Since $\delta$ is arbitrary, \eqref{laplower} follows.
\subsubsection{\textbf{Case 2}:  $\mathbf{\overline{QQ^T}(x)}$ \textbf{is uniformly non-degenerate}} Before we proceed to the proof of the lower bound in this more general setting, we shall introduce a few quantities of interest for the study of the variational problem \eqref{ordform2}, \eqref{Abar}. First, define a family of multiplication operators $\{\bar{\Sigma}_1(\phi)\}_{\phi\in C([0,T];\mathcal{X})}\subset \mathscr{L}(L^2([0,T];\mathcal{U}_1);L^2([0,T];\mathcal{X})),$ by
\begin{equation}\label{multop}
\bar{\Sigma}_1(\phi)[u](t):=\bar{\sigma}_1(\phi_t)u(t)\;, t\in[0,T].
\end{equation}
With this notation we can rewrite, for each $H\in(1/2,1),\phi\in C([0,T];\mathcal{X}),$ the operator $\mathcal{Q}_H(\phi)$ \eqref{QH1} as
\begin{equation*}
\begin{aligned}
\mathcal{Q}_H(\phi)[(u_1,u_2)](t)&:= \bar{\Sigma}_1(\phi)\big[\dot{K}_Hu_1\big](t)+\overline{Q(\phi_t,\cdot)u_2(t,\cdot)},
\end{aligned}		
\end{equation*}
with $Q, \dot{K}_H$ as in \eqref{Qdef1}, \eqref{Kdot} respectively.
Note that $\mathcal{Q}_H(\phi)$ is a bounded linear operator and $\sup_{\phi}\|\mathcal{Q}_H(\phi)\|_{\mathscr{L}}<\infty$ since $\bar{\sigma}_1, \sigma_2,\nabla\Psi, \tau$ are bounded (see Conditions \ref{C2}, \ref{C4}, \ref{C7}) and $\dot{K}_H:L^2([0,T];\mathcal{U}_1)\rightarrow L^2([0,T];\mathcal{X})$ is a bounded operator (see Lemma \ref{udotL2} below).  Moreover, it is straightforward to verify that the adjoint operator $(\mathcal{Q}_H(\phi))^*:L^2([0,T];\mathcal{X})\rightarrow L^2([0,T];\mathcal{U}_1)\oplus L^2([0,T]\times\mathcal{Y}, dt\otimes d\mu;\mathcal{U}_2)$ is given by
\begin{equation}\label{Qadjoint}
\big(\mathcal{Q}_H(\phi)\big)^*[\psi](t,y)=\big( \dot{K}_H^*\bar{\Sigma}^*_1(\phi)[\psi](t) ,     Q^T(\phi_t,y)\psi_t    \big)\;,\;\; (t,y)\in[0,T]\times\mathcal{Y}.
\end{equation}
In view of \eqref{QH1}, the rate function \eqref{ordform2}, \eqref{Abar} can be expressed in the form
\begin{equation}\label{lowerboundratefunction}
S_{x_0}(\phi)=\inf_{u=(u_1,u_2)\in \mathcal{Q}^{-1}_H(\phi)\big[\dot{\phi}- \bar{c}(\phi)-\overline{\nabla\Psi g}(\phi)\big] }\frac{1}{2}\big\|u\big\|^2_{L^2([0,T];\mathcal{U}_1)\oplus L^2([0,T]\times\mathcal{Y}, dt\otimes d\mu;\mathcal{U}_2)},
\end{equation}
for all $\phi\in C([0,T];\mathcal{X})$ such that $\phi_0=x_0$ and $\dot{\phi}- \bar{c}(\phi)-\overline{\nabla\Psi g}(\phi)\in Range(\mathcal{Q}_H(\phi))\subset L^2([0,T];\mathcal{X})$  and $S_{x_0}(\phi)=\infty$ otherwise. 

In the following lemma we prove the existence of exact minimizers satisfying the optimal control problem \eqref{lowerboundratefunction}.
\begin{lem}\label{minimizerlem} Let $H\in[1/2,1), \phi\in C([0,T];\mathcal{X})$ such that $S_{x_0}(\phi)<\infty.$ The  self-adjoint operator 
	$\mathcal{Q}_H(\phi)\mathcal{Q}_H(\phi)^*\in\mathscr{L}(L^2([0,T];\mathcal{X}))$  has a bounded inverse that satisfies
	\begin{equation}\label{inversebnd}
	\sup_{ \phi\in C([0,T];\mathcal{X})}\big\| \big[\mathcal{Q}_H(\phi)\mathcal{Q}_H(\phi)^*\big]^{-1}\big\|_{\mathscr{L}(L^2([0,T];\mathcal{X}))}<\infty.
	\end{equation}
	Moreover, $\mathcal{Q}_H(\phi)$ has a bounded right inverse 	$\mathcal{Q}^+_H(\phi):=\mathcal{Q}_H(\phi)^*\big[\mathcal{Q}_H(\phi)\mathcal{Q}_H(\phi)^*\big]^{-1}.$ As a consequence, the infimum in \eqref{lowerboundratefunction} is attained by
	\begin{equation}\label{minimizers}
	\begin{aligned}
	u^*(t,y)&=(u_1^*(t),u_2^*(t,y))\\&=\mathcal{Q}^+_H(\phi)[\dot{\phi}- \bar{c}(\phi)-\overline{\nabla\Psi g}(\phi)](t,y)\\&=\bigg(\dot{K}_H^*\bar{\Sigma}^*_1(\phi)\big[\mathcal{Q}_H(\phi)\mathcal{Q}_H(\phi)^*\big]^{-1}\big[\dot{\phi}- \bar{c}(\phi)-\overline{\nabla\Psi g}(\phi)\big](t),  
	\\& Q^T(\phi_t,y)\big[\mathcal{Q}_H(\phi)\mathcal{Q}_H(\phi)^*\big]^{-1} \big[\dot{\phi}- \bar{c}(\phi)-\overline{\nabla\Psi g}(\phi)\big](t)            \bigg).
	\end{aligned}
	\end{equation}  
\end{lem}
\begin{proof}
	Let $\psi\in L^2([0,T];\mathcal{X}).$ In view of \eqref{Qadjoint} we have 
	\begin{equation*}
	\begin{aligned}
	\langle \mathcal{Q}_H(\phi)\mathcal{Q}_H(\phi)^*\psi, \psi\rangle_{L^2([0,T];\mathcal{X})}&=\big\| \mathcal{Q}_H(\phi)^*\psi \big\|^2_{L^2([0,T];\mathcal{U}_1)\oplus L^2([0,T]\times\mathcal{Y}, dt\otimes d\mu;\mathcal{U}_2)}\\&
	=\big\| \dot{K}_H^*\bar{\Sigma}^*_1(\phi)[\psi]  \big\|^2_{L^2([0,T];\mathcal{U}_1)}+\int_{0}^{T}\int_{\mathcal{Y}}\big| Q^T(\phi_t,y)\psi_t   \big|^2d\mu(y)dt\\&
	\geq \int_{0}^{T}\int_{\mathcal{Y}}\big\langle Q^T(\phi_t,y)\psi_t, Q^T(\phi_t,y)\psi_t  \big\rangle d\mu(y)dt\\&
	= \int_{0}^{T}\int_{\mathcal{Y}}\big\langle \psi_t, Q(\phi_t,y)Q^T(\phi_t,y)\psi_t  \big\rangle d\mu(y)dt\\&
	= \int_{0}^{T}\big\langle \psi_t, \overline{Q(\phi_t)Q^T(\phi_t)}\psi_t  \big\rangle dt\\&
	\geq c\int_{0}^{T}\big| \psi_t \big|^2dt=c\|\psi\|^2_{L^2([0,T];\mathcal{X})},
	\end{aligned}
	\end{equation*}
	where we used the uniform non-degeneracy of $\overline{QQ^T}(x)$ to obtain the last line. Thus, for any $\psi\neq 0,$
	\begin{equation}\label{lowernorm}
	\begin{aligned}
	\big\| \mathcal{Q}_H(\phi) \mathcal{Q}_H(\phi)^*\psi\big\|_{L^2([0,T];\mathcal{X})}    &=\sup_{\|\tilde{\psi}\|_{L^2}\leq 1}\big| \blangle  \mathcal{Q}_H(\phi) \mathcal{Q}_H(\phi)^*\psi,\tilde{\psi} \brangle_{L^2([0,T];\mathcal{X})}\big|\\&	\geq \bigg\langle  \mathcal{Q}_H(\phi) \mathcal{Q}_H(\phi)^*\psi, \frac{\psi}{\|\psi\|_{L^2}}\bigg\rangle_{L^2([0,T];\mathcal{X})}\geq c\|\psi\|_{L^2([0,T];\mathcal{X})}.
	\end{aligned}
	\end{equation} 
	This implies that the self-adjoint operator $ \mathcal{Q}_H(\phi) \mathcal{Q}_H(\phi)^*$ is injective and has a closed range. Therefore, \begin{equation*}
	\begin{aligned}
	\mathcal{Q}_H(\phi) &\mathcal{Q}_H(\phi)^*\big[L^2([0,T];\mathcal{X})\big]=\overline{ \mathcal{Q}_H(\phi) \mathcal{Q}_H(\phi)^*\big[L^2([0,T];\mathcal{X})\big]}^{L^2}\\&=\ker( [\mathcal{Q}_H(\phi) \mathcal{Q}_H(\phi)^*]^* )^\perp=\ker( \mathcal{Q}_H(\phi) \mathcal{Q}_H(\phi)^* )^\perp=\{0\}^\perp=L^2([0,T];\mathcal{X}),
	\end{aligned}
	\end{equation*} i.e. $\mathcal{Q}_H(\phi)\mathcal{Q}_H(\phi)^*$ is bijective. By virtue of the inverse mapping theorem, its inverse is bounded. The uniform bound of the operator norm then follows from \eqref{lowernorm}. The fact that $\mathcal{Q}^+_H(\phi)$ is a right inverse is immediate from its definition. The final statement follows since $u^*$ is the minimal-norm solution of the equation $\mathcal{Q}_H(\phi)[u]=\big[\dot{\phi}- \bar{c}(\phi)-\overline{\nabla\Psi g}(\phi)\big] $.
\end{proof}
\noindent We are now ready to prove the Laplace principle lower bound. As in Case 1, let $\delta>0$ and $\phi\in C([0,T];\mathcal{X})$ such that 
\begin{equation}\label{phichoice}
\begin{aligned}
\frac{1}{2}\int_{0}^{T}\bigg[|u^*_1(t)|^2&+\int_{\mathcal{Y}}|u^*_2(t,y)|^2d\mu(y)\bigg]dt+h(\phi)\\&=S_{x_0}(\phi)+h(\phi)<\inf_{\psi\in C([0,T];\mathcal{X})}\big[S_{x_0}(\psi)+h(\psi)\big]+\frac{\delta}{2}.
\end{aligned}
\end{equation}
In view of \eqref{minimizers} and Condition \ref{C7}, the deterministic optimal control $u_2^*$ is uniformly bounded in both of its arguments and (globally) Lipschitz continuous in $y$, uniformly in $t$. Moreover, we can approximate it in $L^2([0,T]\times\mathcal{Y})$ by a control that is continuous in $t\in[0,T]$. We prove this claim in the next lemma.
\begin{lem}\label{tcont} Let $h:C([0,T];\mathcal{X})\rightarrow\R$ be continuous and bounded, $u_1^*, \phi$ as in \eqref{minimizers}, \eqref{phichoice} respectively. There exists a deterministic control $v_2^*\in L^2([0,T]\times\mathcal{Y})$ that is continuous in the first variable and Lipschitz continuous in the second variable, uniformly in the first variable, and such that for any $\delta>0$
	\begin{equation}\label{tcontbnd}
	\frac{1}{2}\int_{0}^{T}\bigg[|u^*_1(t)|^2+\int_{\mathcal{Y}}|v^*_2(t,y)|^2d\mu(y)\bigg]dt+h(\tilde{\phi})<\inf_{\psi\in C([0,T];\mathcal{X})}\big[S_{x_0}(\psi)+h(\psi)\big]+\delta,
	\end{equation}
	where $\tilde{\phi}$ solves $\phi(t)=x_0+\int_{0}^{t}\big[\bar{c}(\phi_s)+\overline{\nabla\Psi g}(\phi_s)+\mathcal{Q}_H(\phi)(u^*_1,v_2^*)(s)\big]ds, t\in[0,T].$ 	
\end{lem}
\begin{proof} In view of  \eqref{inversebnd},  $[\mathcal{Q}_H(\phi)\mathcal{Q}_H(\phi)^*]^{-1}\big[\dot{\phi}- \bar{c}(\phi)-\overline{\nabla\Psi g}(\phi)\big]\in L^2([0,T];\mathcal{X}).$ By density, there exists a sequence $\{\psi_n\}$ of continuous functions that converges to  $[\mathcal{Q}_H(\phi)\mathcal{Q}_H(\phi)^*]^{-1}$ $\big[\dot{\phi}- \bar{c}(\phi)-\overline{\nabla\Psi g}(\phi)\big]$ in $L^2([0,T];\mathcal{X}).$ Thus, by letting $v^*_{2,n}(t,y):=Q^T(\phi_t,y)\psi_n(t)$ we have 1) $v^*_{2,n}$ satisfies the continuity properties of the statement and 2)
	\begin{equation*}
	\begin{aligned}
	&\big\| u^*_{2}-v^*_{2,n}   \big\|_{L^2([0,T]\times\mathcal{Y})}\\&	\leq \big\|Q^T(\phi_\cdot,\cdot)\big\|_{L^2([0,T]\times\mathcal{Y})}\big\|\psi_n-     [\mathcal{Q}_H(\phi)\mathcal{Q}_H(\phi)^*]^{-1}\big[\dot{\phi}- \bar{c}(\phi)-\overline{\nabla\Psi g}(\phi)\big]\big\|_{L^2([0,T];\mathcal{X})}\longrightarrow 0,
	\end{aligned}
	\end{equation*}
	as $n\to\infty.$ Now consider the sequence $\{\phi^n\}_{n\in\N}$ of solutions to the limiting dynamics controlled by $(u^*_1, v^*_{2,n})$. From the Lipschitz continuity and boundedness of the coefficients along with the square integrability of $u^*_1$ we have 
	\begin{equation*}
	\begin{aligned}
	&\big|\phi^n_t-\phi_t\big|\\&\leq \int_{0}^{T}\big|\bar{c}(\phi^n_t)-\bar{c}(\phi_t)\big|dt+\int_{0}^{T}\big|\overline{\nabla\Psi g}(\phi^n_t)-\overline{\nabla\Psi g}(\phi_t)\big|dt\\&+\int_{0}^{T}\big|\overline{\sigma}_1(\phi^n_t)-\overline{\sigma}_1(\phi_t)\big|\big|\dot{K}_Hu^*_1(t)  \big|dt\\&
	+ \int_{0}^{T}\big|\overline{Q(\phi_t,\cdot)v^*_{2,n}}(t) -\overline{Q(\phi^n_t,\cdot)v^*_{2,n}}(t)\big|dt+\int_{0}^{T}\big|[\overline{Q(\phi_t,\cdot)u_2^*}(t)-\overline{Q(\phi_t,\cdot)v_{2,n}^*}(t)\big|dt\\&
	\leq C_1\int_{0}^{T}\big|\phi^n_t-\phi_t\big|dt+C_2\bigg(\int_{0}^{T}\big|\phi^n_t-\phi_t\big|^2dt\bigg)^{\frac{1}{2}}\big\|\dot{K}_Hu_1^*\big\|_{L^2}\\&+C_3\sup_{n\in\N}\sup_{(t,y)\in[0,T]\times\mathcal{Y}} \big|v^*_{2,n}(t,y)\big|\int_{0}^{T}\big|\phi^n_t-\phi_t\big|dt
	+C_4\big\|v^*_{2,n}-u_2^*\big\|_{L^2([0,T]\times\mathcal{Y})}.
	\end{aligned}
	\end{equation*}
	Squaring and applying Gr\"onwall's inequality we obtain
	\begin{equation*}
	\begin{aligned}
	\big\| |\phi^n-\phi\big\|^2_{C([0,T];\mathcal{X})}\leq C_T\big\|v^*_{2,n}-u_2^*\big\|^2_{L^2([0,T]\times\mathcal{Y})}\longrightarrow 0,
	\end{aligned}
	\end{equation*}
	as $n\to \infty$. Since $h$ is continuous we have $h(\phi^n)\longrightarrow h(\phi)$ and the proof is complete upon choosing $n_0$ large enough, $\tilde{\phi}=\phi^{n_0}$ and $v_2^*=v^*_{2,n_0}.$
\end{proof}
The next step in the proof of the lower bound is to construct a sequence of feedback stochastic controls that approximate the left-hand side in \eqref{tcontbnd}.	
To this end,  let $Y$ denote the It\^o diffusion with infinitesimal generator $\mathcal{L}$ \eqref{generator}. In particular, $Y$ solves the SDE
\begin{equation}\label{ergodicdiff}
dY_t= f(Y_t)dt+\tau(Y_t)dW_t 
\end{equation}
In view of Conditions \ref{C2}, \ref{C3}, $Y$ is strongly mixing and its unique invariant measure is given by $\mu.$ Next let $y$ be a $\mu$-distributed random initial condition, $Y^y$ denote the corresponding stationary ergodic process and $Y^{y,\eta}_t:=Y^{y}_{t/\eta}.$ Guided by the explicit form of the optimal controls given in the previous lemmas, we define a sequence $v_2^\epsilon$ by
\begin{equation}
v^\epsilon_2(t):=v_2^*\big( [t/\eta]\eta, Y^{y,\eta}_t\big)\;, t\in[0,T],
\end{equation}
where $[\cdot]$ denotes the floor function. By the convergence of $[t/\eta]\eta\rightarrow t,$ as $\epsilon\to 0,$ the continuity properties of $v_2^*$ and the ergodic theorem we see that 
\begin{equation}
\ex\int_{0}^{T}\big|v^\epsilon_2(t)\big|^2dt
\longrightarrow \int_{0}^{T}\int_{\mathcal{Y}}\big|v^*_2(t,y)\big|^2d\mu(y)dt,\;\epsilon\to 0.
\end{equation}
Moreover, let $\tilde{u}^\epsilon=(K_H[u_1^*],\int_{0}^{\cdot}v^\epsilon_2(s)ds)$ (note that, almost surely, $\tilde{u}^\epsilon\in\h_H\oplus\h_{1/2}$ \eqref{CMdef} hence it is an admissible control) and consider the slow motion $X^{\epsilon,\eta,\tilde{u}^\epsilon}$ controlled by $\tilde{u}^\epsilon.$ By standard homogenization theory and uniqueness of the limiting (deterministic) dynamics we have $\lim_{\epsilon\to 0}X^{\epsilon,\eta,\tilde{u}^\epsilon}=\tilde{\phi}$ in distribution in $C([0,T];\mathcal{X}).$ Invoking the variational representation \eqref{varform} once again, we conclude that
\begin{equation*}
\begin{aligned}
\limsup_{\epsilon\to 0} -\epsilon\log\ex\big[ e^{-h(X^\epsilon)/\epsilon}      \big]&\leq \limsup_{\epsilon\to 0}  \ex\bigg[   \frac{1}{2}\|\tilde{u}^\epsilon\|^2_{\h_H\oplus\h_{1/2}} +h\big(X^{\epsilon,\eta,\tilde{u}^\epsilon}\big)\bigg]\\&
= \limsup_{\epsilon\to 0}  \ex\bigg[\frac{1}{2}\bigg(\int_{0}^{T}\big| u_1^*(t)\big|^2dt +\int_{0}^{T}\big| v_2^\epsilon(t)\big|^2dt   \bigg)+ h\big(X^{\epsilon,\eta,\tilde{u}}\big)\bigg]\\&
=   \frac{1}{2}\bigg(\int_{0}^{T}\big| u_1^*(t)\big|^2dt+\int_{0}^{T}\int_{\mathcal{Y}}\big|v^*_2(t,y)\big|^2d\mu(y)dt\bigg)+h(\tilde{\phi})\\&
<\inf_{\psi\in C([0,T];\mathcal{X}) }\big[ S_{x_0}(\psi) +h(\psi) \big]+\delta,
\end{aligned}
\end{equation*}
where we used  \eqref{tcontbnd} to obtain the last line. Since $\delta$ is arbitrary, the proof of the lower bound is complete. 
\subsection{On the assumption of the lower bound}\label{assumptionsection} Our second proof of the Laplace principle lower bound works under the assumption that the averaged matrix $\overline{QQ^T}$ is uniformly non-degenerate. Below we provide a few examples which demonstrate that this condition can be satisfied by imposing a number of different assumptions on the coefficients.\\ \\
$\mathbf{1)}$  $b=0$ and the diffusion matrix $\sigma_2\sigma_2^T$ is uniformly non-degenerate. Thus, $\nabla\Psi=0$ and $QQ^T$ reduces to $\sigma_2\sigma_2^T.$  \\ \\
$\mathbf{2)}$ The fast motion evolves on the torus $\mathcal{Y}=\mathbb{T}^{d-m}$ and the diffusion coefficient $\sigma_2(x,y)=\tilde{\sigma_2}(x)\tau(y)$ for some matrix-valued function $\tilde{\sigma}_2$ such that $\tilde{\sigma_2}\tilde{\sigma_2}^T$ is uniformly non-degenerate. The first condition is equivalent to taking,  for all $x\in\mathcal{X},$  $c(x,\cdot), \sigma_1(x,\cdot), g(x,\cdot),$ $b, f,\tau$ to be periodic of the same period in every direction and the drifts $f, b$ to be continuous bounded functions. Since the phase space $\mathcal{Y}$ is compact, the Krylov-Bogolyubov theorem (see \cite{da1996ergodicity}, Section 3.1 for Markov processes and \cite{katok1997introduction},Theorem 4.1.1 for more general dynamical systems) asserts the existence of a unique invariant measure $\mu$ defined in $\mathscr{B}(\mathcal{Y}).$ The measure $\mu$ is absolutely continuous with respect to Lebesgue measure on the torus and the  invariant density $\rho_\infty$ is strictly positive and twice differentiable. Moreover, the Poisson equation $\{\mathcal{L}\Psi=-b, \int_{\mathcal{Y}}bd\mu=0\},$ equipped with periodic boundary conditions, has a unique strong solution $\Psi$ that is twice continuously differentiable with bounded derivatives up to second order. All the results of this paper continue to hold in this case.
From the uniform non-degeneracy of $\tau\tau^T$ we have, for all $\xi\in\mathcal{X},$ 
\begin{equation*}
\begin{aligned}
\blangle Q(x,y)Q^T(x,y)\xi, \xi\brangle&=\blangle \tau(y)\tau^T(y)[\nabla\Psi(y)+\tilde{\sigma_2}(x)]^T\xi, [\nabla\Psi(y)+\tilde{\sigma_2}(x)]^T\xi\brangle\\&\geq C\blangle [\nabla\Psi(y)+\tilde{\sigma_2}(x)]^T\xi, [\nabla\Psi(y)+\tilde{\sigma_2}(x)]^T\xi\brangle.
\end{aligned}	
\end{equation*}
Since $\rho_\infty>0$ and the integral of the derivative of a smooth periodic function is $0$ we apply the Cauchy-Schwarz inequality to obtain
\begin{equation*}
\begin{aligned}
\blangle \overline{Q(x)Q^T(x)}\xi, \xi\brangle&\geq C\bigg(\int_{\mathcal{Y}} \big\|[\nabla\Psi(y)+\tilde{\sigma}_2(x)]\xi\big\|^2\rho_\infty(y)dy\bigg)\frac{\int_{\mathcal{Y}}\rho^{-1}_\infty(y) dy   }{\int_{\mathcal{Y}}\rho_\infty^{-1}(y)dy}\\&
\geq C\bigg(\int_{\mathcal{Y}} \big\|[\nabla\Psi(y)+\tilde{\sigma}_2(x)]\xi\big\|\rho^{1/2}_\infty(y)\rho^{-1/2}_\infty(y)dy\bigg)^{2}\bigg(\int_{\mathcal{Y}}\rho_\infty^{-1}(y)dy\bigg)^{-1}\\&
\geq C\bigg\|\int_{\mathcal{Y}} [\nabla\Psi(y)+\tilde{\sigma}_2(x)]\xi dy\bigg\|^{2}\bigg(\int_{\mathcal{Y}}\rho_\infty^{-1}(y)dy\bigg)^{-1}\\&
=C\big\|\tilde{\sigma}_2(x)\xi \big\|^{2}\bigg(\int_{\mathcal{Y}}\rho_\infty^{-1}(y)dy\bigg)^{-1}
\geq C'\bigg(\int_{\mathcal{Y}}\rho_\infty^{-1}(y)dy\bigg)^{-1}\|\xi\|^2,
\end{aligned}
\end{equation*}
where we used the uniform non-degeneracy of $\tilde{\sigma}_2$ in the last line. A very similar argument can be found in \cite{pavliotis2008multiscale}, Theorem 13.5.\\ \\
$\mathbf{3)}$  $\mathcal{Y}=\mathcal{X}=\R^{d}$, $\tau=\sqrt{2\alpha}Id, f(y)=-\alpha y, b(y)=\lambda y$ for some $\alpha,\lambda>0,$ where $Id$ is the identity matrix and $\sigma_2$ is non-negative definite. In this case the fast motion is an Ornstein-Uhlenbeck process and the $L^2-$adjoint of its generator is given by $\mathcal{L}^*h(y)=\alpha(\Delta h(y)+\nabla (yh(y))).$ The solution of the equation $\{\mathcal{L}^*\rho_\infty=0, \int_{\mathcal{Y}}\rho_\infty=1\}$ is then given by the density of a $\mathcal{N}(0,Id)$ distribution (here $\mathcal{N}$ denotes the multidimensional normal distribution) . Thus, since $\overline{b}=-\lambda\int_{\mathcal{Y}}yd\mathcal{N}(0,Id)=0$, the centering condition is satisfied and the Poisson equation $\{\mathcal{L}\Psi=-b, \int_{\mathcal{Y}}\Psi d\mu=0\}$ is solved by $\Psi(y)=\lambda y/\alpha.$ It follows that $\nabla\Psi(y)=(\lambda/\alpha)Id$,$\nabla\Psi\tau=(\sqrt{2}\lambda/\sqrt{\alpha})Id$ and
\begin{equation*}
\begin{aligned}
\blangle \overline{Q(x)Q^T(x)}\xi, \xi\brangle&=\bigg(\int_{\mathcal{Y}} \blangle\nabla\Psi(y)\tau(y)+\sigma_2(x,y)]\xi, [\nabla\Psi(y)\tau(y)+\sigma_2(x,y)]\xi \brangle d\mu(y)\bigg)\\&
=\int_{\mathcal{Y}}\bigg(\big\|\nabla\Psi\tau(y)\xi\big\|^2+\big\|\sigma_2(x,y)\xi\big\|^2+2\blangle\nabla\Psi(y)\tau(y)\xi, \sigma_2(x,y)\xi\brangle\bigg)d\mu(y)\\&
\geq \frac{2\lambda^2}{\alpha}\big\|\xi\big\|^2+\frac{2\sqrt{2}\lambda}{\sqrt{\alpha}}\blangle\xi,\bar{\sigma}_2(x)\xi\brangle\\&\geq\frac{2\lambda^2}{\alpha}\big\|\xi\big\|^2,
\end{aligned}
\end{equation*}
which holds since $\sigma_2$ is nonnegative definite. Thus the assumption for the lower bound is satisfied.
\section{ Comparison to the case H=1/2}\label{comparisonsec} \noindent Having proved a Large Deviation Principle for the slow process \eqref{model} when $H>1/2$ (or $H>3/4$ in the case of Condition \ref{C6}$(i)$) , it is natural to ask in what sense is the rate function $S_{x_0}$ \eqref{ordform2} different from the one obtained in the classical Freidlin-Wentzell theory, where $H=1/2$ and $B^{H}$ is a standard Brownian motion independent of $W$. To this end let $Q$ as in \eqref{Qdef1} and, for each $x\in\mathcal{X},$  \begin{equation*}\label{Q1/2}
\begin{aligned}
Q_{1/2}(x)&:=\int_{\mathcal{Y}}\bigg(\sigma_1(x,y)\sigma_1^T(x,y)+[\sigma_2(x,y)+\nabla\Psi(y)][\sigma_2(x,y)+\nabla\Psi(y)\tau(y)]^T\bigg)d\mu(y)\\&
=\overline{\sigma_1\sigma_1^T}(x)+\overline{QQ^T}(x).
\end{aligned}
\end{equation*}
Assuming for simplicity that  $Q_{1/2}(x)$ is uniformly non-degenerate, it is well-known that the LDP rate function for $H=1/2$ is given by the explicit formula
\begin{equation}\label{FW1}
S^{1/2}_{x_0}(\phi):=\frac{1}{2}\int_{0}^{T}\bigg\langle \dot{\phi}_t-\bar{c}(\phi_t)-\overline{\nabla\Psi g}(\phi_t), Q^{-1}_{1/2}(\phi_t)\big[ \dot{\phi_t}-\bar{c}(\phi_t)-\overline{\nabla\Psi g}(\phi_t)      \big]\bigg\rangle_{\mathcal{X}}dt,
\end{equation}
for all $\phi\in C([0,T];\mathcal{X})$ such that $\phi_0=x_0$ and $Q^{-1/2}_{1/2}(\phi)[\dot{\phi}-\bar{c}(\phi)-\overline{\nabla\Psi g}(\phi)]\in L^2([0,T];\mathcal{X})$ and $S^{1/2}_{x_0}(\phi)=\infty$ otherwise. In the simpler case $\sigma_2=b=0$ and $\sigma_1\sigma_1^T$ is uniformly non-degenerate, the rate function reduces to 
\begin{equation}\label{FW2}
S^{1/2}_{x_0}(\phi)=\frac{1}{2}\int_{0}^{T}\bigg\langle \dot{\phi}_t-\bar{c}(\phi_t),\big(\overline{\sigma_1\sigma_1^T}\big)^{-1}(\phi_t)\big[ \dot{\phi_t}-\bar{c}(\phi_t)   \big]\bigg\rangle_{\mathcal{X}}dt,
\end{equation}
for all $\phi\in C([0,T];\mathcal{X})$ such that $\phi_0=x_0$ and $\dot{\phi}-\bar{c}(\phi)\in L^2([0,T];\mathcal{X})$ and $S^{1/2}_{x_0}(\phi)=\infty$ otherwise.

In the first part of this section we show that, in certain cases, $S_{x_0}$  admits representations similar to \eqref{FW1} and \eqref{FW2} and discuss the differences between the formulas. Finally, we consider its pointwise limit as $H\to\frac{1}{2}^+$ and show that the rate function is discontinuous at $H=1/2$ in Proposition \ref{Hlimprop} and Remark \ref{Hlimrem}.  To this end we shall write  $S^H_{x_0}\equiv S_{x_0}$ to emphasize the dependence of the rate function on the Hurst index.
\subsection{On the form of the rate function} We start by considering the case that $\overline{QQ^T}$ has a  uniformly bounded inverse. The following is a simple corollary of Lemma \ref{minimizerlem}.
\begin{cor}\label{cor:RateNonVar} Let $x_0\in\mathcal{X},Q, \mathcal{Q}_H(\phi)$ as in \eqref{Qdef1}, \eqref{QH1} respectively and assume that the matrix-valued function $\overline{QQ^T}$ is uniformly non-degenerate. Then, for all $\phi\in C([0,T];\mathcal{X})$ such that $\phi_0=x_0$ and $\dot{\phi}-\bar{c}(\phi)-\overline{\nabla\Psi g}(\phi)\in L^2([0,T];\mathcal{X}),$   we have
	\begin{equation}\label{FW3}
	S^{H}_{x_0}(\phi)=\frac{1}{2}\int_{0}^{T}\bigg\langle \dot{\phi}_t-\bar{c}(\phi_t)-\overline{\nabla\Psi g}(\phi_t),[\mathcal{Q}_H(\phi)\mathcal{Q}_H(\phi)^*]^{-1}\big[ \dot{\phi}-\bar{c}(\phi)-\overline{\nabla\Psi g}(\phi)      \big](t)\bigg\rangle_{\mathcal{X}}dt,
	\end{equation}
	and $S_{x_0}(\phi)=\infty$ otherwise.
\end{cor}
\begin{proof} Invoking \eqref{minimizers} we have 
	\begin{equation*}
	\begin{aligned}
	S^{H}_{x_0}(\phi)&=\inf_{u=(u_1,u_2)\in \mathcal{Q}^{-1}_H(\phi)\big[\dot{\phi}- \bar{c}(\phi)-\overline{\nabla\Psi g}(\phi)\big] }\frac{1}{2}\big\|u\big\|^2_{L^2([0,T];\mathcal{U}_1)\oplus L^2([0,T]\times\mathcal{Y}, dt\otimes d\mu;\mathcal{U}_2)}\\&
	=\frac{1}{2}\big\|u^*
	\big\|^2_{L^2([0,T];\mathcal{U}_1)\oplus L^2([0,T]\times\mathcal{Y}, dt\otimes d\mu;\mathcal{U}_2)}
	\\&	=\frac{1}{2}\bigg(\int_{0}^{T}\langle u_1^*(t), u_1^*(t)\rangle dt+\int_{0}^{T}\int_{\mathcal{Y}}\langle u_2^*(t,y), u_2^*(t,y)\rangle d\mu(y) dt\bigg)
	\\&=\frac{1}{2}\bigg(\int_{0}^{T}\bigg\langle [\mathcal{Q}_H(\phi)\mathcal{Q}_H(\phi)^*]^{-1}\big[\dot{\phi}- \bar{c}(\phi)-\overline{\nabla\Psi g}(\phi)\big](t),\\&\quad\quad\quad \big[\bar{\Sigma}_1(\phi)\dot{K}_H\big] \dot{K}_H^*\bar{\Sigma}^*_1(\phi)[\mathcal{Q}_H(\phi)\mathcal{Q}_H(\phi)^*]^{-1}\big[\dot{\phi}- \bar{c}(\phi)-\overline{\nabla\Psi g}(\phi)\big](t)\bigg\rangle dt\\&+\int_{0}^{T}\int_{\mathcal{Y}}\bigg\langle[\mathcal{Q}_H(\phi)\mathcal{Q}_H(\phi)^*]^{-1} \big[\dot{\phi}- \bar{c}(\phi)-\overline{\nabla\Psi g}(\phi)\big](t),\\&\quad\quad\quad  Q(\phi_t,y) Q^T(\phi_t,y)[\mathcal{Q}_H(\phi)\mathcal{Q}_H(\phi)^*]^{-1}\big[\dot{\phi}- \bar{c}(\phi)-\overline{\nabla\Psi g}(\phi)\big](t)\bigg\rangle d\mu(y) dt\bigg)\\&
	=\frac{1}{2}\bigg(\int_{0}^{T}\bigg\langle [\mathcal{Q}_H(\phi)\mathcal{Q}_H(\phi)^*]^{-1}\big[\dot{\phi}- \bar{c}(\phi)-\overline{\nabla\Psi g}(\phi)\big](t),\\&\quad\quad\quad \big[\bar{\Sigma}_1(\phi)\dot{K}_H\big] [\bar{\Sigma}_1(\phi)\dot{K}_H]^*[\mathcal{Q}_H(\phi)\mathcal{Q}_H(\phi)^*]^{-1}\big[\dot{\phi}- \bar{c}(\phi)-\overline{\nabla\Psi g}(\phi)\big](t)\bigg\rangle dt
	\\&+\int_{0}^{T}\bigg\langle[\mathcal{Q}_H(\phi)\mathcal{Q}_H(\phi)^*]^{-1} \big[\dot{\phi}- \bar{c}(\phi)-\overline{\nabla\Psi g}(\phi)\big](t),\\&\quad\quad\quad \overline{Q(\phi_t,\cdot) Q^T(\phi_t,\cdot)}[\mathcal{Q}_H(\phi)\mathcal{Q}_H(\phi)^*]^{-1} \big[\dot{\phi}- \bar{c}(\phi)-\overline{\nabla\Psi g}(\phi)\big](t)\bigg\rangle dt\bigg)\\&
	=\frac{1}{2}\int_{0}^{T}\bigg\langle [\mathcal{Q}_H(\phi)\mathcal{Q}_H(\phi)^*]^{-1}\big[\dot{\phi}- \bar{c}(\phi)-\overline{\nabla\Psi g}(\phi)\big](t)\\&\quad\quad, [\mathcal{Q}_H(\phi)\mathcal{Q}_H(\phi)^*][\mathcal{Q}_H(\phi)\mathcal{Q}_H(\phi)^*]^{-1}\big[\dot{\phi}- \bar{c}(\phi)-\overline{\nabla\Psi g}(\phi)\big](t)\bigg\rangle dt.
	\end{aligned}
	\end{equation*}   	
\end{proof}
\begin{rem} Comparing \eqref{FW3}
and \eqref{FW1} we see that the difference between the two rate functions lies in the "effective diffusivity" operators $[\mathcal{Q}_H(\phi)\mathcal{Q}_H(\phi)^*], Q_{1/2}.$ On the one hand, $[\mathcal{Q}_H(\phi)\mathcal{Q}_H(\phi)^*]$ is a non-local operator which takes into account the covariance structure of the fBm (recall the definition of the fractional integral operator $\dot{K}_H$ \eqref{Kdot}). Moreover, $[\mathcal{Q}_H(\phi)\mathcal{Q}_H(\phi)^*]$ features the na\"ively averaged term  $\bar{\sigma}_1$ (recall \eqref{multop}). On the other hand, $Q_{1/2}$ is expressed as a matrix-valued function which essentially depends on the classically averaged coefficient $(\overline{\sigma_1\sigma_1^T})^{1/2}.$ Finally, note that if $\sigma_1=0$ the two operators coincide.	
\end{rem}

Even though \eqref{FW3} provides a non-variational form of the rate function, it depends on the inverse of the "effective diffusivity" operator $[\mathcal{Q}_H(\phi)\mathcal{Q}_H(\phi)^*]^{-1}.$ The latter is a bounded operator and we have only proved its existence. Nevertheless, in some special cases we obtain a more explicit formula which can be directly compared to \eqref{FW2} and is useful in studying continuity properties of $S^H_{x_0}$  with respect to $H.$
\begin{cor}\label{explicitcor} Let $b=\sigma_2=0,$ $c_H$ as in \eqref{cH} and assume that $\bar{\sigma}_1$ is symmetric with a bounded inverse.  Then the rate function takes the form
	\begin{equation}\label{explicit}
	\begin{aligned}
	S^H_{x_0}(\phi)&=\frac{1}{2c^2_H\Gamma(\frac{3}{2}-H)^2}\int_{0}^{T}\bigg|t^{1/2-H}\bar{\sigma}_1(\phi_t)^{-1}\big[\dot{\phi}_t-\bar{c}(\phi_t) \big]\\&
	+\bigg(H-\frac{1}{2}\bigg)t^{H-\frac{1}{2}}\int_{0}^{t}\frac{t^{\frac{1}{2}-H}\bar{\sigma}_1(\phi_t)^{-1}\big[\dot{\phi}_t-\bar{c}(\phi_t) \big]-s^{\frac{1}{2}-H}\bar{\sigma}_1(\phi_s)^{-1}\big[\dot{\phi}_s-\bar{c}(\phi_s) \big]}{(t-s)^{H+\frac{1}{2}}}ds\bigg|^2dt
	\end{aligned}
	\end{equation}
	for all $\phi\in C([0,T];\mathcal{X})$ such that $\phi_0=x_0$ and $\bar{\sigma}_1(\phi)^{-1}[\dot{\phi}-\bar{c}(\phi)]\in \dot{K}_H[ L^2([0,T];\mathcal{U}_1) ]$ and $S_{x_0}(\phi)=\infty$ otherwise.
\end{cor}
\begin{proof} From \eqref{smallnoiseform}, \eqref{smallnoiscon} we have $S_{x_0}(\phi)<\infty$ if and only if $\phi_0=x_0,$ $\bar{\sigma}_1(\phi)^{-1}[\dot{\phi}-\bar{c}(\phi)]\in \dot{K}_H[ L^2([0,T];\mathcal{U}_1) ]$ and the infimum is (uniquely) attained by $u^*=\int_{0}^{\cdot}\bar{\sigma}_1(\phi_s)^{-1}(\dot{\phi}_s-\bar{c}(\phi_s))ds.$ Thus,  
	\begin{equation*}\label{nonvarrate2}
	S^{H}_{x_0}(\phi)=\frac{1}{2}\|u^*\|^2_{\h_H}=\frac{1}{2}\int_{0}^{T}\big|K^{-1}_H[u^*](t)\big|^2dt
	\end{equation*} 
	and \eqref{explicit} follows from \eqref{Kinverse}.
\end{proof}

\subsection{Limit as $H\rightarrow\frac{1}{2}^+$} In view of \eqref{FW3}, we expect that, as $H\rightarrow\frac{1}{2}^+,$ $S^{H}_{x_0}$ does not, in general, converge to $S^{1/2}_{x_0}$ \eqref{FW2}. In this section we provide a proof of this discontinuity of the rate function at $H=1/2.$  In order to study the limiting behavior of \eqref{FW3} we will work in a family of weighted H\"older spaces defined as follows: Let $\alpha\geq 1,$ $\beta\in(0,1),$ define a weight $w_{\alpha}(t)=t^\alpha, t\in[0,T]$ and consider the vector space
$$C^\beta_{\alpha}([0,T];\mathcal{U}_1):=\bigg\{ f:[0,T]\rightarrow \mathcal{U}_1: \frac{f}{w_{\alpha}}\in C^\beta([0,T];\mathcal{U}_1)\bigg\}.$$
The norm $$\|f\|_{\alpha,\beta}:=\|f/w_\alpha\|_{C^\beta([0,T];\mathcal{U}_1)}$$
turns $C^\beta_{\alpha}([0,T];\mathcal{U}_1)$ to a Banach space. The latter is easy to verify by noting that the map 
$$C^\beta([0,T];\mathcal{U}_1)\ni f\longmapsto w_\alpha f\in C^\beta_{\alpha}([0,T];\mathcal{U}_1)$$
is a bijective linear isometry and $C^\beta([0,T];\mathcal{U}_1)$ is a Banach space. Finally, due to the Lipschitz continuity of the weight $w_\alpha$, the linear inclusion $C^\beta_{\alpha}([0,T];\mathcal{U}_1)\subset C^\beta([0,T];\mathcal{U}_1)$ is continuous. For the use of these weighted H\"older spaces in the context of fractional calculus, the reader is referred to \cite{Samko1993FractionalIA}, Chapter 1.

We start with by proving some mapping properties of the operator $\dot{K}^{-1}_H.$ The estimates that follow are uniform in $H$ when $H$ is close to $1/2.$
\begin{lem}\label{KHmaplem}
	Let $\rho>0$. For any $\alpha\geq 1+\rho$ the following hold:\\
	$(i)$  For all $\beta\geq \rho$, there exists a constant $C=C_{T,\alpha,\beta,\rho}>0$ such that
	\begin{equation*}\label{Hbnd1}
	\begin{aligned}
	\sup_{H\in(1/2,1/2+\rho)}\sup_{t\in[0,T]}\int_{0}^{t}\bigg|\frac{t^{\frac{1}{2}-H}f(t)-s^{\frac{1}{2}-H}f(s)}{(t-s)^{H+\frac{1}{2}}}\bigg|ds
	\leq C\|f\|_{\alpha,\beta}.
	\end{aligned}
	\end{equation*}
	$(ii)$ Let $H\in(1/2, 1/2+\rho).$ The linear operator $\dot{K}^{-1}_H$ maps $C^\beta_{\alpha}([0,T];\mathcal{U}_1)$ continuously to $L^\infty([0,T];\mathcal{U}_1)$. Moreover, for $\rho$ sufficiently small, we have 
	\begin{equation*}
	\sup_{H\in(\frac{1}{2},\frac{1}{2}+\rho)}\big\|\dot{K}_H^{-1}\big\|_{C^\beta_{\alpha}\rightarrow L^\infty}<\infty.
	\end{equation*}\\
	$(iii)$ Let $H\in(1/2, 1/2+\rho),$ $\beta\geq 1/2+\rho.$ The linear operator $\dot{K}^{-1}_H$ maps $W^{2,\beta}_{\alpha}([0,T];\mathcal{U}_1)$ continuously to $L^2([0,T];\mathcal{U}_1)$. Moreover, for all $\beta\geq \rho$ and $\rho$ sufficiently small, we have 
	\begin{equation*}
	\sup_{H\in(\frac{1}{2},\frac{1}{2}+\rho)}\big\|\dot{K}_H^{-1}\big\|_{W^{2,\beta}_{\alpha}\rightarrow L^2}<\infty.
	\end{equation*}
	
\end{lem}
\begin{proof}
	$(i)$ 	Let $H\in(1/2,1/2+\rho), f\in C^\beta_{\alpha}([0,T];\mathcal{U}_1).$   For $t\in[0,T]$ we have
	\begin{equation*}
	\begin{aligned}
	\int_{0}^{t}\frac{t^{\frac{1}{2}-H}f(t)-s^{\frac{1}{2}-H}f(s)}{(t-s)^{H+\frac{1}{2}}}ds&=\int_{0}^{t}\frac{t^{\frac{1}{2}-H+\alpha}[f(t)/w_\alpha(t)]-s^{\frac{1}{2}-H+\alpha}[f(s)/w_\alpha(s)]}{(t-s)^{H+\frac{1}{2}}}ds\\&
	=\int_{0}^{t}\frac{(t^{\frac{1}{2}-H+\alpha}-s^{\frac{1}{2}-H+\alpha})[f(t)/w_\alpha(t)]}{(t-s)^{H+\frac{1}{2}}}ds\\&
	+\int_{0}^{t}\frac{s^{\frac{1}{2}-H+\alpha}[f(t)/w_\alpha(t)-f(s)/w_\alpha(s)]}{(t-s)^{H+\frac{1}{2}}}ds.
	\end{aligned}
	\end{equation*}
	Applying the mean value theorem in the first integral we obtain the bound
	\begin{equation*}
	\begin{aligned}
&	\int_{0}^{t}\bigg|\frac{t^{\frac{1}{2}-H}f(t)-s^{\frac{1}{2}-H}f(s)}{(t-s)^{H+\frac{1}{2}}}\bigg|ds\\&\leq \sup_{t\in[0,T]}\big|f(t)/w_\alpha(t)\big|\bigg(\frac{1}{2}-H+\alpha\bigg)T^{-\frac{1}{2}-H+\alpha}\int_{0}^{t}\frac{(t-s)}{(t-s)^{H+\frac{1}{2}}}ds
	\\&+\big[f/w_\alpha\big]_{C^\beta}T^{\frac{1}{2}-H+\alpha}\int_{0}^{t}\frac{(t-s)^\beta}{(t-s)^{H+\frac{1}{2}}}ds\\&
	\leq \bigg(\frac{1}{2}-H+\alpha\bigg)T^{-\frac{1}{2}-H+\alpha}\big\|f/w_\alpha\big\|_{C[0,T]}
	\frac{t^{\frac{3}{2}-H}}{\frac{3}{2}-H}+\big[f/w_\alpha\big]_{C^\beta}T^{\frac{1}{2}-H+\alpha}\frac{t^{\beta-H+\frac{1}{2}}}{\beta-H+\frac{1}{2}}\\&
	\leq   \alpha T^{-\frac{1}{2}-H+\alpha}\big\|f/w_\alpha\big\|_{C[0,T]}
	\frac{T^{\frac{3}{2}-H}}{1-\rho}+\big[f/w_\alpha\big]_{C^\beta}T^{\frac{1}{2}-H+\alpha}\frac{T^{\beta-H+\frac{1}{2}}}{\beta-\rho}, 
	\end{aligned}
	\end{equation*}
	where we used that $H\in(1/2,1/2+\rho)$ in the last line and the integrability properties hold since $\alpha\geq 1+\rho>\frac{1}{2}+H>H-\frac{1}{2}$ and $\beta\geq \rho>\frac{1}{2}-H.$ Thus 	
	\begin{equation*}
	\begin{aligned}
	&\sup_{H\in(1/2,1/2+\rho)}\sup_{t\in[0,T]}\int_{0}^{t}\bigg|\frac{t^{\frac{1}{2}-H}f(t)-s^{\frac{1}{2}-H}f(s)}{(t-s)^{H+\frac{1}{2}}}\bigg|ds
	\\&\leq   \alpha (1\vee T^{\alpha})\big\|f/w_\alpha\big\|_{C[0,T]}
	\frac{1\vee T}{1-\rho}+\big[f/w_\alpha\big]_{C^\beta}(1\vee T^{\alpha})\frac{T^{\beta}}{\beta-\rho} \leq C_{T,\alpha,\beta,\rho}\|f\|_{\alpha,\beta}.
	\end{aligned}
	\end{equation*}
	The proof is complete.\\
	$(ii)$ Let $t\in[0,T], H\in(1/2,1/2+\rho)$. In view of $(i)$ and \eqref{Kinverse} we have 
	\begin{equation*}
	\begin{aligned}
	c_H\Gamma\bigg(\frac{3}{2}-H\bigg)\big|\dot{K}^{-1}_H[f](t)\big|&\leq t^{\frac{1}{2}-H}\big|f(t)\big|+\bigg(H-\frac{1}{2}\bigg)T^{H-\frac{1}{2}}\bigg|\int_{0}^{t}\frac{t^{\frac{1}{2}-H}f(t)-s^{\frac{1}{2}-H}f(s)}{(t-s)^{H+\frac{1}{2}}}ds\bigg|\\&\leq T^{\frac{1}{2}-H+\alpha}\sup_{t\in[0,T]}\big|f(t)/w_{\alpha}(t)\big|+\rho T^{H-\frac{1}{2}}C\|f\|_{\alpha,\beta}\\&
	\leq \big[  1\vee T^{\alpha}+C\rho (1\vee T^\rho)\big]\|f\|_{\alpha,\beta}.
	\end{aligned}
	\end{equation*}
	This proves the first assertion. 	As for the uniform bound, note that
	
	\begin{equation*}\label{cHlim}
	\lim_{H\to\frac{1}{2}^+}c_H\Gamma\bigg(\frac{3}{2}-H\bigg)=\lim_{H\to\frac{1}{2}}\bigg(\frac{2H\Gamma(\frac{3}{2}-H)\Gamma(H+\frac{1}{2})}{\Gamma(2-2H)}\bigg)^\frac{1}{2}\Gamma\bigg(\frac{3}{2}-H\bigg)=(\Gamma(1))^\frac{3}{2}=1
	\end{equation*}
	from the continuity of the Gamma function. Thus, for $H\in(1/2,1/2+\rho)$, $\rho$ sufficiently small we have 
	\begin{equation*}
	\begin{aligned}
	\big|\dot{K}^{-1}_H[f](t)\big|&
	\leq    C\|f\|_{\alpha,\beta}.
	\end{aligned}
	\end{equation*}
	where the constant is independent of $H$.\\
	$(iii)$ As in $(i),$ we have a decomposition
	\begin{equation*}
	\begin{aligned}
	\int_{0}^{t}\frac{t^{\frac{1}{2}-H}f(t)-s^{\frac{1}{2}-H}f(s)}{(t-s)^{H+\frac{1}{2}}}ds&
	=\int_{0}^{t}\frac{(t^{\frac{1}{2}-H+\alpha}-s^{\frac{1}{2}-H+\alpha})[f(t)/w_\alpha(t)]}{(t-s)^{H+\frac{1}{2}}}ds
	\\&+\int_{0}^{t}\frac{s^{\frac{1}{2}-H+\alpha}[f(t)/w_\alpha(t)-f(s)/w_\alpha(s)]}{(t-s)^{H+\frac{1}{2}}}ds\\&
	\leq \big|f(t)/w_\alpha(t)\big|\bigg(\frac{1}{2}-H+\alpha\bigg)T^{-\frac{1}{2}-H-\alpha}\frac{T^{\frac{3}{2}-H}}{\frac{3}{2}-H}\\&+
	T^{\frac{1}{2}-H+\alpha}\int_{0}^{t}\frac{(t-s)^{\beta-H}[f(t)/w_\alpha(t)-f(s)/w_\alpha(s)]}{(t-s)^{\beta+\frac{1}{2}}}ds\\&
	\leq \frac{ \alpha}{\rho}(1\vee T^{\alpha})(1\vee T)\big|f(t)/w_\alpha(t)\big|\\&+(1\vee T^{\alpha})(1\vee T^{\beta-\frac{1}{2}})\int_{0}^{t}\frac{f(t)/w_\alpha(t)-f(s)/w_\alpha(s)}{(t-s)^{\beta+\frac{1}{2}}}ds,
	\end{aligned}
	\end{equation*}
	where we used that $\beta\geq\frac{1}{2}+\rho>H$ in the last line. Therefore, 
	\begin{equation*}
	\begin{aligned}
	&\big\| \dot{K}^{-1}_Hf\big\|^2_{L^2}\\&\leq C_{\alpha,\beta,\rho,T}\bigg(T^{1-2H+2\alpha}\int_{0}^{T}\big|f(t)/w_\alpha(t)\big|^2dt\\&+\rho^2T^{2H-1}\iint_{[0,T]^2}\frac{|f(t)/w_\alpha(t)-f(s)/w_\alpha(s)|^2}{(t-s)^{2\beta+1}}dsdt\bigg)
	\\&\leq C_{\alpha,\beta,\rho,T}c^2_H\Gamma^2\bigg(\frac{3}{2}-H\bigg)\bigg( (1\vee T^{2\alpha}) +\rho^2(1\vee T^{2\rho})   \bigg)\|f\|^2_{W^{2,\beta}_\alpha}
	\end{aligned}
	\end{equation*}	
	which proves the first assertion. The uniform boundedness then follows as in $(ii)$ for $\rho$ sufficiently small.
\end{proof}
\begin{rem}
	Note that $(iii)$ of the previous lemma follows directly from $(ii)$ in view of the continuous inclusions $L^{\infty}\subset L^2$ and $C^{\beta}\subset W^{2,\beta+\frac{1}{2}}$. The latter is guaranteed by the fractional Sobolev embedding theorems (e.g. Theorem 8.2 in \cite{di2012hitchhikers}  ) which also hold for the weighted spaces we are using.
\end{rem}

\begin{prop}\label{Hlimprop}  Let $\sigma_2=b=0$ and assume that $\bar{\sigma}_1$ is symmetric with a bounded inverse. For $x_0\in\mathcal{X}, T>0, \alpha>1,$ $ \beta>1/2$ define $\mathcal{H}_w:=\big\{ \phi\in C([0,T];\mathcal{X}): \phi(0)=x_0 ,\bar{\sigma}_1(\phi)^{-1}[\dot{\phi}-\bar{c}(\phi)]\in W^{2,\beta}_\alpha([0,T];\mathcal{X})     \big\}$ and a functional $\tilde{S}^{1/2}_{x_0}: C([0,T];\mathcal{X})\rightarrow[0,\infty]$ with
	\begin{equation}\label{Sbar}
	\tilde{S}^{1/2}_{x_0}(\phi):=\frac{1}{2}\int_{0}^{T}\bigg\langle \dot{\phi}_t-\bar{c}(\phi_t), \big(\bar{\sigma}_1\bar{\sigma}^T_1\big)^{-1}(\phi_t)\big[ \dot{\phi}_t-\bar{c}(\phi_t)\big]\bigg\rangle_{\mathcal{X}}dt
	\end{equation}
	if $\dot{\phi}-\bar{c}(\phi)\in L^{2}([0,T];\mathcal{X}) $ and $\tilde{S}^{1/2}_{x_0}(\phi)=\infty$ otherwise. Then for all $\phi\in \mathcal{H}_w$ we have 
	\begin{equation*}
	\lim_{H\to\frac{1}{2}^+} S^H_{x_0}(\phi)=\tilde{S}^{1/2}_{x_0}(\phi).
	\end{equation*}
\end{prop}
\begin{proof}Let $\phi\in \mathcal{H}_w,$ and set $\psi=\bar{\sigma}_1(\phi)^{-1}[\dot{\phi}-\bar{c}(\phi)].$ Moreover let $\rho>0$ such that $\alpha\geq 1+\rho, \beta\geq\frac{1}{2}+\rho$ and $H\in(1/2, 1/2+\rho).$ From Lemma \ref{KHmaplem}$(iii)$ we have  $\dot{K}^{-1}_H\psi\in L^2.$ Thus $S^{H}_{x_0}(\phi), \tilde{S}^{1/2}_{x_0}(\phi)<\infty$ and there exists a (possibly smaller) $\rho$ such that 
	\begin{equation*}
	\begin{aligned}
	&\big|S^H_{x_0}(\phi)-\tilde{S}^{1/2}_{x_0}(\phi)\big|=\bigg|\big\|\dot{K}_H^{-1}\psi  \big\|^2_{L^2}-\big\|\psi\big\|_{L^2}^2          \bigg|\\&
	\leq c_H^2\Gamma\bigg(\frac{3}{2}-H\bigg)^2\bigg(\int_{0}^{T}\big|t^{1-2H+2\alpha}-t^{2\alpha}\big|\big|\psi(t)/w_{\alpha}(t)\big|^2dt+C\bigg(H-\frac{1}{2}\bigg)^2T^{2H-1}\|\psi\|^2_{W^{2,\beta}_\alpha}\bigg)
	\end{aligned}
	\end{equation*}
	where $C$ does not depend on $H$ and the last line follows from the triangle inequality along with the estimates in the proof of Lemma \ref{KHmaplem}$(iii).$ Since $\alpha>1+\rho>H-1/2$ and $c_H\Gamma\bigg(\frac{3}{2}-H\bigg)\rightarrow1$ as $H\to1/2,$ we can apply the dominated convergence theorem to conclude that 
	\begin{equation*}
	\big|S^{H}_{x_0}(\phi)-\tilde{S}^{1/2}_{x_0}(\phi)\big|\longrightarrow 0,
	\end{equation*}
	as $H\rightarrow{1/2}^+$.	
\end{proof}

\begin{rem}\label{Hlimrem} The functional $\tilde{S}^{1/2}_{x_0}$ \eqref{Sbar} coincides with the Freidlin-Wentzell rate function $S^{1/2}_{x_0}$ \eqref{FW2} if $\sigma_1$ does not depend on the fast variables (i.e. $\sigma_1=\sigma_1(x)$).	In this case, Proposition \eqref{Hlimprop} shows that for all $\phi\in\mathcal{H}_w,$ $H\mapsto S^H_{x_0}(\phi)$ is right continuous at $H=1/2.$
	In general, however, the two functionals are not equal and this implies that $S^H_{x_0}$ is (pointwise) discontinuous at $H=1/2.$ Indeed, consider for example  $\mathcal{X}=\mathcal{Y}=\R,$ $\sigma_1(y)=\cos(y)$ and let $Y^{\epsilon,\eta}$ be an Ornstein-Uhlenbeck process as in the third example of Section \ref{assumptionsection}. Then
	\begin{equation*}
	\bar{\sigma}^2_1=\bigg(\int_{\mathcal{Y}}\cos(y)d\mathcal{N}(0,1)(y)\bigg)^2=\frac{1}{e}\neq\frac{1}{2}\bigg(1+\frac{1}{e^2} \bigg)=\int_{\mathcal{Y}}\cos^2(y)d\mathcal{N}(0,1)(y)=\overline{\sigma^2_1}
	\end{equation*}
	and thus $\tilde{S}^{1/2}_{x_0}\neq S^{1/2}_{x_0}.$ 
    
    In fact, when $\mathcal{X}= \R^m,$ $\mathcal{Y}=\R,$ this discontinuity holds more generally. Indeed, Jensen's inequality implies that for all $x\in\mathcal{X},$ $\bar{\sigma}^2_1(x)\leq \overline{\sigma^2_1}(x)$ and hence, in view of \eqref{Sbar}, \eqref{FW2} the limiting rate function satisfies $$ \tilde{S}^{1/2}_{x_0}(\phi)=\int_{0}^{T}\frac{\big|\dot{\phi}_t-\bar{c}(\phi_t)\big|^2}{2\bar{\sigma}^2_1(\phi_t)}dt\geq \int_{0}^{T}\frac{\big|\dot{\phi}_t-\bar{c}(\phi_t)\big|^2}{2\overline{\sigma^2_1}(\phi_t)}dt=S^{1/2}_{x_0}(\phi),$$
   whenever $\phi$ is such that both sides are finite and under the assumption that $\sigma_1, \bar{\sigma}_1$ are uniformly lower bounded. From the last display we see that $\tilde{S}^{1/2}_{x_0}(\phi)=S^{1/2}_{x_0}(\phi)$ if for all $x\in\mathcal{X},$ $\bar{\sigma}^2_1(x)=\overline{\sigma^2_1}(x).$ Again from Jensen's inequality the latter holds if and only if, for all $x\in\mathcal{X},$ $\sigma_1(x,\cdot)$ is constant $\mu-$almost surely. In other words, the map $H\mapsto S^H_{x_0}(\phi)$ is right-continuous at $H=1/2$ if $\sigma_1(x,\cdot)$ is constant $\mu-$almost surely for all $x\in\mathcal{X}$. We conclude this remark with a similar observation for the multidimensional case $\mathcal{X}= \R^m,$ $\mathcal{Y}=\R^{d-m}$ for some $\N\ni d>m.$ In this setting it holds that $\tilde{S}^{1/2}_{x_0}(\phi)=S^{1/2}_{x_0}(\phi),$ and hence $H\mapsto S^H_{x_0}(\phi)$ is right-continuous at $H=1/2,$ if and only if the $m\times k$ (recall that $B^H$ is $k-$dimensional) effective diffusivity matrices satisfy $\bar{\sigma}_1\bar{\sigma}^T_1(x)=\overline{\sigma_1\sigma_1^T}(x)$  for all $x\in\mathcal{X}.$ Therefore,  if there exists $x\in\mathcal{X}$ and $(i,j)\in\{1,\dots, m\}\times\{1,\dots, k\}$ for which $\sigma_1^{i,j}(x,\cdot)$ is not constant $\mu-$almost surely, then the rate function $H\mapsto S^H_{x_0}$ is pointwise discontinuous at $H=1/2.$
\end{rem}

\section{Conclusions, extensions and future work}\label{concsec} In this paper we studied large deviations of the slow process $X^{\epsilon,\eta}$ \eqref{model} from the homogenized limit, as $\epsilon,\eta\to 0,$ in the case where the driving fBm $B^H$ has Hurst index $H>1/2.$ Working under \eqref{regime} and Condition \ref{C6}, we were able to extend the classical theory by proving a Laplace Principle (Theorem  \ref{main2}) with good rate function $S=S^H$ \eqref{ratefun}. Moreover we showed that, in certain cases,  $S^H$ has an explicit form \eqref{FW3} and proved that it is discontinuous at $H=1/2$ (Proposition \ref{Hlimprop}, Remark \ref{Hlimrem}).

At this point, the reader might wonder whether our LDP result Theorem \ref{main2} continues to hold if we replace the fBm $B^H$ in \eqref{model} by a generic Gaussian process $Z$ with $H-$H\"older continuous sample paths, for some $H\in(1/2, 1),$ covariance operator $Q_Z$ and Cameron-Martin space $\h_Z.$ In this setting and under all the assumptions spelled out in Section \ref{subsecAssumptions}, the variational formula \eqref{varform}, our tightness analysis of controlled slow dynamics (Proposition \ref{tightnessprop} and Section \ref{tightsec}) and the auxiliary estimates from Appendix \ref{AppA} continue to hold, provided that the Hilbert space $\h_Z$ is continuously embedded to the Cameron-Martin space $\h_{1/2}$ of a standard Brownian motion (in our setting the latter is proved in Lemma \ref{udotL2}). Indeed, neither Condition \ref{C6} (which only depends on the path regularity of $Z$) nor the proofs of the aforementioned results rely on the particular representations of $B^H$ and its Cameron-Martin space $\h_H.$ Moreover, occupation measures $P^\epsilon$ can be defined analogously to \eqref{occupation} with the operator $K_H$ replaced by the square-root $K_Z$ of the covariance operator $Q_Z$ (in fact $\h_Z$ can be explicitly defined as the Hilbert space $K_Z[L^2([0,T];\mathcal{X})]$ endowed with the inner product $\langle f, g\rangle_{\h_Z}:=\langle K_Z^{-1}f,  K_Z^{-1}g\rangle_{L^2};$ see e.g. \cite{lifshits2012lectures}, Theorem 4.1). Nevertheless, our identification of weak limit points, Theorem \ref{main1} and in particular Proposition \ref{sigmaave}, relies on the knowledge of an explicit expression for the operator $K_H.$ Hence, if $Z, K_Z$ are such that an analogue of Proposition \ref{sigmaave} holds (mutatis mutandis), then Theorems \ref{main1}, \ref{main2} hold without further assumptions on the process $Z.$ Furthermore, it is worth noticing that, apart from substituting $K_H$ by $K_Z$ and modulo the aforementioned assumptions on $Z,$ the form of the effective diffusivity operator $\mathcal{Q}_H$ \eqref{QH1} as well as the proof of Corollary \eqref{cor:RateNonVar} are expected to remain unchanged.

We shall now describe a number of potential directions for future work on this topic.

\noindent We have treated the case of a general diffusion coefficient $\sigma_1$ under Condition \ref{C6}$(i).$ This allows us to obtain tightness bounds for the Young integral $\sqrt{\epsilon}\sigma_1 dB^H$ (see \eqref{maxbnd}, Lemma \ref{Yfraclem}) by taking advantage of the small noise. The latter comes at the cost of restricting both the Hurst index, i.e. $H\in(3/4, 1),$ and the asymptotic regime; see \eqref{betaregime}. We note here that a similar threshold of $H=3/4$ has appeared in the context of the Breuer-Major theorem \cite{breuer1983central} (see also \cite{nourdin2012normal}, Chapter 7 and \cite{neufcourt2016third}). We believe that these restrictions are technical in nature and that Condition \ref{C6}(i) is not necessary for an LDP to hold.   An extension of the LDP with the weak convergence approach and in the absence of this condition provides an interesting problem which we plan to investigate in the future.

Throughout this work we have assumed that the scale separation parameter $\eta$ vanishes faster than the noise intensity $\epsilon,$ as $\epsilon$ goes to zero; see \eqref{regime}. Establishing an LDP in the regimes
\begin{equation*}
\begin{aligned}
\lim_{\epsilon\to 0}\frac{\sqrt{\eta}}{\sqrt{\epsilon}}=\begin{cases}
&\gamma\in(0,\infty)\;,\\&
\infty
\end{cases}
\end{aligned}
\end{equation*}
remains an open problem. From the perspective of the weak convergence approach, proving the Laplace Principle lower bound \eqref{laplower} presents additional challenges. First, in both regimes, the invariant measure of the uncontrolled fast dynamics depends on the slow component. Moreover, in both regimes, the long-time behavior of the controlled fast dynamics $Y^{\epsilon,\eta,u}$  \eqref{consys} depends non-trivially on the control $u_2.$ Thus the $y-$marginal of the limiting occupation measure $P$ is no longer decoupled from $u_2$ (see \eqref{Pdec}) and our construction of a (nearly) optimal control that achieves the lower bound does not carry over to this setting. In the case $H=1/2,$ such a construction was achieved in \cite{dupuis2012large} for the first of the aforementioned regimes. The arguments there rely on a local form of the rate function along with tools from ergodic control of diffusion processes which are available due to the Markovianity of the dynamics.

Finally, the LDP rate function can be used to design efficient accelerated Monte Carlo methods for the simulation of rare events for multiscale dynamics perturbed by fBm. Similar work in the case $H=1/2$ has been carried out in \cite{spiliopoulos2013large}.

\appendix
\section{Appendix}\label{AppA}

\noindent In this section we collect the proofs of several auxiliary estimates. Lemmas \ref{Ybndlem}, \ref{Yfraclem} provide estimates for the controlled fast process $Y^{\epsilon,\eta,u}$ \eqref{consys}. Lemma \ref{Yfraclem} is central to the proof of Proposition \ref{tightnessprop} and is connected to the relative asymptotic rate of $\epsilon,\eta$ in Condition \ref{C6}$(i)$ (see also Remark \ref{relativeraterem}). Lemmas \ref{psider} and \ref{psiliplem} are concerned with the derivative $\nabla\Psi$ of the Poisson equation \eqref{Poisson}. In particular, we provide sufficient conditions for the coefficients of the fast motion under which Condition \ref{C6} is satisfied. Finally, in Lemma \ref{udotL2} we prove that for each family of controls $\{u^\epsilon;\epsilon>0\}=\{(u^\epsilon_1,u^\epsilon_2);\epsilon>0\}\subset \mathcal{A}_N$ \eqref{AN}, the family $\{\dot{u}^\epsilon_1;\epsilon>0\}$ of derivatives is uniformly bounded in $L^2.$

\begin{lem}\label{Ybndlem}
	Let $T>0, p\geq 1$. Under Conditions \ref{C1}-\ref{C3} the following hold:\\
	(i) There exists $C>0$ and $\epsilon_0>0$ such that 
	\begin{equation*}\label{Ybnd}
	\sup_{\epsilon<\epsilon_0, u\in\mathcal{A}_N}\ex\bigg(\int_{0}^{T}|Y^{\epsilon,\eta,u}_t|^{2}dt\bigg)^p\leq C.
	\end{equation*}
	(ii)  Let $\theta\in(0,\frac{1}{2})$. There exists $C>0$ such that for all $\epsilon>0$ 
	\begin{equation*}\label{Yholder}
	\sup_{ u\in\mathcal{A}_N}\ex[Y^{\epsilon,\eta,u}]^p_{C^{\theta}([0,T];\mathcal{Y})}\leq C	(\epsilon\eta)^{-\frac{p}{2}}.
	\end{equation*}
\end{lem}
\begin{proof}
	Condition \ref{C3} allows us to invoke Duhamel's principle and write
	\begin{equation*}\label{Duhamel}
	\begin{aligned}
	Y^{\epsilon,\eta,u}_t&=e^{-\Gamma t/\eta}y_0+\frac{1}{\eta}\int_{0}^{t}e^{-\Gamma(t-s)/\eta}\zeta(Y^{\epsilon,\eta,u}_s)ds\\&+\frac{1}{\sqrt{\epsilon\eta}}\int_{0}^{t}e^{-\Gamma(t-s)/\eta}\big(g(X^{\epsilon,\eta,u}_s,Y^{\epsilon,\eta,u}_s)+\tau(Y^{\epsilon,\eta,u}_s)\dot{u}_2(s)\big)ds\\&+\frac{1}{\sqrt{\eta}}\int_{0}^{t}e^{-\Gamma(t-s)/\eta}\tau(Y^{\epsilon,\eta,u}_s)dW_s.
	\end{aligned}
	\end{equation*}
	(i) The uniform moment bound  follows along the same lines as the second estimate in Lemma 3.1 of \cite{spiliopoulos2020importance} and, to avoid repetition, its proof is omitted.\\
	(ii) For notational simplicity we drop the superscripts and write $Y=Y^{\epsilon,\eta,u}, X=X^{\epsilon,\eta,u}$. For $0\leq s<t\leq T$ we have 
	\begin{equation}\label{Holderdec}
	\begin{aligned}
	Y_t-Y_s-\big[I-e^{-\Gamma(t-s)/\eta}\big]Y_s&=\frac{1}{\eta}\int_{s}^{t}e^{-\Gamma(t-r)/\eta}\zeta( Y_r) dr\\&+\frac{1}{\sqrt{\epsilon\eta}}\int_{s}^{t}e^{-\Gamma(t-r)/\eta}\big(g(X_r, Y_r)+\tau(Y_r)\dot{u}_2(r)\big)dr\\&+\frac{1}{\sqrt{\eta}}\int_{s}^{t}e^{-\Gamma(t-r)/\eta}\tau(Y_r)dW_r\\&
	=: \sum_{j=1}^{3}A_j(s,t).	   \end{aligned}
	\end{equation}
	Letting  $y\in\mathcal{Y}$, $\{\gamma_k\}_{k=1}^{d-m}$ be the positive eigenvalues of $\Gamma$ corresponding to an orthonormal eigenbasis $\{e_k\}_{k=1}^{d-m}$ and $\gamma:=\inf_k\gamma_k>0$ we have
	\begin{equation*}
	\big|e^{-\Gamma t}y\big|=\bigg|\sum_{k=1}^{d-m}e^{-\gamma_k t}y_k e_k\bigg|\leq Ce^{-\gamma t}|y|.
	\end{equation*}
	Hence, with $L_\zeta$ as in Condition \ref{C3} and using Conditions \ref{C1}, \ref{C2} we have
	\begin{equation*}
	\begin{aligned}
	\big|A_1(s,t)&+A_2(s,t)\big|\leq \frac{C(1+L_\zeta)}{\eta}\int_{s}^{t}e^{-\gamma(t-r)/\eta} dr\\&+\frac{1}{\sqrt{\epsilon\eta}}\int_{s}^{t}e^{-\gamma(t-r)/\eta}\big( C_g(1+|Y_r|)+|\tau(Y_r)||\dot{u}_2(r)|\big)dr\\&
	\leq \frac{C}{\eta}\frac{\sqrt{\eta}}{\sqrt{\gamma}}\int_{s}^{t}(t-r)^{-\frac{1}{2}} dr
    + \frac{C_g}{\sqrt{\epsilon\eta}}\bigg[\int_s^t  e^{-2\gamma(t-r)/\eta}dr\bigg]^{\frac{1}{2}}\bigg[\int_0^T (1+|Y_r|)^2dr\bigg]^{\frac{1}{2}}
    \\&+\frac{\|\tau\|_\infty}{\sqrt{\epsilon\eta}}\bigg[\int_{s}^{t}e^{-2\gamma(t-r)/\eta}dr\bigg]^{\frac{1}{2}}
	\bigg[\int_{0}^{T}|\dot{u}_2(r)|^2dr\bigg]^{\frac{1}{2}}\\&
	\leq \frac{C}{\sqrt{\eta}}(t-s)^{1/2}+\frac{C_g}{\sqrt{\epsilon\eta}}(t-s)^{\frac{1}{2}}\bigg[\int_0^T (1+|Y_r|)^2dr\bigg]^{\frac{1}{2}}+\frac{\|\tau\|_\infty N}{\sqrt{\epsilon\eta}}(t-s)^{\frac{1}{2}}.
	\end{aligned}
	\end{equation*} 
	Thus, after taking expectation and using the moment bound from (i), we obtain
	\begin{equation}\label{A12}
	\begin{aligned}
	\ex\big|A_1(s,t)+A_2(s,t)\big|&\leq C\bigg[	   \frac{1}{\sqrt{\eta}} (t-s)^{1/2}+\frac{1}{\sqrt{\epsilon\eta}}(t-s)^{\frac{1}{2}}\bigg].
	\end{aligned}
	\end{equation} 
	The stochastic convolution term can be treated using the following factorization formula followed by the BDG inequality
	\begin{equation}\label{factor}
	A_3(s,t)=\frac{1}{\sqrt{\eta}}\frac{\sin(a\pi)}{\pi}\int_{s}^{t}  (t-r)^{\theta-1} e^{-\Gamma(t-r)/\eta}\int_{0}^{r} (r-z)^{-\theta}e^{-\Gamma(r-z)/\eta}\tau(Y_z)dW_zdr.
	\end{equation}
	In particular, it is straightforward to show that 
	\begin{equation}\label{A3}
	\ex\big(\sup_{t\neq s} |t-s|^{-\theta}|A_3(s,t)|\big)^p\leq C_{\tau,p}\eta^{-\frac{p}{2}}.
	\end{equation}

	\noindent For the proof and applications of the factorization formula to the study of stochastic evolution equations  the reader is referred to \cite{da2014stochastic}, Section 5.3. 
	It remains to estimate the last term on the left-hand side of \eqref{Holderdec}. We have
	\begin{equation*}
	\begin{aligned}
	\big[I-e^{-\Gamma(t-s)/\eta}\big]Y_s=\sum_{j=1}^{3}\big[I-e^{-\Gamma(t-s)/\eta}\big]A_j(0,s).
	\end{aligned}
	\end{equation*}
	For the first summand we write $I-e^{-\Gamma(t-s)/\eta}= (I-e^{-\Gamma(t-s)/\eta})^\theta(I-e^{-\Gamma(t-s)/\eta})^{1-\theta}$ and use properties of exponentials to obtain
	\begin{equation*}
	\begin{aligned}
	\big|\big[I-e^{-\Gamma(t-s)/\eta}\big]A_1(0,s)\big|&\leq C\|\zeta\|_\infty (t-s)^\theta\eta^{-1-\theta}\int_{0}^{s}e^{-\gamma(s-r)/\eta}dr\\&\leq C' \|\zeta\|_\infty(t-s)^\theta\eta^{-1-\theta}\eta^{\frac{1}{2}+\theta}\int_{0}^{s}(s-r)^{-\frac{1}{2}-\theta}ds\leq C_\zeta \eta^{-\frac{1}{2}} (t-s)^{\theta}.
	\end{aligned}
	\end{equation*}
	Note that the singularity in the last integral above is integrable since $\theta<1/2$. Similarly,
	\begin{equation*}
	\begin{aligned}
	\big|&\big[I-e^{-\Gamma(t-s)/\eta}\big]A_2(0,s)\big|\\&\leq \frac{C}{\sqrt{\epsilon\eta}}(\|\tau\|_\infty+C_g)\frac{(t-s)^\theta}{\eta^{\theta}}\bigg[\int_{0}^{s}e^{-2\gamma(s-r)/\eta}dr\bigg]^{\frac{1}{2}}\cdot\\&\cdot\bigg(\|\dot{u}_2\|_{L^2([0,T];\mathcal{Y})}+\bigg[\int_0^T (1+|Y_r|)^2dr\bigg]^{\frac{1}{2}}\bigg)\\&
	\leq \frac{C'(\|\tau\|_\infty+C_g)}{\sqrt{\epsilon\eta}}(t-s)^\theta\eta^{-\theta}\eta^\theta\bigg[\int_{0}^{s}(s-r)^{-2\theta}dr\bigg]^{\frac{1}{2}}\bigg(N+\bigg[\int_0^T (1+|Y_r|)^2dr\bigg]^{\frac{1}{2}}\bigg)\\&
	\leq \frac{C}{\sqrt{\epsilon\eta}}(t-s)^\theta\bigg(N+\bigg[\int_0^T (1+|Y_r|)^2dr\bigg]^{\frac{1}{2}}\bigg),
	\end{aligned}
	\end{equation*}
	where we used the Cauchy-Schwarz inequality on the first line and the last integral is finite since $\theta<1/2$. Finally, \eqref{factor} and the BDG inequality yield 
	\begin{equation*}
	\begin{aligned}
	\ex\bigg[\sup_{t\neq s} |t-s|^{-\theta}\big|\big[I-e^{-\Gamma(t-s)/\eta}\big]A_3(0,s)\big|\bigg]^p&\leq C_{\tau,T,p}\eta^{-\frac{p}{2}}.
	\end{aligned}
	\end{equation*}
	These estimates along with \eqref{A12} and \eqref{A3} conclude the proof.
\end{proof}

\begin{lem}\label{Yfraclem}
	For $p\geq 1$ and $T>0$ and $|\Delta_a|$ as in \eqref{Deltapmabs} the following hold:\\
	(i) Let $0<a<\frac{1}{2}$ and $\theta\in(a,1/2)$. There exists $C>0$ and $\epsilon_0>0$ such that for all $\epsilon<\epsilon_0$ we have
	\begin{equation}\label{Hbig}
	\sup_{u\in\mathcal{A}_N}\ex\bigg(\sup_{t_1\neq t_2}\frac{1}{(t_2-t_1)^{\frac{1}{2}-a}}\int_{t_1}^{t_2}\big|\Delta_a\big|Y^{\epsilon,\eta,u}_{t_1,s} ds\bigg)^p\leq C\eta^{-\frac{pa}{\theta}}.
	\end{equation}
	(ii) Let $H\in(3/4,1)$ and $a\in(1-H, 1/4)$. Furthermore, assume that there exists $\beta\in(2(1-H), 1/2)$ such that $\sqrt{\epsilon}/\eta^{\;\beta}\rightarrow 0$ as $\epsilon\to 0$. Then
	\begin{equation*}
	\ex\bigg(\sqrt\epsilon\sup_{t_2\in[0,T]}\int_{0}^{t_2}(t_2-t_1)^{-a-1}\int_{t_1}^{t_2} \big|\Delta_a\big|Y^{\epsilon,\eta,u}_{t_1,s}dsdt_1\bigg)^p\leq C\bigg(\frac{\sqrt{\epsilon}}{\eta^{\;\beta}}\bigg)^p\longrightarrow 0\;,\;\epsilon\to0.
	\end{equation*}
\end{lem}
\begin{proof}
	\noindent (i) For each $\epsilon$, let $\rho=\rho(\epsilon,t_1,t_2)<t_2-t_1$ such that $\rho\rightarrow 0,$ as $\epsilon\to 0,$ uniformly over $t_1,t_2\in[0,T]$. The exact dependence of $\rho$ on $\epsilon$ will be specified later. Decomposing the domain of integration we have
	\begin{equation*}
	\begin{aligned}
	&\int_{t_1}^{t_2}\big|\Delta_a\big|Y^{\epsilon,\eta,u}_{t_1,s} ds\\&=\int_{t_1+\rho}^{t_2}\int_{t_1}^{s-\rho}\frac{|Y^{\epsilon,\eta,u}_{s}-Y^{\epsilon,\eta,u}_z|}{(s-z)^{a+1}}dzds+\int_{t_1}^{t_2-\rho}\int_{z}^{z+\rho}\frac{|Y^{\epsilon,\eta,u}_{s}-Y^{\epsilon,\eta,u}_z|}{(s-z)^{a+1}}dsdz\\&+\int_{t_2-\rho}^{t_2}\int_{z}^{t_2}\frac{|Y^{\epsilon,\eta,u}_{s}-Y^{\epsilon,\eta,u}_z|}{(s-z)^{a+1}}dsdz
	=: I+II+III.
	\end{aligned}
	\end{equation*}
	Since the first term is integrated away from the diagonal, the singular kernel is integrable. A combination of Fubini's theorem and the Cauchy-Schwarz inequality thus yields
	\begin{equation*}\label{Iest}
	\begin{aligned}
	I&\leq \int_{t_1+\rho}^{t_2}|Y^{\epsilon,\eta,u}_{s}|\int_{t_1}^{s-\rho}(s-z)^{-a-1}dzds+    \int_{t_1}^{t_2-\rho}|Y^{\epsilon,\eta,u}_z|\int_{z+\rho}^{t_2}(s-z)^{-a-1}dsdz\\&
	\leq \frac{\rho^{-a}}{a}\bigg[\int_{t_1+\rho}^{t_2}|Y^{\epsilon,\eta,u}_{s}|ds
	+ \int_{t_1}^{t_2-\rho}|Y^{\epsilon,\eta,u}_z|dz  \bigg] 
	\leq \frac{\rho^{-a}}{a}(t_2-t_1)^{\frac{1}{2}}\|Y^{\epsilon,\eta,u}\|_{L^2([0,T];\mathcal{Y})}.
	\end{aligned}
	\end{equation*}
	For the second and third terms we take advantage of the H\"older regularity of $Y^{\epsilon,\eta,u}$. In particular, let $\theta\in(a,1/2)$ and fix a $\theta$-H\"older continuous version. It follows that
	\begin{equation*}
	\begin{aligned}
	II&\leq [Y^{\epsilon,\eta,u}]_{C^{\theta}([0,T];\mathcal{Y})}\int_{t_1}^{t_2-\rho}\int_{z}^{z+\rho}(s-z)^{\theta-a-1}dsdz\leq\frac{\rho^{\theta-a}}{\theta-a}[Y^{\epsilon,\eta,u}]_{C^{\theta}([0,T];\mathcal{Y})}(t_2-t_1)
	\end{aligned}
	\end{equation*}
	with probability $1$. Similarly,
	\begin{equation*}\label{IIIest}
	\begin{aligned}
	III&\leq[Y^{\epsilon,\eta,u}]_{C^{\theta}([0,T];\mathcal{Y})}\int_{t_2-\rho}^{t_2}\int_{z}^{t_2}(s-z)^{\theta-a-1}dsdz\\&\leq \frac{1}{\theta-a}[Y^{\epsilon,\eta,u}]_{C^{\theta}([0,T];\mathcal{Y})}\int_{t_2-\rho}^{t_2}(t_2-z)^{\theta-a}dz\leq C_{a,\theta}[Y^{\epsilon,\eta,u}]_{C^{\theta}([0,T];\mathcal{Y})}\rho^{\theta-a+1}.
	\end{aligned}
	\end{equation*}
	Now choose $\rho(\epsilon):=\eta^\frac{1}{\theta}(\epsilon)(t_2-t_1)$, $\epsilon$ small enough to satisfy $\rho(\epsilon)<t_2-t_1$ and combine the previous bounds to obtain:
	\begin{equation*}
	\begin{aligned}
	&\int_{t_1}^{t_2}\big|\Delta_a\big|Y^{\epsilon,\eta,u}_{t_1,s}ds  \leq C\bigg[ \eta^{-\frac{a}{\theta}} \|Y^{\epsilon,\eta,u}\|_{L^{2}}(t_2-t_1)^{\frac{1}{2}-a}
	+\eta^{1-\frac{a}{\theta}}[Y^{\epsilon,\eta,u}]_{C^{\theta}([0,T];\mathcal{Y})}(t_2-t_1)^{1+\theta-a}\\&+\eta^{1-\frac{a}{\theta}+\frac{1}{\theta}}[Y^{\epsilon,\eta,u}]_{C^{\theta}([0,T];\mathcal{Y})}(t_2-t_1)^{1+\theta-a}\bigg].
	\end{aligned}
	\end{equation*}
	Lemma \ref{Ybndlem} then furnishes
	\begin{equation*}
	\eta^{\frac{pa}{\theta}}\cdot\ex\bigg(\sup_{t_1\neq t_2}\frac{1}{(t_2-t_1)^{\frac{1}{2}-a}}\int_{t_1}^{t_2}\big|\Delta_a\big|Y^{\epsilon,\eta,u}_{t_1,s}ds\bigg)^p  \leq C_{T,a,\theta}\bigg[1+\frac{\sqrt{\eta}}{\sqrt{\epsilon}}(1+\eta^{\frac{1}{\theta}})\bigg]^p
	\end{equation*}
	which concludes the argument since $\sqrt{\eta}/\sqrt{\epsilon}$ is bounded for $\epsilon$ sufficiently small (recall \eqref{regime}).\\
	(ii)	From \eqref{Hbig} we have 
	\begin{equation*}
	\begin{aligned}
	\ex\bigg(&\sqrt\epsilon\sup_{t_2\in[0,T]}\int_{0}^{t_2}(t_2-t_1)^{-a-1}\int_{t_1}^{t_2} \big|\Delta_a\big|Y^{\epsilon,\eta,u}_{t_1,s}dsdt_1\bigg)^p\\&\leq  \epsilon^{p/2}\bigg(\int_{0}^{t_2}(t_2-t_1)^{-2a-\frac{1}{2}}dt_1\bigg)^p\sup_{\epsilon<\epsilon_0, u\in\mathcal{A}_N}\ex\bigg(\sup_{t_1\neq t_2}\frac{1}{(t_2-t_1)^{\frac{1}{2}-a}}\int_{t_1}^{t_2}\big|\Delta_a\big|Y^{\epsilon,\eta,u}_{t_1,s} ds\bigg)^p\\&
	\leq CT^{p(\frac{1}{2}-2a)}\bigg(\frac{\sqrt{\epsilon}}{\eta^{\frac{a}{\theta}}}\bigg)^p.
	\end{aligned}
	\end{equation*}
	where we used that $a<1/4.$ The proof is complete upon choosing $\theta=a/\beta\in(2a, 1/2).$
\end{proof}

\noindent Our next auxiliary estimate concerns the solution of the Poisson equation \eqref{Poisson}. Recently, the authors of \cite{ganguly2021inhomogeneous} gave a slightly modified version of estimate (21) in Theorem 2 of \cite{Pardoux} regarding the growth of $\nabla\Psi$ (see Proposition A.2 in \cite{ganguly2021inhomogeneous}). In there, $\nabla\Psi$ is shown to have quadratic growth, even if the coefficient $b$ is bounded. However, as stated by the authors, this growth rate is not optimal and in fact can be improved under more restrictive assumptions on $f,\tau$. The next lemma identifies sufficient conditions under which $\nabla\Psi$ is bounded.

\begin{lem}\label{psider} Let $\Psi$ solve \eqref{Poisson}, $C_f$ as in Condition \ref{C3}, $C_2>0$ be the optimal constant of the BDG inequality for $p=2$ and set
	$K_f:=4\bigg(\frac{\|\nabla f\|^2_{\infty}}{C_f^2}\vee 1\bigg).$
	If $\|D\tau\|^2_{\infty}<\frac{C_f}{C_2K_f}$ then	$\sup_{y\in\mathcal{Y}}|\nabla\Psi(y)|< \infty.$
\end{lem}

\begin{proof} For all $y\in\mathcal{Y}$ we have the probabilistic representation  
	\begin{equation*}
	\Psi(y)=\int_{0}^{\infty}\ex[ b(Y^y_t )]dt,
	\end{equation*}
	where $Y$ is the ergodic diffusion \eqref{ergodicdiff} with initial condition $y$.
	Next let $h\in\mathcal{Y}.$ Differentiating under the sign of expectation we obtain 
	\begin{equation}\label{poissonder}
	\nabla\Psi(y)h=\int_{0}^{\infty}\ex[ \nabla b(Y^y_t )Z^y_t]dt,
	\end{equation}
	where the derivative $Z^y=D_yY^y$ in the direction of $h$ solves the \textit{first variation equation}
	\begin{equation}\label{vareq}
	dZ_t= \nabla f(Y_t^y)Z_tdt+D\tau(Y^y_t)Z_t dW_t\;, Z_0=h.
	\end{equation}
	Now let $\tilde{Z}_\cdot:=Z_\cdot-\int_{0}^{\cdot}D\tau(Y^y_t)Z_t dW_t$ and $M=Z-\tilde{Z}$. The process $\tilde{Z}$ has almost surely differentiable paths and satisfies 
	\begin{equation*}
	\begin{aligned}
	\frac{1}{2}\frac{d}{dt}|\tilde{Z}_t|^2&=\blangle \frac{d}{dt}\tilde{Z}_t, \tilde{Z}_t\brangle\\&=	\langle \nabla f(Y^y_t)[\tilde{Z_t}+M_t], \tilde{Z}_t \rangle\\&\leq \langle \nabla f(Y^y_t)\tilde{Z_t}, \tilde{Z}_t \rangle+\|\nabla f\|_{\infty}|M_t||\tilde{Z}_t|\\&\leq -C_f|\tilde{Z}_t|^2+\frac{C_f}{2}|\tilde{Z}_t|^2+\frac{\|\nabla f\|^2_{\infty}}{2C_f}|M_t|^2,
	\end{aligned}
	\end{equation*}
	where we used Condition \ref{C3} and Young's product inequality to obtain the last line. Integrating yields
	\begin{equation*}
	\begin{aligned}
	&|\tilde{Z}_t|^2\leq e^{-C_ft}|h|^2+\frac{\|\nabla f\|^2_{\infty}}{C_f}\int_{0}^{t}e^{-C_f(t-s)}|M_s|^2ds
	\end{aligned}
	\end{equation*}
	and 
	\begin{equation*}
	\begin{aligned}
	|Z_t|^2\leq 2|\tilde{Z}_t|^2+2|M_t|^2&\leq 2e^{-C_ft}|h|^2+\frac{2\|\nabla f\|^2_{\infty}}{C_f}\int_{0}^{t}e^{-C_f(t-s)}|M_s|^2ds+ 2|M_t|^2\\&
	\leq  2e^{-C_ft}|h|^2+\frac{2\|\nabla f\|^2_{\infty}}{C_f}\sup_{s\in[0,t]}|M_s|^2\int_{0}^{t}e^{-C_f(t-s)}ds+2|M_t|^2\\&
	\leq 2e^{-C_ft}|h|^2+K_f\sup_{s\in[0,t]}|M_s|^2.
	\end{aligned}
	\end{equation*}
	From the BDG inequality, there exists a constant $C_2>0$ such that
	\begin{equation*}
	\begin{aligned}
	\ex|Z_t|^{2}&\leq 2e^{-C_ft}|h|^{2}+C_2K_f\int_{0}^{t}\ex |D\tau(Y^y_r)Z_r|^2dr      \\&
	\leq  2e^{-C_ft}|h|^2+C_2K_f\|D\tau\|_\infty^{2}\int_{0}^{t}\ex |Z_r|^2dr.
	\end{aligned}
	\end{equation*}
	Finally, Gr\"onwall's inequality furnishes
	\begin{equation}\label{Zbnd}
	\begin{aligned}
	\ex|Z_t|^2\leq 2e^{-(C_{f}-C_2K_f\|D\tau\|_\infty^{2})t}|h|^2
	\end{aligned}
	\end{equation}
	and by assumption $C_f-C_2K_fL^2_\tau>0$. Since $h$ is arbitrary we conclude from Condition \ref{C4} and \ref{poissonder} that
	\begin{equation*}
	|\nabla\Psi(y)|\leq \|\nabla b\|_\infty\int_{0}^{\infty}e^{-(C_f-C_2K_f\|D\tau\|_\infty^{2} )t/2}dt<\infty.
	\end{equation*}	\end{proof}
\begin{lem}\label{psiliplem} Let $C_f,C_2, K_f, \Psi$ as in Lemma \ref{psider}, $C_4$ be the optimal constant for the BDG inequality with $p=4$ and assume that $\nabla b,D\tau$ are Lipschitz continuous. If Condition \ref{C3} holds with $\zeta=0$ and 	
	$ C_4K^2_f\|D\tau\|_\infty^{4}+C_2\|D\tau\|^2_{\infty}<C_f$
	
	then $\nabla\Psi$ is Lipschitz continuous.
\end{lem}
\begin{proof} Let $y_1,y_2\in\mathcal{Y}$ and $L_b, L_\tau$ denote the Lipschitz constants of $\nabla b,D\tau$ respectively. From \eqref{poissonder} we have
	\begin{equation}\label{Psiliprebnd}
	\begin{aligned}
	\big|\nabla\Psi(y_2)-\nabla\Psi(y_1)\big|&\leq \int_{0}^{\infty}\ex\big|\nabla b(Y_t^{y_2})Z^{y_2}_t -\nabla b(Y_t^{y_1})Z^{y_1}_t \big|dt\\&
	\leq \int_{0}^{\infty}\ex\big|\nabla b(Y_t^{y_2})\big|\big|Z^{y_2}_t -Z^{y_1}_t \big|dt+ \int_{0}^{\infty}\ex\big|\nabla b(Y_t^{y_2})-\nabla b(Y_t^{y_1})\big|\big|Z^{y_1}_t \big|dt\\&
	\leq \|\nabla b\|_\infty\int_{0}^{\infty}\ex\big|Z^{y_2}_t -Z^{y_1}_t \big|dt+L_{b}|y_1-y_2|\int_{0}^{\infty}\sup_{y\in\mathcal{Y}}\big\|Z^y_t\big\|^2_{L^2(\Omega)}dt\\&
	\leq  \|\nabla b\|_\infty\int_{0}^{\infty}\ex\big|Z^{y_2}_t -Z^{y_1}_t \big|dt+L_{b}|y_1-y_2|\int_{0}^{\infty}e^{(-C_f+C_2K_f\|D\tau\|^2_\infty)t}dt,
	\end{aligned}
	\end{equation}
	where we used the mean value inequality  and \eqref{Zbnd} to obtain the third and fourth lines above respectively. Since $\zeta=0,$ \eqref{vareq} has a constant linear drift and we can write
	\begin{equation*}
	\begin{aligned}
	Z^{y_2}_t -Z^{y_1}_t=\int_{0}^{t}e^{-\Gamma(t-s)}\big(D\tau(Y^{y_2}_t)Z^{y_2}_s-D \tau(Y^{y_1}_t)Z^{y_1}_s\big)dW_s
	\end{aligned}
	\end{equation*}
	with $\Gamma$ as in Condition \ref{C3}. From the BDG and mean-value inequalities it follows that 	
	\begin{equation*}
	\begin{aligned}
	&\ex|Z^{y_2}_t -Z^{y_1}_t \big|^2\leq C_2\int_{0}^{t}e^{-2C_f(t-s)} \ex\big|D \tau(Y^{y_2}_t)Z^{y_2}_s-D \tau(Y^{y_1}_t)Z^{y_1}_s\big|^2ds\\&
	\leq C_2L^2_\tau|y_2-y_1|^2\int_{0}^{t}e^{-2C_f(t-s)}\ex|Z_s^{y_2}|^4ds+C_2\|D \tau\|^2_{\infty}\int_{0}^{t}e^{-2C_f(t-s)} \ex\big|Z^{y_2}_s-Z^{y_1}_s\big|^2ds\\&
	\leq  C_2L^2_\tau|y_2-y_1|^2\int_{0}^{t}e^{-2C_f(t-s)}e^{-(2C_f-2C_4K^2_f\|D\tau\|_\infty^{4} )s}|h|^2ds\\&+C_2\|D\tau\|^2_{\infty}\int_{0}^{t}e^{-2C_f(t-s)} \ex\big|Z^{y_2}_s-Z^{y_1}_s\big|^2ds.
	\end{aligned}
	\end{equation*}
	\noindent where we adapted the estimate \eqref{Zbnd} to obtain the fourth moment bound for $Z^{y}$ that was used in the last line. Hence,
	\begin{equation*}
	\begin{aligned}
	e^{2C_ft}\ex|Z^{y_2}_t -Z^{y_1}_t \big|^2&
	\leq  C|y_2-y_1|^2|h|^2\int_{0}^{t}e^{2C_4K^2_f\|D\tau\|_\infty^{4}s}ds\\&+C_2\|D\tau\|^2_{\infty}\int_{0}^{t}e^{2C_fs} \ex\big|Z^{y_2}_s-Z^{y_1}_s\big|^2ds.
	\end{aligned}
	\end{equation*}
	and Gr\"onwall's inequality yields
	\begin{equation*}
	\begin{aligned}
	e^{2C_ft}\ex|Z^{y_2}_t -Z^{y_1}_t \big|^2&
	\leq  C|y_2-y_1|^2|h|^2e^{(2C_4K^2_f\|D\tau\|_\infty^{4}+C_2\|D\tau\|^2_{\infty} )t}.
	\end{aligned}
	\end{equation*}
	By assumption we have $-2C_f+2C_4K^2_f\|D\tau\|_\infty^{4}+C_2\|D\tau\|^2_{\infty}<0.$   Substituting the latter to \eqref{Psiliprebnd} concludes the argument.\end{proof}

\begin{lem}\label{udotL2}
	Let $H\in(1/2, 1)$ and $\{u^\epsilon;\epsilon>0\}$ be a family of real-valued processes that is uniformly bounded in $\h_H$ \eqref{CMdef}. The family $\{\dot{u}^\epsilon;\epsilon>0\}$ is uniformly bounded in $L^2([0,T];\R)$.
\end{lem}
\begin{proof}
	The statement follows essentially from the continuity of the last inclusion in \eqref{Taqquembed}. In particular, for each $\epsilon>0$, there exists $\tilde{u}^\epsilon$ such that $u^\epsilon=K_H\tilde{u}^\epsilon$ and uniform boundedness in the topology of $\h_H$ is, by definition, equivalent to uniform boundedness of the family $\{\tilde{u}^\epsilon;\epsilon>0\}$ in $L^2$. From the second inclusion in \eqref{Taqquembed} it follows that 
	\begin{equation*}
	\|\tilde{u}^\epsilon\|_{|\mathfrak{H}|}\leq C \|\tilde{u}^\epsilon\|_{L^2}\leq  C \sup_\epsilon\|\tilde{u}^\epsilon\|_{L^2}<\infty,
	\end{equation*}
	with probability $1$ and for some constant $C>0$ independent of $\epsilon$. As shown in display (4.5) of \cite{pipiras2001classes},  we have 
	\begin{equation*}
	\begin{aligned}
	\|\dot{u}^\epsilon\|^2_{L^2}=\|\dot{K}_H\tilde{u}^\epsilon\|^2_{L^2}&= \|u^\epsilon\|^2_{\mathfrak{H}}\\&=\int_{0}^{T}t^{2H-1}I^{H-\frac{1}{2}}_{0^+}\big(\big[s^{\frac{1}{2}-H}\tilde{u}^\epsilon\big](t)\big)^2dt\\&=\frac{B(H-\frac{1}{2},2-2H)}{\Gamma(H-\frac{1}{2})^2}\int_{0}^{T}\int_{0}^{T}\tilde{u}^\epsilon(s)\tilde{u}^\epsilon(t)|s-t|^{2H-2}dsdt\leq C\|\tilde{u}^\epsilon\|_{|\mathfrak{H}|},
	\end{aligned}
	\end{equation*}
	where $B$ denotes the beta function. In view of the last two displays, the proof is complete.	\end{proof}

\section*{Acknowledgements}

The authors wish to thank the two anonymous referees whose comments and suggestions significantly improved the presentation of the results and overall quality of the present work. In particular, the authors express their gratitude to an anonymous referee for encouraging them to expand Remark \ref{Hlimrem} and parts of Section \ref{concsec}, bringing references \cite{breuer1983central, neufcourt2016third,  nourdin2012normal} to their attention and for emphasizing the connection between memory properties of the fBm and the discontinuity of the LDP rate function in Proposition \ref{Hlimprop} and Remark \ref{Hlimrem}.

IG acknowledges financial support from the EPSRC grant EP/T032146/1.

\end{document}